\documentclass{article}

\title{Serre functors and graded categories} 
\author{Joseph Grant}
\date{\vspace{-1em}}

\usepackage[UKenglish]{isodate}% http://ctan.org/pkg/isodate
\cleanlookdateon% Remove ordinal day reference
\usepackage{enumitem}% Remove whitespace before itemize lists
\setlist[itemize]{%noitemsep, 
 topsep=0pt}
\setlist[enumerate]{%noitemsep, 
 topsep=0pt}

\usepackage[nottoc,notlot,notlof]{tocbibind}
\usepackage{latexsym}
\usepackage{verbatim}
\usepackage{amsmath}
\usepackage{amssymb}
\usepackage{mathrsfs}
\usepackage{amsthm}
\usepackage{sfmath}
\usepackage{bbm} % gives blackboard bold 1

\usepackage{stmaryrd} % gives \mapsfrom
\usepackage{graphicx} % gives rotatebox

\usepackage{hyperref}  % makes contents page clickable

\usepackage[british]{babel}% http://ctan.org/pkg/babel

\usepackage{tikz}
\usetikzlibrary{arrows,automata,shapes}
\usetikzlibrary{decorations.markings}

\input xy
\usepackage[all,cmtip,2cell]{xy}
\UseTwocells

\newcommand{\Sh}{\operatorname{Sh}}
\newcommand{\Oring}{\mathcal{O}}
\newcommand{\bimod}{\operatorname{-mod-}}
\newcommand{\da}{\text{-}}
\newcommand{\mMod}{\operatorname{-mod}}
\newcommand{\B}{\mathcal{B}}
\newcommand{\C}{\mathcal{C}}
\newcommand{\D}{\mathcal{D}}

\newcommand{\biB}{{\boldsymbol{\mathcal{B}}}}
\newcommand{\biC}{{\boldsymbol{\mathcal{C}}}}
\newcommand{\biF}{\twofun{F}}
\newcommand{\biG}{\twofun{G}}
\newcommand{\biP}{\twofun{P}}
\newcommand{\M}{\mathcal{M}}
\newcommand{\N}{\mathcal{N}}
\newcommand{\PP}{\mathcal{P}}
\newcommand{\U}{\mathcal{U}}
\newcommand{\V}{\mathcal{V}}
\newcommand{\End}{\operatorname{End}\nolimits}
\newcommand{\Pic}{\operatorname{Pic}\nolimits}
\newcommand{\DPic}{\operatorname{DPic}\nolimits}
\newcommand{\PicEnd}{\operatorname{PicEnd}\nolimits}
\newcommand{\Db}{\operatorname{D^b}\nolimits}
\newcommand{\into}{\hookrightarrow}

\newcommand{\unit}{{\mathbbmss{1}}} 
\newcommand{\arr}[1]{\stackrel{#1}{\to}}
\newcommand{\arrr}[1]{\stackrel{#1}{\longrightarrow}}

\newcommand{\id}{1} 
\newcommand{\ob}{\operatorname{ob}}
\newcommand{\op}{{\operatorname{op}\nolimits}}
\newcommand{\rev}{{\operatorname{rev}\nolimits}}
\newcommand{\fdv}{{\operatorname{fVec}\nolimits}} 
\newcommand{\twofun}[1]{\operatorname{\textbf{#1}}}
\newcommand{\kk}{\mathbbmss{k}}
\newcommand{\Z}{\mathbbmss{Z}}

\newcommand{\Autom}{\operatorname{Autom}\nolimits}
\newcommand{\Hom}{\operatorname{Hom}\nolimits}
\newcommand{\Autoeq}{\operatorname{Autoeq}\nolimits}
\newcommand{\Cat}{\operatorname{\twofun{Cat}}}
\newcommand{\aCat}{\operatorname{-\kk\Cat}}
\newcommand{\GrCat}[1]{\operatorname{\kk\Cat}^{#1}}
\newcommand{\aSCat}{\operatorname{-\se^\chi\Cat}}
\newcommand{\GrSCat}[1]{\operatorname{\se^\chi\Cat}^{#1}}
\newcommand{\then}{\Rightarrow}
\newcommand{\Alg}{\operatorname{\textbf{Alg}}}

\newcommand{\Aut}{\operatorname{Aut}\nolimits}

\newcommand{\ic}{\twofun{ic}}
\newcommand{\Ind}{\twofun{Ind}}
\newcommand{\ind}{\Ind}
\newcommand{\skel}{\operatorname{skel}}
\newcommand{\core}{\operatorname{core}\nolimits}
\newcommand{\tgpd}{\twofun{2Gpd}}
\newcommand{\tcat}{\twofun{2Cat}}
\newcommand{\Mat}{\twofun{Mat}}

\newcommand{\deloop}{\twofun{B}}
\newcommand{\looping}{\Omega}

\newcommand{\catei}{\twofun{Cat}_\ei}
\newcommand{\cati}{\twofun{Cat}_{\mathbf{2,1}}}

\newcommand{\ei}{\mathbf{2,0}}
\newcommand{\BaseCat}{\twofun{BaseCat}}
\newcommand{\AddCat}{\twofun{AddCat}}
\newcommand{\fBaseCat}{\twofun{fBaseCat}}
\newcommand{\fAddCat}{\twofun{fAddCat}}
\newcommand{\fAddSCat}{\twofun{fAdd}\se\Cat}
\newcommand{\fBasicAlg}{\twofun{fBasicAlg}}
\newcommand{\fBasicFAlg}{\twofun{fBasicFAlg}}
\newcommand{\Tri}{{\twofun{Tri}}}
\newcommand{\fBim}{{\twofun{ffBim}}}
\newcommand{\ev}{\operatorname{ev}}

\newcommand{\smsh}{\#}
\newcommand{\se}{{\mathbbmss{S}}}
\newcommand{\te}{{\mathbbmss{T}}}
\newcommand{\sse}{s}
\newcommand{\trivG }{{\Delta}}
\newcommand{\gen}[1]{\langle#1\rangle}
\newcommand{\st}{\;\left|\right.\;}
\newcommand{\grsh}[1]{\{{#1}\}}

\newcommand{\centre}{\mathcal{Z}}
\newcommand{\sgn}{\operatorname{sgn}}
\newcommand{\tr}{\operatorname{tr}}
\newcommand{\res}[1]{|_{#1}}
\newcommand{\dert}{\otimes^{\operatorname{\textbf{L}}}}
\newcommand{\prim}{\operatorname{prim}}
\newcommand{\xyc}[1]{\hbox to 3em{\hss$\displaystyle{#1}$\hss}}
\newcommand{\dashvvert}{\rotatebox[origin=c]{-90}{\ensuremath\dashv}}
\newcommand{\abs}[1]{\left|#1\right|}
\newcommand{\isom}{\simeq}
\newcommand{\Mon}{\operatorname{\mathcal{M}on}}
\newcommand{\Grp}{\operatorname{\mathcal{G}rp}}
\newcommand{\SCat}{\se\V\Cat}
\newcommand{\FAlg}{\operatorname{\textbf{FAlg}}}

\newtheorem{theorem}{Theorem}[section]
\newtheorem{corollary}[theorem]{Corollary}
\newtheorem{lemma}[theorem]{Lemma}
\newtheorem{proposition}[theorem]{Proposition}
\newtheorem{question}[theorem]{Question}

\newtheorem*{thma}{Theorem A}
\newtheorem*{thmb}{Theorem B}

\theoremstyle{definition}
\newtheorem{definition}[theorem]{Definition}
\newtheorem{remark}[theorem]{Remark}
\newtheorem{example}[theorem]{Example}

\begin{document}

\maketitle

\begin{abstract}
We study Serre structures on categories enriched in pivotal monoidal categories, and apply this to study Serre structures on two types of graded $\kk$-linear categories: categories with group actions and categories with graded hom spaces.  We check that Serre structures are preserved by taking orbit categories and skew group categories, and describe the relationship with graded Frobenius algebras.  Using a formal version of Auslander-Reiten translations, we show that the derived category of a $d$-representation finite algebra is fractionally Calabi-Yau if and only if its preprojective algebra has a graded Nakayama automorphism of finite order.  This connects various results in the literature and gives new examples of fractional Calabi-Yau algebras.

\end{abstract}

\tableofcontents

\setlength{\parindent}{0pt} 
\setlength{\parskip}{1em plus 0.5ex minus 0.2ex}

%=-=-=-=-==-=-=-=-=-=
%=-=-=-=-==-=-=-=-=-=
%=-=-=-=-==-=-=-=-=-=
%=-=-=-=-==-=-=-=-=-=
%=-=-=-=-==-=-=-=-=-=

% -------------------------- Intro --------------------------
%\newpage
\section{Introduction}

This paper relates two algebraic structures: Frobenius algebras and categories with Serre duality.  It’s well-known that they are related, but we set out the relationship quite carefully in order to prove some statements in the representation theory of quivers.

We want to show that the Nakayama automorphism of a preprojective algebra has finite order if and only if the derived category of the corresponding quiver is fractional Calabi-Yau.  The Nakayama automorphism and the Serre functor agree in certain situations so, up to some technical considerations about gradings, this is about them both having finite order.

To give a precise statement we use graded categories, which live in a 2-category, and each one is enriched in a fixed monoidal category.  We give equivalences between such structures, so we are really studying higher category theory, but we only use 2-categories in this paper.

\subsection*{Frobenius algebras and Calabi-Yau categories}

Given a tensor category, such as vector spaces, one can look for algebra objects and coalgebra objects.  A Frobenius algebra is an object with both algebra and coalgebra structures simultaneously, satisfying certain axioms of a topological nature \cite{abrams}.  This reflects the fact that commutative Frobenius algebras classify 2-dimensional topological quantum field theories: see \cite{kock} for a nice explanation.  If our tensor category has duals then we can express the Frobenius algebra axioms in a form that would look more recognisable to a classically trained algebraist, using nondegenerate forms or module isomorphisms.  Reformulating the definitions in this way has the advantage of exhibiting an automorphism of our object known classically as the Nakayama automorphism.  If this automorphism is the identity, one calls the algebra a symmetric algebra.

We might want a ``many-object'' version of a Frobenius algebra.  This is a category $\C$ with Serre duality.  It comes with a Serre functor $\se:\C\arr\sim\C$, an autoequivalence which replaces the Nakayama automorphism.  In situations of interest our categories often have extra structure (e.g., a triangulated structure) and come with a suspension functor $\Sigma:\C\arr\sim\C$.  Then automorphisms of $\C$ should really be equivariant, i.e., they should come equipped with natural isomorphisms witnessing a weak commutation with $\Sigma$.

If the Serre functor is naturally isomorphic to an $n$th power of $\Sigma$, $\C$ is said to be a \emph{Calabi-Yau} category of dimension $n$ \cite{kon-ens}.  Calabi-Yau categories, and more generally Calabi-Yau $A_\infty$-categories, are also connected to TQFTs \cite{costello} as well as to homological mirror symmetry.  There is a subtlety in their definition: we should really have an equivariant isomorphism of equivariant functors between $\se$ and $\Sigma^n$.  Calabi-Yau categories are abundant: as well as the examples coming from geometry which motivate the terminology, there is a formal construction which, from a dg-algebra, produces a new dg-algebra whose derived category is Calabi-Yau \cite{kel-dcy}.  

One could weaken the definition of a Calabi-Yau category and look for categories with Serre duality where we have a natural isomorphism between some power $\se^m$ of the Serre functor and $\Sigma^N$: these are known as fractional Calabi-Yau categories of dimension $N/m$ \cite{kon-ens}.  At first the definition may seem surprising, but there are interesting examples of such categories in many parts of mathematics.  Perhaps the simplest is quiver representations \cite{my}, but they also appear in algebraic geometry, matrix factorisations, theoretical physics, and other parts of representation theory \cite{taka1,aa,klm,himo,kuz}.

The aim of this paper is to establish methods to prove that certain categories are fractional Calabi-Yau.  We do this in two steps: first we change the equivariant structure to one inspired by classical Auslander-Reiten theory, thereby moving to a situation which is often better understood in the representation theory of algebras.  Then we move back to a one-object setting, so we can test the fractional Calabi-Yau property via the Nakayama automorphism of a well-known algebra called the preprojective algebra.

When investigating Calabi-Yau properties, we are asking about isomorphims of functors.  This is a 2-categorical notion, and we believe it is easier if we embrace this from the start.  Although this seems reasonable to us, it is less common within the representation theory of finite-dimensional algebras, so we take the opportunity to formulate 2-categorical versions of some well-known results.

\subsection*{Quivers and preprojective algebras}

Let $Q$ be a quiver and let $\kk Q$ be its path algebra.  Assume $Q$ has no oriented cycles, then $\kk Q$ is finite-dimensional.  According to Gabriel's theorem, the representation theory of the algebra $\kk Q$ behaves very differently depending on the underlying graph of $Q$: if it is an ADE Dynkin graph then $\kk Q$ has finitely many indecomposable modules; otherwise it has infinitely many.

There is another algebra we can associate to $Q$: its preprojective algebra $\Pi(Q)$.  An explicit presentation of this is given by doubling the arrows of $Q$ and quotienting out by certain relations; a basis-free description of $\Pi(Q)$ was given by Baer, Geigle, and Lenzing \cite{bgl}. 
Whether or not the underlying graph is Dynkin is reflected in the algebra $\Pi(Q)$ itself: in the Dynkin case it is finite-dimensional; otherwise it is infinite-dimensional.

The preprojective algebra arose out of an attempt to better understand Gabriel's theorem.  The proof of Gabriel's theorem by Bernstein, Gelfand, and Ponomarev introduced and made use of Coxeter functors on the category of $\kk Q$-modules, then Gelfand and Ponomarev introduced model algebras to understand the image of the projective $\kk Q$-modules under these functors.  In the modern approach, Coxeter functors are replaced by Auslander-Reiten functors $\tau^-$ and model algebras by preprojective algebras.  

In the Dynkin case, $\Pi(Q)$ has another nice property: it is self-injective.  This was a folklore result for a long time.  A careful proof, depending on complicated case-by-case checks, was written down by Brenner, Butler, and King \cite{bbk}.  In fact, they showed more than just self-inectivity: they proved that $\Pi(Q)$ is a Frobenius algebra and gave an explicit formula for its Nakayama automorphism.  In particular, it squares to the identity.  Using the derived category interpretation of the preprojective algebra, the Nakayama automorphism should correspond to the Serre functor $\se$: this has previously been explained on the level of the module category \cite{grant-nak}.  
So we expect $\se^2$ to act trivially on the orbit category; on $\Db(\kk Q)$, it should be a power of $\tau^-$.  This can be deduced on the Grothendieck group $K_0(\kk Q)$ from work of Gabriel.
The derived functor of $\tau^-$ is realised as the shifted inverse Serre functor $\se^-\Sigma$.  So a natural isomorphism between $\se^2$ and $\tau^{-N}$ should correspond to a natural isomorphism between $\se^{N+2}$ and $\Sigma^N$.  This is the fractional Calabi-Yau property of $\Db(\kk Q)$, which was introduced by Kontsevich and proved in general by Miyachi and Yekutieli \cite{kon-ens,my}.

Cluster algebras, introduced by Fomin and Zelevinsky, have been hugely influential in modern representation theory.  The construction of the cluster category by Buan, Marsh, Reineke, Reiten, and Todorov \cite{bmrrt} 
involves starting with the derived category of $\kk Q$-modules and constructing a new category whose objects are the orbits of an endofunctor.

Iyama and Oppermann realised that the abstract construction of the preprojective algebra given by Baer, Geigle, and Lenzing could be interpreted using orbit categories: $\Pi(Q)$ is an endomophism algebra of a generator of $\Db(\kk Q)/\tau^-$.  With this perspective, they were able to give a very general explanation for the fact that $\Pi(Q)$ is self-injective: it follows from the existence of a Serre functor on the derived category of $\kk Q$ \cite{io-stab}.

The work of Iyama and Oppermann is more general than described above: they work with $d$-representation finite algebras $\Lambda$, for which there exists a good higher dimensional analogue of the Auslander-Reiten functor \cite{iya-ar}.  
Around the same time, Herschend and Iyama noticed that examples of $d$-representation finite algebras had the fractional Calabi-Yau property, and were able to show that this property always holds up to some twists \cite{hi-frac}.

At this point, the picture seems clear: 
\begin{center}
  \begin{tabular}{ c c c }
  $d$-representation finite algebra & $\longleftrightarrow$ & $(d+1)$-preprojective algebra \\ 
  Serre duality on the derived category & $\longleftrightarrow$ & self-injectivity of preprojective algebra \\ 
  fractional Calabi-Yau property & $\longleftrightarrow$ & finite order Nakayama automorphism
  \end{tabular}
\end{center}
However, there are two problems:
\begin{itemize}
\item  Even in the classical $d=1$ case, for Dynkin quivers, things don't seem to match up exactly.  For a quiver of type $D_4$, the derived category is  Calabi-Yau of dimension $2/3$ \cite{my}, which suggests the Nakayama automorphism of the corresponding preprojective algebra should have order $3-2$, i.e., it should be the identity.  But we know that this automorphism is non-trivial, even in the group of outer automorphisms \cite{bbk}.
\item  The results of Herschend and Iyama aren't as strong as one might hope.  They only obtain twisted fractional Calabi-Yau properties, and are only able to make a precise connection to the higher preprojective algebra under a ``homogeneous'' assumption which is quite restrictive.
\end{itemize}
In this article we address these problems as follows:
\begin{itemize}
\item We carefully analyse the relationship between Serre functors and Nakayama automorphisms in a graded setting and expose a mismatch of signs in the naive correspondence between them.  When taking account of these signs, the results of Miyachi-Yekutieli and Brenner-Butler-King do match.  Based on this, we propose new graded Nakayama automorphism of the preprojective algebra, which is not related to the classical Nakayama automorphism by an inner automorphism.  From the perspective of derived categories, we see this graded automorphism as more fundamental than the classical Nakayama automorphism.
\item We prove a general result which states that the derived category of a $d$-representation finite algebra is  Calabi-Yau if and only if the appropriate graded Nakayama automorphism of the preprojective algebra is of finite order.  
\end{itemize}
This also allows us to connect two other results in the literature: the higher preprojective algebras of type $A$ have finite order Nakayama automorphism \cite{hi-frac} and the iterated higher Auslander algebras of type $A$ are fractional Calabi-Yau \cite{djw,djl}.

\subsection*{Description of main results}

Given suitable finiteness and linearity assumptions, there is a correspondence between algebras and categories.  We discuss this in Section \ref{s:cats}.
\[ \text{ algebras (A) } \longleftrightarrow \text{ categories (C) } \]
Given a group $G$, there are two notions of $G$-structure on a category $\C$: either the hom-spaces can be $G$-graded, or we can have an action of $G$ by functors $\C\to\C$ and work equivariantly.
  We discuss this in Section \ref{s:graded}.
\[ \text{ graded (G) } \longleftrightarrow \text{ equivariant (E) } \]
In all these situations, there is a notion of Serre duality where we have a Serre functor $\se:\C\to\C$ with quasi-inverse $\se^-$ (see Section \ref{s:serre}.)
If our category is triangulated, we also have a suspension $\Sigma:\C\to\C$ .  We could use the Serre functor to replace the suspension with a translation $\te=\se^-\Sigma$.
  We discuss this in Section \ref{s:cy}.
\[ \text{ translated (T) } \longleftrightarrow \text{ suspended (S)} \]
These three choices give $2^3$ possibilities:
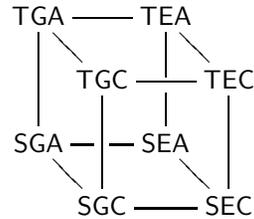
\[
\xymatrix@!0{
\text{TGA}\ar@{-}[rr]\ar@{-}[dd]
& & \text{TEA} \ar@{-}'[d][dd] \\
& \text{TGC} \ar@{-}[ul]\ar@{-}[rr]\ar@{-}[dd]
& & \text{TEC} \ar@{-}[ul]\ar@{-}[dd] \\
 \text{SGA} \ar@{-}'[r][rr] & & \text{SEA}
\\
& \text{SGC} \ar@{-}[rr]\ar@{-}[ul] & & \text{SEC} \ar@{-}[ul]
}
\]
Partly due to technical convenience and partly due to historical accident, some of these combinations are more well-studied than others.  Two corners are particularly well studied.
To the top left corner of our diagram, translated graded algebras (TGA), belongs the preprojective algebra of a quiver.  To the bottom right corner, suspended equivariant categories (SEC), belongs the derived category of a quiver.  We study how to move between these corners, keeping track of Serre duality.

We would like to swap the fractional Calabi-Yau relation $\se^m\cong \Sigma^N$ in a triangulated category for the relation $\se^{m-N}\cong \te^N$, but this is too naive: we need to keep track of the suspended structure and record this in a new translated structure.  To give a correct technical formulation, we note that Serre functors come with canonical commutation transformations.  In particular, if $(\D,F)$ is a $\Z$-equivariant category, such as a triangulated category with $F=\Sigma$, we have maps $\zeta_F:\se F\arr\sim F\se$.  The study of commutation transformations for triangulated functors goes back to Verdier \cite{verdier}, and their importance in studying Calabi-Yau structures was emphasised in work of Keller \cite{k-cy,k-orbit-corr}.  
With this insight, we can state our main result (see Theorems \ref{thm:scy-criteria} and \ref{thm:fcy-orbit}) as follows:
\begin{thma}
Let $\D$ be a category with automorphism $F:\D\to\D$ and let $\chi:\Z\to\kk^\times$ be a character.  Let $\te=\se^- F$.  Then we have an  $F$-equivariant isomorphism of functors 
\[ (\se, \chi(1) \zeta_F)^m\cong (F, \chi(1) \id_{F^2})^N \]
if and only if we have 
a $\te$-equivariant isomorphism of functors 
\[ (\se, \chi(1) \zeta_\te)^{m-N}\cong (\te, \id_{\te^2})^N. \]
Moreover, if the orbit category $\C=\D/\te$ has finitely many isomorphism classes of objects, we construct a graded Frobenius algebra $A$.  Then the natural isomorphisms above exist if and only if the $\chi$-Nakayama automorphism of $A$ is the identity map up to a shift by $N$.
\end{thma}

We apply this in the case of $d$-representation algebras and higher preprojective algebras.  Suppose $\Lambda$ is a basic $d$-representation finite algebra, so it has a $d$-cluster tilting module (see \cite{io-stab}).  Let $\Pi$ denote its higher preprojective algebra, which has a natural grading coming from its definition as a tensor algebra.  Then $\Pi$ is Frobenius, so it has a (classical) Nakayama automorphism $\alpha$.  Define $\beta(a)=(-1)^{dp}$ for $a\in\Pi$ a degree $p$ homogeneous element of $\Pi$.
Then, by Theorem \ref{thm:drf-fcy} and Corollary \ref{cor:fcyfinitenak}, we have:
\begin{thmb}
The derived category of a $d$-representation finite algebra $\Lambda$ is fractionally Calabi-Yau if and only if the graded Nakayama automorphism $\beta$ of its preprojective algebra has finite order.
\end{thmb}
The algebra automorphism $\beta$ is part of a degree-adjusted automorphism $(\beta,\ell)$ and this tells us about the Calabi-Yau dimension: $(\beta,\ell)^{m-N}\cong (\id,\underline{N})$ if and only if $\Db(\Lambda)$ is fractional Calabi-Yau of dimension $dN/m$.  Using Theorem B, we give a new proof that Iyama's higher Auslander algebras of type A are fractional Calabi-Yau, and we prove that Pasquali's 2-representation finite algebras coming from Postnikov diagrams are fractional Calabi-Yau.

\subsection*{Summary of contents}

We now outline the contents of this article.

Section \ref{s:cats} starts with the basic 2-dimensional category theory we will use, including adjunctions, equivalences, and 2-groupoids.  We describe idempotent completion as a 2-functor.  We then discuss monoidal categories, including duals, pivotal structures, and the Drinfeld centre.  Next we explain the Picard group of a monoidal category.  We then give basic facts about $\kk$-linear categories and a 2-categorical treatment of indecomposable objects.  Finally we discuss the connection between $\kk$-categories and $\kk$-algebras, and give the definition of a base algebra (Definition \ref{def:basealg}) which will play an important role later.

In Section \ref{s:serre} we meet the definition of a Serre structure on a category enriched in a monoidal category with pivotal structure.  We discuss uniqueness and transfer of structure across equivalences of categories.  We see the compatibility lemma: roughly, this says that if we know how a Serre functor acts on objects, and we have the structural isomorphisms in the definition of a Serre structure, then we can recover the action of the Serre functor on morphisms.  We revise the well-known result that Serre functors $\se$ commute with autoequivalences, note that this in fact gives us an element of the Drinfeld centre of the monoidal category of automorphisms, and prove a technical result on the commutation map $\se\se^-\to\se^-\se$ which will be useful later.  Finally, we discuss how Serre structures on $\kk$-categories are related to Frobenius structures on $\kk$-algebras.

In Section \ref{s:graded} we see the two notions of grading on a category.  The first we call an equivariant structure: this is where a group acts on our category.  We discuss strict and weak actions, and the equivariant centre.  The second we call a hom-graded structure: this is where the hom spaces form a ``graded algebra with several objects''.  Then we introduce graded Serre structures (in both settings) which depend on a character of our group.  In the hom-graded setting, they have appeared in the physics literature \cite{laz} but don't seem to have been used much.  This author finds working with an arbitrary character conceptually easier than keeping track of minus signs.  We carefully prove uniqueness results and make a connection to graded Frobenius structures.  Then we consider moving between the two notions of grading using the well-known operations of orbit categories and smash products.  We use Asashiba's theorem that these are 2-functors and they give an equivalence of 2-categories \cite{asa2}.  We show that both operations preserve Serre structures in the appropriate graded setting.  In the case of orbit categories, this is related to a result of Dugas \cite{dug-mesh}.

We restrict to the $\Z$-graded setting in Section \ref{s:cy}.  We start by discussing triangulated categories and triangulated functors, and recall results of Bondal-Kapranov and Van den Bergh on triangulated Serre functors.  Then we define an abstract (inverse) Auslander-Reiten functor $\te$ on a $\Z$-equivariant category.  We study ``change of action'', where the $\Z$-equivariant structure on a category is modified by an element of the Drinfeld centre, such a Serre functor.  Keeping the triangulated situation in mind, this allows us to prove a nice compatibility result for the Auslander-Reiten functor: the map $\te\te\to\te\te$ we obtain by change of action is just the identity map.  Next we study a formal, or ``synthetic'', version of the fractional Calabi-Yau property.  Given three functors $F$, $\se$, and $\te=\se^-F$, a relation can be expressed in three equivalent ways: $\se^m=F^N$, $\se^k=\te^N$, and $\te^m=F^k$, where $m=k+N$.  We provide an equivariant version of this basic idea, and observe that $\se^k=\te^N$ can be checked on the orbit category or, given a finiteness condition, on the graded base algebra.

Finally, we apply our theory in Section \ref{s:apps}.  First we outline how tensor products of chain complexes give graded functors between derived categories and use this to connect the strong Calabi-Yau property and the bimodule Calabi-Yau property.  Then we discuss Dynkin quivers.  After recalling the background information we quote the fractional Calabi-Yau result of Miyachi and Yekutieli, and the Nakayama automorphism result of Brenner, Butler, and King.  We show how these results are related and, largely, determine each other.  Then we explain how the theory is applied to the $d$-representation finite algebras which appear in Iyama's higher homological algebra: these are algebras of higher global dimension $0\leq d<\infty$ to which many parts of the representation theory of Dynkin quivers can be generalised.  We explain how to use Herschend and Iyama's calculation of the Nakayama automorphism for the preprojective algebras of ``higher type $A$ algebras'' to recover Dyckerhoff, Jasso, and Walde's result that these algebras are fractional Calabi-Yau.  We finish by considering algebras constructed from Postnikov diagrams by Pasquali \cite{pas}: these 2-representation finite algebras are always fractional Calabi-Yau and we show in an explicit example how to calculate their Calabi-Yau dimension.

\emph{Acknowledgements:} 
Thanks to Bethany Marsh for originally directing me to the fractional Calabi-Yau property.
Thanks to Alex Dugas, Martin Herschend, and Gustavo Jasso for helpful discussions and pointers to the literature.  
Thanks to an anonymous referee for helpful comments and corrections.

Parts of this paper were written during visits to the Institut des Hautes \'Etudes Scientifiques
and the Institut Henri Poincar\'e in Paris.  Thanks to both institutions, and the Jean-Paul Gimon Fund, for financial support and for providing a great working environment.

\section{2-categories} \label{s:cats}

\emph{Convention:} 
$g\circ f$ means $\bullet\arr f \bullet\arr g\bullet$.

% -------------------------- 2-categories --------------------------

\subsection{2-dimensional category theory}

We recall some standard facts from 2-dimensional category theory, partly to fix notations and terminology.  Precise definitions can be found in \cite{leibi} or \cite{jy}, and our discussion is also motivated by \cite{nlab}.

\subsubsection{Bicategories and 2-categories}

A bicategory is a weak $2$-dimensional category.  
Bicategories $\biB$ have 0-cells $x,y,\ldots$ (also called objects), $1$-cells $f,g:x\to y$ between 0-cells, and 2-cells $\alpha:f\to g$ between 1-cells.  We write $\ob\biB$ for the objects of $\biB$.  Between any two objects $x$ and $y$ there is a hom category $\biB(x,y)$ whose objects are the $1$-cells and whose morphisms are the $2$-cells.  
$\biB$ comes with (horizontal) composition functors 
\[ \biB(y,z) \times \biB(x,y) \to \biB(x,z) \]
which are associative up to specified natural isomorphism and with unit functors
\[ \{ \xymatrix{\star \ar@(ur,dr)[]|{id} 
 }\} \to  \biB(x,x) \]
which are unital up to specified natural isomorphism.
In particular, each object $x$ has an identity 1-cell $\id_x$, and each 1-cell $f:x\to y$ has an identity 2-cell $\id_f$.

\begin{example}
There is a bicategory of bimodules.  Its 0-cells are rings, its 1-cells are $R\da S$-bimodules, and its 2-cells are bimodule maps.  Horizontal composition is given by tensor product.  Note that the weak structure in the definition of a bicategory is necessary here: for an $R\da S$-bimodule $M$, we have $R\otimes_RM\cong M$ but $R\otimes_RM\neq M$.
\end{example}

A (strict, locally small) 2-category $\biC$ is a category enriched in categories.  This means that 
between each ordered pair of objects we have a morphism category $\biC(x,y)$, and the composition is functorial.  2-categories are bicategories where the natural  isomorphisms for associativity and unitality are identities.  
Note that the $0$- and $1$-cells of $\biC$ form a (classical) 1-category, called the underlying category of $\biC$.

\begin{example}
A category $\C$ is called \emph{small} if $\ob\C$ is a set.
There is a $2$-category $\Cat$ whose 0-cells are small categories, 1-cells are functors, and 2-cells are natural transformations.  
\end{example}

\subsubsection{2-functors}

Let $\biB$ and $\biC$ be bicategories.  A \emph{weak 2-functor} $\biF:\biB\to\biC$ is a function $\ob\biB\to\ob\biC$, which we also denote $\biF$, together with a collection of functors
\[ \biF_{x,y}: \biB(x,y)\to \biC(\biF x,\biF y) \]
which commute with the composition and unit functors up to specified natural isomorphisms.  

Let $\biB$ and $\biC$ be 2-categories.  A 2-functor $\biF:\biB\to\biC$ is a weak 2-functor which strictly preserves units and horizontal composition of 2-cells (see \cite[Proposition 4.1.8]{jy}).

Given two 2-functors $\biF,\biG:\biB\to\biC$,
a 2-transformation $\alpha:\biF\to \biG$ is a function sending each $x\in\ob\biB$ to a 1-cell $\alpha_x:\biF x\to \biG x$ such that, for all $f\in\biB(x,y)$, we have $\biG f\circ \alpha_x=\alpha_y\circ \biF f:\biF x\to \biG y$.

\begin{example}
There is a $2$-category $\tcat$ whose 0-cells are small 2-categories, 1-cells are 2-functors, and 2-cells are 2-transformations.  (As for categories, a 2-category $\biC$ is \emph{small} if $\ob\biC$ is a set.)
\end{example}

Given any property $P$ of functors, we say a $2$-functor $\biF:\biB\to\biC$ is \emph{locally $P$} if for all $x,y\in\ob\biB$, the component functor $\biF_{x,y}$
has property $P$.  For example, $\biF$ is locally an equivalence if every functor $\biF_{x,y}$ is an equivalence of categories.

\begin{definition}
Given a 2-category $\biB$ and a subset $S\subset\ob\biB$, the \emph{full sub-2-category} of $\biB$ on $S$ is the 2-category $\biB\res{S}$ with 0-cells $S$ and morphism categories $\biB\res{S}(x,y)=\biB(x,y)$.
\end{definition}
Note that every full sub-2-category comes with an inclusion 2-functor $\biB\res{S}\to\biB$ which is locally an equivalence.
\begin{lemma}\label{lem:2funfull}
Let $\biF:\biB\to\biC$ be a 2-functor and let $S\subset\ob\biB$ and $T\subset\ob\biC$.
If $\biF(S)\subseteq T$ then $\biF$ restricts to a 2-functor $\biF\res{S}:\biB\res{S}\to\biC\res{T}$.
\end{lemma}

\subsubsection{Equivalences and adjunctions}

A 2-cell is called \emph{iso} (or an \emph{isomorphism}) if it has a two-sided inverse.  If there exists an iso $2$-cell $\alpha:f\to g$ we write $f\isom g$.   
A 1-cell $f:x\to y$ is an \emph{equivalence} if there exists $g:y\to x$ such that $gf\isom \id_x$ and $fg\isom \id_y$, and we say $g$ is a \emph{quasi-inverse} of $f$.  If such an equivalence exists, we say $x$ and $y$ are equivalent.
A weak 2-functor $\biF:\biB\to\biC$ is a \emph{biequivalence} if every functor $\biF_{x,y}$ is an equivalence of categories and every object of $\biB$ is equivalent to $\biF x$ for some $x\in\ob\biB$.

The following simple fact is quite useful.
\begin{lemma}\label{lem:equvbiequiv}
Every equivalence $f:x\to y$ in a bicategory $\biB$ induces a biequivalence $F:\biB\to\biB$ such that $Fx=y$.
\end{lemma}
\begin{proof} This is a standard construction.  Define a weak 2-functor $\biF$ which acts on 0-cells by sending $x$ to $y$ and fixing other objects.  For $w,z\neq x$ let $\biF_{w,z}$ be the identity functor.  If one or both of $w,z$ are $x$, define $\biF_{w,z}$ by pre- or post-composing with $f$ or its quasi-inverse.
\end{proof}

An \emph{adjunction} in a 2-category is the data $(f,g,\eta,\varepsilon)$ where $f:x\to y$ and $g:y\to x$ are 1-cells and 
$\eta:\id_x\to gf$ and $\varepsilon:fg\to \id_y$ 
are 2-cells 
satisfying the triangle identities:
\[ \xymatrix @C=30pt {
f  \ar[rr]^{\id_f} \ar[dr]_{\id_{f}\eta} && f &&g \ar[rr]^{\id_g} \ar[dr]_{\eta \id_g} && g \\
& fgf \ar[ur]_{\varepsilon\id_f} & &&& gfg \ar[ur]_{\id_g\varepsilon} & 
}\]
$f$ is called the left adjoint and $g$ the right adjoint.  We write $f\dashv g$.

An \emph{adjoint equivalence} is an adjunction such that $\eta$ and $\varepsilon$ are both isomorphisms.  Any equivalence $f:x\arr\sim y$ can be upgraded to an adjoint equivalence $(f,g,\eta,\varepsilon)$.

\subsubsection{2-groupoids}\label{sss:2grpds}

\begin{definition}
A 2-category is called a \emph{2-groupoid} if all its 1-cells are equivalences and all its 2-cells are isomorphisms.
\end{definition}
The full sub-2-category of $\tcat$ on the 2-groupoids is called the 2-category of 2-groupoids, and is denoted $\tgpd$.

Every 2-category $\biC$ has a maximal sub-2-groupoid, which we denote $\core\biC$.  This construction is functorial: if $\biF:\biB\to\biC$ is a $2$-functor we get a new 2-functor $\core \biF:\core\biB\to\core\biC$.  It is not $2$-functorial on $\tcat$, because an arbitrary $2$-transformation involves component 1-cells which may not be isomorphisms, but it is 2-functorial on the sub-2-category $\cati$ of $\tcat$ which contains only iso $2$-transformations.

The $2$-category $\core(\Cat)$, which we denote $\catei$, will be particularly important for us.  
Its 0-cells are categories, its 1-cells are equivalences of categories, and its 2-cells are natural isomorphisms.

% -------------------------- Idempotent completion --------------------------

\subsubsection{Idempotent completion}

Let $\C$ be a category and let $e\in\C(x,x)$ be an idempotent, i.e., $e^2=e$.  We say that $e$ \emph{splits} if $\C$ has an object $y$ and and morphisms $f:x\to y$ and $g:y\to x$ such that $e=gf$ and $\id_y=fg$.  
If all idempotents in $\C$ split we say that $\C$ is \emph{idempotent complete}.  We have a full sub-2-category $\ic\Cat$ of $\Cat$ on the indempotent complete categories.

Any category $\C$ has an idempotent completion $\ic\C$.  Its objects are pairs $(x,e)$ where $x\in\ob\C$ and $e^2=e\in\C(x,x)$.  The morphism set from $(x,e)$ to $(y,d)$ is $d\C(x,y)e=\{dfe\st f\in\C(x,y)\}$.
Note $\ic\C$ is idempotent complete, and we have a fully faithful inclusion functor $\C\into\ic\C$.
\begin{lemma}
If $\C$ is idempotent complete then $\C\into\ic\C$ is dense, so $\C$ and $\ic\C$ are equivalent.
\end{lemma}

This construction extends to a 2-functor $\ic:\Cat\to\Cat$ as follows.  
 Given a functor $F:\C\to\D$ we define $\ic F$ on objects by $(x,e)\mapsto (Fx,Fe)$ and on maps it is just $F$. 
For a natural transformation $\alpha:F\to G$, define components by 
$$\ic \alpha_{(x,e)}:Fx\arrr{Fe}Fx\arrr{\alpha_x}Gx\arrr{Ge} Gx.$$

The full subcategory of $\ic\C$ on objects $(x,\id_x)$ is equivalent to $\C$.  By considering the components of natural transformations on these objects, we get the following:
\begin{lemma}\label{lem:ic-lff}
The 2-functor $\ic:\Cat\to\Cat$ is locally fully faithful.
\end{lemma}
\begin{proof}
Given functors $F,G:\C\to\D$ and natural transformations $\alpha,\beta:F\to G$, if $\ic\alpha=\ic\beta$ then in particular, for all $x\in\C$, $\alpha_x=\ic\alpha_{(x,\id_x)}=\ic\beta_{(x,\id_x)}=\beta_x$, so $\alpha=\beta$.  Thus $\ic_{\C,\D}$ is faithful.  Fullness follows from naturality: given a natural transformation $\gamma:\ic F\to\ic G$, define $\tilde{\gamma}:F\to G$ by $\tilde{\gamma}_x=\gamma_{(x,\id_x)}$ and then the naturality square for $(x,e)\to(x,\id_x)$, together with $e^2=e$, shows that $\ic\tilde{\gamma}=\gamma$.
\end{proof}

\subsection{Monoidal categories}\label{ss:moncat}

Monoidal categories will appear in at least two places: as enriching categories (e.g., a monoidal category with vector spaces as objects) and as endomorphism categories (e.g., a monoidal category with endofunctors as objects).

We outline just what we need; precise definitions can be found in \cite{egno}.  

\subsubsection{Monoidal categories and monoidal transformations}

A monoidal category $\M$ is a category with a tensor product bifunctor $-\otimes-:\M\times\M\to\M$ and a unit object $\unit$.  We often omit associator isomorphisms, pretending $\M$ is strictly associative; by Mac Lane's coherence theorem we can do this without loss of generality.
We do not assume $\M$ is symmetric or even braided.  

One-object bicategories correspond to monoidal categories by a process known as ``looping''.
\begin{example}
Any object $x$ in a bicategory $\biB$ 
has an endomorphism category $\biB(x,x)$ which has monoidal structure given by composition.  
\end{example}

Let $\M,\N$ be two monoidal categories.  A \emph{monoidal functor} from $\M$ to $\N$ is a pair $(F,J)$ consisting of a functor $F:\M\to\N$ and a binatural isomorphism $J:F(-)\otimes F(-)\arr\sim F(-\otimes-)$ satisfying unit and associativity conditions.  
If $J$ is the identity, we say $F$ is a \emph{strict} monoidal functor.
We get a category $\Mon$ of monoidal categories and monoidal functors.  

A \emph{monoidal transformation} $\alpha:(F,J)\to (G,K)$ between monoidal functors is a natural transformation such that, for all $X,Y\in\M$, the following diagram commutes:
\[ \xymatrix @C=30pt {
FX\otimes FY\ar[d]^{J_{X,Y}} \ar[r]^{\alpha_X\otimes \alpha_Y} & GX \otimes GY \ar[d]^{K_{X,Y}}  \\
F(X\otimes Y) \ar[r]^{\alpha_{X\otimes Y}} &G(X \otimes Y)
} \]
We get a 2-category of monoidal categories, monoidal functors, and monoidal transformations.   
If $F$ is an equivalence of categories then we say $(F,J)$ is a \emph{monoidal equivalence}.
These are precisely the 1-cells which are equivalences.

\subsubsection{Duals}

A  \emph{dual pair} in $\M$ is a 4-tuple $(L, R, c,e)$ where $L$ and $R$ are objects in $\M$ and $c:\unit\to R\otimes L$ (``coevaluation'') and $e: L\otimes R\to\unit$ (``evaluation'') are maps in $\M$ which satisy the usual triangle identities:  
\[ \xymatrix{
L \ar[rr]^{\id_L} \ar[dr]_{\id_L\otimes c} && L && R  \ar[rr]^{\id_R} \ar[dr]_{c\otimes \id_R} && R \\
& L\otimes R \otimes L \ar[ur]_{ e\otimes \id_L} & && & R\otimes L\otimes R\ar[ur]_{\id_R \otimes e} & 
}\]
We say $L$ is left dual to $R$, written $L={R^\vee }$, and $R$ is right dual to $L$, written $R={^\vee L}$.  So the left dual is written on the left in an evaluation map.  Note that the left dual is written with a superscript on the \emph{right}, and vice versa.  However, left matches left in the following:
\begin{example}
If $f:x\to x$ is left adjoint to $g:x\to x$ in a bicategory $\biB$, then 
$f$ is left dual to $g$ in the monoidal category $\biB(x,x)$.
\end{example}

Left and right duals are defined up to canonical isomorphism: 
if $L_1$ and $L_2$ are both left duals of $R$, so we have two dual pairs $(L_1,R,c_1,e_1)$ and $(L_2,R,c_2,e_2)$, then one can check that the map
\[ L_1 \arrr{\id\otimes c_2} L_1\otimes R\otimes L_2 \arrr{e_1\otimes \id} L_2 \]
is iso.
Duals are preserved by monoidal functors.  The unit object $\unit$ is self-dual.  

We will mainly work with left duals, and will sometimes add a subscript to the maps $c$ and $e$ as follows:
\[ c_X: \unit\to X\otimes X^\vee \;\; \text{ and } \;\; e_X: X^\vee\otimes X \to \unit.\]

From now on we assume $\M$ is \emph{left rigid}, i.e., every object has a left dual.
A category which is both left and right rigid is simply called \emph{rigid} (or, in some other parts of the literature, \emph{autonomous}).

Given a map $f:X\to Y$ we get a left dual map $f^\vee:Y^\vee\to X^\vee$ defined by
\[ Y^\vee \arrr{\id\otimes c_X} Y^\vee\otimes X\otimes X^\vee  \arrr{\id\otimes f\otimes \id} Y^\vee\otimes Y\otimes X^\vee \arrr{\ev_Y\otimes\id} X^\vee \]
If we fix a dual $X^\vee$ for each object $X\in\M$ then taking duals is a monoidal functor $\M\to\M^\op$, where $\M^\op$ is the same category as $\M$ with the opposite tensor product \cite[Exercise 2.10.7]{egno}.

From the definitions, we get the following useful results:
\begin{lemma}\label{lem:dualevcoev}
For all $X\in\M$ we have $e_X^\vee=c_{X^\vee}$ and $c_X^\vee=e_{X^\vee}$.  
\end{lemma}
\begin{lemma}\label{lem:coevisom}
If the map $f:X\to Y$ in $\M$ is an isomorphism then the composition
\[ \unit \arrr{c_X} X\otimes X^\vee \arrr{f\otimes (f^{-1})^\vee} Y\otimes Y^\vee \]
is the coevaluation $c_Y$.
\end{lemma}

The following lemma is stated in \cite[Lemma 4.9]{sel} and credited to Saavedra Rivano \cite[Prop. 5.2.3]{saav}.  The reference isn't easy to access, so we give a proof.
\begin{lemma}\label{lem:saav}
If $\alpha:(F,J)\to (G,K)$ is a monoidal transformation 
then the following diagram commutes:
\[ \xymatrix{
F(X^\vee) \ar[r]^{\alpha_{X^\vee}} \ar[d]^\sim& G(X^\vee) \ar[d]^\sim \\
F(X)^\vee   & G(X)^\vee\ar[l]^{\alpha_{X}^\vee}
} \]
In particular, every monoidal transformation between functors on rigid categories 
is an isomorphism.
\end{lemma}
\begin{proof}
We want to show that $\alpha_{X}^\vee\alpha_{X^\vee}=\id_{FX^\vee}:FX^\vee\to FX^\vee$.  Using the definition of dual maps, and the adjunction $\M(FX^\vee,FX^\vee)\cong\M(FX^\vee\otimes FX,\unit)$, this is equivalent to showing that the map
\[ FX^\vee\otimes FX \arrr{\alpha_{X^\vee}\otimes{\alpha_{X}}} GX^\vee\otimes GX \arrr{e_{GX}} \unit \]
is the evaluation for $FX$.  This follows from monoidality, naturality, and unitality, as in the following commutative diagram:
\[\xymatrix{
 FX^\vee\otimes FX \ar[r]^{\alpha_{X^\vee}\otimes \alpha_X} \ar@/_3pc/[ddd]_{e_{FX}}   & GX^\vee\otimes GX \ar@/^3pc/[ddd]^{e_{GX}}  \\
 F(X^\vee\otimes X) \ar[u]_{J_{X^{\vee},X}}^\sim \ar[r]^{\alpha_{X^\vee\otimes X}} \ar[d]^{Fe_X} &
   G(X^\vee\otimes X) \ar[u]^{K_{X^{\vee},X}}_\sim \ar[d]_{Ge_X} \\
 F\unit \ar[r]^{\alpha_\unit} \ar[d]_\sim & G\unit \ar[d]^{\sim}  \\
 \unit \ar@{=}[r] & \unit
}\]
\end{proof}

\subsubsection{A technical lemma}\label{ss:techlemma}

Given a map $f:X\otimes Y\to Z$, we can dualize to get $f^\vee:Z^\vee \to Y^\vee\otimes X^\vee$.  Then we define two maps $f^\ell$ and $f^r$ as follows:
\[ \xymatrix{
Y^{\vee\vee}\otimes Z^\vee \ar[rr]^{f^\ell} \ar[dr]_{1\otimes f^\vee} && X^\vee && Z^\vee\otimes X  \ar[rr]^{f^r} \ar[dr]_{f^\vee \otimes 1} && Y^\vee \\
& Y^{\vee\vee}\otimes Y^\vee \otimes X^\vee \ar[ur]_{e_{Y^\vee}\otimes 1} & && & Y^\vee\otimes X^\vee\otimes X \ar[ur]_{1\otimes e_X} & 
}\]
The following lemma is immediate from the triangle identities and the definition of $f^\vee$:
\begin{lemma}\label{lem:altduals}
The following diagrams commute (we omit the $\otimes$ sign):
\[ \xymatrix{
(Y^{\vee\vee}) ( Z^\vee) \ar[r]^{f^\ell} \ar[d]_{1  c_{X  Y}} & X^\vee & (Z^\vee)  X  \ar[r]^{f^r} \ar[d]_{1   c_Y} & Y^\vee \\
(Y^{\vee\vee})  (Z^\vee)  X   Y   (Y^\vee)  (X^\vee) \ar[r]_{1  f  1} & (Y^{\vee\vee})  (Z^\vee)  Z   (Y^\vee)  (X^\vee) \ar[u]^{e_{Z  Y^\vee}  1} & (Z^\vee)  X  Y  (Y^\vee)  \ar[r]_{1  f  1}   & (Z^\vee)  Z  (Y^\vee) \ar[u]^{e_Z  1}
}\]
\end{lemma}

As a corollary, we get a useful technical lemma:
\begin{lemma}\label{lem:tech1}
The following diagram commutes:
\[ \xymatrix{
Y^{\vee\vee} \ar[rr]^{(f^r)^\vee} \ar[dr]^{1\otimes c_{Z^\vee}} && X^\vee\otimes Z^{\vee\vee} \\
 & Y^{\vee\vee} \otimes Z^\vee \otimes Z^{\vee\vee} \ar[ur]^{f^\ell\otimes 1} &
}\]
\end{lemma}
\begin{proof}
Apply Lemmas \ref{lem:dualevcoev} and  \ref{lem:altduals} and use the definitions of $f^\ell$ and $f^r$. 
\end{proof}

\subsubsection{Pivotal categories}\label{sss:pivotal}

A \emph{pivotal structure} on a rigid category $\M$ is a natural isomorphism of monoidal functors from the identity functor on $\M$ to the double dual.  Therefore, for any object $X$, we have $\iota_{X}^\vee=\iota_{X^\vee}^{-1}$ by Lemma \ref{lem:saav}.

Unwinding the definition of a pivotal structure, for each object $X\in\M$ we have an isomorphism
$\iota_X:X\arr\sim X^{\vee\vee}$ satisfying $\iota_{X\otimes Y}=\iota_X\otimes \iota_Y$ and, for any map $f:X\to Y$, the following diagram commutes:
\[\xymatrix{
X\ar[r]^{\iota_X} \ar[d]^f & X^{\vee\vee}\ar[d]^{f^{\vee\vee}} \\
Y\ar[r]^{\iota_Y} & Y^{\vee\vee}
}\]
A \emph{pivotal} category (sometimes called a \emph{sovereign} category) is a rigid category with a specified pivotal structure.  Pivotal categories have a useful graphical calculus: see \cite{sel} or \cite{tv-mon} for surveys.

\begin{example}\label{eg:vec-piv}
The category $\fdv$ of finite dimensional vector spaces over a field $\kk$ is left rigid, with $V^\vee=V^*=\Hom_{\kk}(V,\kk)$, the space of linear functions from $V$ to $\kk$.  The evaluation map $V^*\otimes V\to\kk$ sends $f\otimes v$ to $f(v)$.  It has a pivotal structure $\iota_V:V\to V^{**}$ sending $v\in V$ to $\ev_v\in V^{**}$ which acts by $\ev_v(f)=f(v)$. 
\end{example}

\begin{example}
Given a bicategory $\biB$ and an object $x\in\biB$, the full subcategory $\Autoeq_\biB(x)\subseteq \biB(x,x)$ of equivalences $f:x\to x$ is left rigid: left duals $f^\vee$ are precisely left adjoints of $f$.  As $f$ is an equivalence, its left and right adjoints are canonically isomorphic.  This gives a pivotal structure on $\Autoeq_\biB(x)$. 
\end{example}

Note that a pivotal structure identifies left and right duals.  
In fact, such an identification is equivalent to the existence of a pivotal structure, and in some sources it is taken as the definition.

\subsubsection{The centre}\label{sss:centre}

Let $\M$ be a monoidal category.
\begin{definition}\label{def:centre}
The \emph{centre} (or \emph{Drinfeld centre}) of $\M$ is the following category $\centre(\M)$.  Its objects are pairs $(Z,z)$ where $Z\in\ob\M$ and $z$ is a natural isomorphism $Z\otimes-\arr\sim -\otimes Z$ such that, for all $X,Y\in \M$, the following diagram commutes:
\[ \xymatrix{
Z\otimes X\otimes Y \ar[rr]^{z_{X\otimes Y}} \ar[dr]_{z_X\otimes\id_Y} && X\otimes Y\otimes Z \\
 &  X\otimes Z\otimes Y \ar[ur]_{\id_X\otimes z_Y} &
}\]
Its morphisms $f:(Z,z)\to (W,w)$ are morphisms $f:Z\to W$ in $\M$ such that, for all $X\in \M$, the following diagram commutes:
\[\xymatrix{
Z\otimes X \ar[r]^{z_X} \ar[d]^{f\otimes \id_X} & X\otimes Z \ar[d]^{\id_X\otimes f} \\
W\otimes X \ar[r]^{z_W}  & X\otimes W 
}\]
\end{definition}
$\centre(\M)$ is a monoidal category, with tensor product $(Z,z)\otimes (W,w)=(Z\otimes W, (z\otimes\id)\circ (\id\otimes w))$.

Suppose $(Z,z)\in\centre(\M)$ and $Z$ has left dual $Z^\vee$, with evaluation and coevaluation maps $e:Z^\vee\otimes Z\to\unit$ and $c:\unit\to Z\otimes Z^\vee$.  Then we define a natural transformation $z^\dagger:Z^\vee\otimes-\arr\sim -\otimes Z^\vee$ with components constructed as follows:
\[ z^\dagger_X:Z^\vee\otimes X\arrr{\id\otimes c} Z^\vee\otimes X\otimes Z\otimes Z^\vee \arrr{\id\otimes z_X^{-1}\otimes\id} Z^\vee\otimes Z\otimes X\otimes Z^\vee \arrr{e\otimes\id} X\otimes Z^\vee \]
The following is easy to check: see \cite[Section 5.2]{tv-mon}.
\begin{lemma}\label{lem:dualincentre}
$(Z^\vee,z^\dagger)$ is left dual to $(Z,z)$, with evaluation and coevaluation maps $e$ and $c$.
\end{lemma}
Similarly, a pivotal structure on $\M$ induces a pivotal structure on $\centre(\M)$.

The following result is also useful.  Its proof follows from the naturality of $z$: see \cite[Exercise 5.1.5]{tv-mon}.  Let $(Z,z)\in\centre(\M)$.
\begin{lemma}\label{lem:inversebraid}
If $z_\unit=\id_Z$ and 
$Z$ has a right dual ${^\vee Z}$ then $z_X$ has two-sided inverse $z_{(^\vee Z)}^\dagger$.
\end{lemma}
Therefore, if (i) $z_\unit=\id_Z$, (ii) $\M$ is rigid, and (iii) we have a pivotal structure $\iota_X:X\to X^{\vee\vee}$, then the inverse of $z_X$ is the following map:
\[ z_X^{-1}: X\otimes Z\arrr{\id\otimes c} X\otimes Z\otimes X^\vee\otimes X^{\vee\vee}  \arrr{\iota_X\otimes z_{X^\vee}\otimes \iota_X^{-1}} X^{\vee\vee}\otimes X^\vee\otimes Z\otimes X \arrr{e\otimes\id} Z\otimes X.\]

% -------------------------- Picard groups --------------------------

\subsection{Picard groups}

Let $\M$ be a monoidal category.  We say that an object $X\in\M$ is \emph{invertible} if there exists another object $Y\in\M$ such that $X\otimes Y$ and $Y\otimes X$ are both isomorphic to $\unit$.  Note that invertible objects are dualizable, and their evaluation and coevaluation maps are isomorphisms.

The following definition was given in \cite[Definition A.2.7]{hps} in the case of a closed symmetric monoidal category, based on earlier work in algebraic topology (e.g., \cite{hopms}).  As suggested at the start of \cite[Section 2]{may-pic}, it also works in the non-symmetric case, and this is useful for algebra.
\begin{definition}\label{def:pic}
The \emph{Picard group} of $\M$, denoted $\Pic(\M)$, is the set of isomorphism classes $[X]$ of invertible objects $X$ of $\M$ with group operation $[X][Y]=[X\otimes Y]$.
\end{definition}
Equivalently, one could take the maximal subgroup of the monoid given by isomorphism classes of all objects.

Definition \ref{def:pic} generalizes the following classical situation.
\begin{example}
Let $(X,\Oring_X)$ be a ringed space and let $\Sh X$ be the monoidal category of sheaves on $X$, with monoidal structure given by the usual tensor product of sheaves.  Then $\Pic(\Sh X)$ is the usual Picard group $\Pic(X)$.  (See Prop II.6.12 of \cite{har}.)
\end{example}

The following construction gives many examples of Picard groups.
\begin{example}
If $x$ is an object in a bicategory $\biB$, define
$\PicEnd_\biB(x)$ to be the Picard group of the monoidal category $\biB(x,x)$.
\end{example}
One could define the Picard groupoid of a bicategory, but we will not need this.

There are other examples coming from algebra.
\begin{example}
Let $R$ be a ring, which we do not assume is commutative.  Let $R\bimod R$ be the monoidal category of finitely generated $R\da R$-bimodules, with monoidal structure given by tensor product over $R$.  Then $\Pic(R\bimod R)$ is the group of invertible finitely generated $R\da R$-bimodules, sometimes called the Picard group of $R$.  Note that, by Morita theory, we can recover this group from functors between categories of left modules: $\Pic(R\bimod R)\cong \PicEnd_{\Cat}(R\mMod)$.
\end{example}
\begin{example}
Let $R$ be a ring and let $\Db(R\da R)$ be the derived category of $R\da R$-bimodules, with monoidal structure given by derived tensor product over $R$.  Then $\Pic(\Db(R\da R))$ is the derived Picard group of $R$ (see \cite{yeku} and \cite{rz}).  
Let $\Tri$ denote the 2-category of triangulated categories and let $\Db(R)$ be the derived category of $R$-modules.  Then, by Rickard's theory \cite{ric2}, we have an embedding of groups
$\Pic(\Db(R\da R))\into \PicEnd_\Tri(\Db(R))$.
\end{example}
Derived Picard groups have been computed in some cases: see, for example, \cite{my,vz,mi-dpg}.

\begin{proposition}
Taking Picard groups gives a functor $\Pic:\Mon\to\Grp$.
\end{proposition}
\begin{proof}
Given a monoidal functor $(F,J)$, write $[F]=\Pic((F,J))$ with $J$ implicit.  Define a function $[F]:\Pic(\M)\to\Pic(\N)$ in the obvious way: $[F][X]=[FX]$.  This is well-defined on isomorphism classes of objects because $F$ is a functor, and clearly preserves identity morphisms.  The isomorphisms given by $J$ ensure that $FX$ is invertible, so $[FX]\in\Pic(\N)$, and that $[F]$ is a homomorphism.
\end{proof}
So given a monoidal functor $(F,J):\M\to\N$ we get a group homomorphism $$\Pic(F,J):\Pic(\M)\to\Pic(\N).$$
Functoriality immediately implies that isomorphic monoidal categories have isomorphic Picard groups.  

The following result is easy.
\begin{proposition}\label{prop:picinjsurj}
Let $(F,J)$ be a monoidal functor.
\begin{enumerate}[label=(\alph*)]
\item If $F:\M\to\N$ is full and faithful then $\Pic(F,J)$ is injective.
\item $F$ is essentially surjective if and only if $\Pic(F,J)$ is surjective.
\end{enumerate}
\end{proposition}
Note that injectivity of $\Pic(F,J)$ does not imply that $F$ is full or that $F$ is faithful.  
\begin{corollary}\label{cor:moneqpiciso}
$\Pic$ sends monoidal equivalences to group isomorphisms.
\end{corollary}

\begin{proposition}\label{prop:bieqmoneq}
A biequivalence $\Phi:\biC\to\biB$ induces a monoidal equivalence $$(F,J):\End_\biC(\C)\arr\sim\End_\biB(\Phi\C)$$ for every $\C\in\biC$.
\end{proposition}
\begin{proof}
We set $F=\Phi_{\C,\C}$ and $J$ is given by the natural isomorphism which compares composition in $\biC(\C,\C)$ and $\biB(\Phi\C,\Phi\C)$.
\end{proof}
Together, Proposition \ref{prop:bieqmoneq} and Corollary \ref{cor:moneqpiciso} imply: 
\begin{corollary}
A biequivalence $\Phi:\biC\to\biB$ induces group isomorphisms $$\PicEnd_\biC(\C)\cong\PicEnd_\biB(\Phi\C).$$
\end{corollary}
So, from Proposition \ref{prop:bieqmoneq} and Lemma \ref{lem:equvbiequiv}, we get:
\begin{corollary}
Equivalent objects in a bicategory have isomorphic PicEnd groups.
\end{corollary}

\subsection{$\kk$-linear and additive categories} 

Much of the material here can be found in Chapters 2 and 3 of Gabriel and Roiter's book \cite{garo}.

Fix a field $\kk$.  
Throughout, we work with $\kk$-linear categories, also called $\kk$-categories, in which all hom spaces are $\kk$-modules.  Note that products and coproducts in such categories coincide: they are biproducts, given by the direct sum.    In general we do not assume our categories are additive, so biproducts may not exist.
We always work with $\kk$-linear functors, which automatically preserve biproducts: see \cite[Section VIII.2]{macl}.  These are the 0- and 1-cells of a 2-category $\kk\Cat$ with arbitrary natural transformations as 2-cells.  As in Section \ref{sss:2grpds}, we have variants 
$\kk\cati$ and $\kk\catei$
where cells of dimension $\geq2$ and $\geq1$, respectively, are equivalences.

Note that idempotent completion restricts to a locally fully faithful 2-functor on $\kk$-linear categories.

A $\kk$-category is called \emph{additive} if it has all finite biproducts.
Given any $\kk$-category $\C$, we can form a new $\kk$-category $\Mat\C$ whose objects are formal finite direct sums of objects in $\C$.  Maps are given by matrices with entries in $\C$.  For example, if $\C$ is the one-object category with endomorphism ring $\kk$ then $\Mat\C$ is equivalent to the category of $\kk$-vector spaces.

\begin{lemma}\label{lem:mat-idemp}
$\Mat\C$ is additive.
If $\C$ is additive then $\C$ and $\Mat\C$ are equivalent.
\end{lemma}

The construction of $\Mat\C$ extends to a $2$-functor $\Mat:\kk\Cat\to\kk\Cat$ as follows.  
Given two $\kk$-categories $\C$ and $\D$ and a functor $F:\C\to\D$ between them, define $\Mat F:\Mat\C\to\Mat \D$ on objects by $(\Mat F)(x\oplus y)=Fx\oplus Fy$.  On maps, just apply $F$ termwise to the matrix.
Given a natural transformation $\eta:F\to G$ with components $\eta_x:Fx\to Gx$, 
define $(\Mat\eta)_{x\oplus y}:Fx\oplus Fy\to Gx\oplus Gy$ as the diagonal matrix with entries $\eta_x$ and $\eta_y$.

The full subcategory of $\Mat\C$ on the objects $x\in\C$ is isomorphic to $\C$.  So, similarly to Lemma \ref{lem:ic-lff}, we get the following:
\begin{lemma}\label{lem:mat-lff}
The 2-functor $\Mat:\kk\Cat\to\kk\Cat$ is locally fully faithful.
\end{lemma}

\begin{remark}
One might wonder if Lemmas \ref{lem:ic-lff} and \ref{lem:mat-lff} are both consequences of some formal result in category theory about free 2-functors or Cauchy completion (see \cite{lawvere-metric}). 
\end{remark}

Note that it is possible for a morphism in $\Mat\C$ to be an isomorphism even if none of its matrix components is an isomorphism in $\C$. 
\begin{example}
Let $\fdv$ be the category of finite dimensional vector spaces and let $V_i$ be a vector space of dimension $i$. 
Let $\C$ be the full subcategory of $\fdv$ on the objects $V_2$ and $V_3$.  Then, in $\Mat\C$, there is an isomorphism between $V_2\oplus V_2\oplus V_2$ and $V_3\oplus V_3$.
\end{example}
In the previous example, things go wrong because $\C$, and hence also $\Mat\C$, is not idempotent complete.

%%% Indecomposable objects

\begin{definition}
Let $\C\in\ob\kk\Cat$.  
We say that an object $x\in\ob\C$ is \emph{indecomposable} if it is nonzero and there do not exist nonzero objects $y,z\in\ob\C$ such that $x\cong y\oplus z$.  
\end{definition}
\begin{proposition}\label{prop:ind2fun}
Taking the full subcategory of indecomposable objects defines a 2-functor
\[\ind:\kk\catei\to\kk\catei\]
and this 2-functor is locally fully faithful on the full sub-2-category of $\catei$ on categories where every object is a finite direct sum of indecomposable objects.
\end{proposition}
\begin{proof}
Equivalences of categories preserve limits, such as direct sums, and therefore preserve indecomposable objects.  So equivalences $F:\C\to\D$ do restrict to equivalences $\ind F:\ind\C\to\ind\D$.  As isomorphisms preserve indecomposable objects, components of natural isomorphisms also restrict to indecomposable objects.

Now we restrict to the full sub-2-category where every object is a finite direct sum of indecomposable objects.
By naturality, natural transformations between functors $F,G:\C\to\D$ are determined by the components of indecomposable objects.  So $\ind_{\C,\D}$ is faithful.  And given a natural transformation $\gamma:\ic F\to\ic G$, we define $\tilde{\gamma}:F\to G$ by extending to any direct sums in the obvious way.  Then $\ind\tilde{\gamma}=\gamma$, so $\ind_{\C,\D}$ is full.
\end{proof}
If $F:\C\to\D$ does not preserve indecomposables then $\ind F$ is not defined, so we cannot extend $\ind$ to a 2-functor on $\kk\cati$.

Similar to Lemma \ref{lem:mat-idemp}, we have:
\begin{lemma}\label{lem:ind-idemp}
If every object of $\C$ is indecomposable then $\C$ and $\Ind\C$ are equivalent.
\end{lemma}

We say an idempotent $e=e^2$ is \emph{primitive} when it cannot be written as the sum of two other nonzero idempotents.
The following result is straightforward.
\begin{lemma}\label{lem:indic}
An object $(x,e)\in\ic\C$ is indecomposable precisely when $e$ is a primitive idempotent.
\end{lemma}

Also, directly from the definitions, we have the following.
\begin{lemma}\label{lem:matind}
If $\C$ is $\kk$-linear then $\Mat\C\cong\Mat\Ind\C$ and $\Ind\C\cong\Ind\Mat\C$.
\end{lemma}

The following terminology is used in \cite[Section 3.5]{garo}.
\begin{definition}
We say a $\kk$-category $\C$ is \emph{finite} if: 
\begin{itemize}
\item $\C$ has finitely many isomorphism classes of indecomposable objects, and
\item $\C$ is \emph{hom-finite}, i.e., all hom spaces $\C(x,y)$ are finite-dimensional.
\end{itemize}
\end{definition}
\begin{lemma}
If $\C$ is finite then every object of $\C$ is isomorphic to a direct sum of finitely many indecomposable objects.
\end{lemma}
\begin{proof}
Suppose not, then we have an object $x\cong \bigoplus_{i=1}^\infty x_i$ with each $x_i$ nonzero.  Each object $x_i$ has an identity morphism $\id_i$ so, as hom functors preserve limits, the infinite dimensional vector space with basis $\{\id_i\}_{i=1}^\infty$ is a subspace of $\C(x,x)$.  So $\C$ is not hom-finite.
\end{proof}

\begin{definition}
We define the following full sub-2-categories of $\kk\catei$:
\begin{itemize}
\item $\BaseCat$ is the full sub-2-category on the idempotent complete $\kk$-categories where every object is indecomposable; 
\item $\AddCat$ is the full sub-2-category on the additive idempotent complete $\kk$-categories.
\end{itemize}
We define $\fBaseCat$ and $\fAddCat$ as the full sub-2-categories of 
$\BaseCat$ and $\AddCat$, respectively,
on the finite $\kk$-categories.
\end{definition}
By Lemma \ref{lem:matind}, we have:
\[ \BaseCat  = \Ind\ic\kk\catei \;\; \text{ and } \;\; \AddCat = \Mat\ic\kk\catei. \]

\begin{proposition}\label{prop:base-add}
$\Mat$ and $\Ind$ restrict to 2-functors between $\BaseCat$ and $\AddCat$,
and $\Mat$ is left adjoint to $\Ind$ in the 2-category $\tgpd$.
\[ \xymatrix @R=30pt {
  \BaseCat \ar@/^/[rr]^\Mat   & \dashvvert & \AddCat \ar@/^/[ll]^\Ind   
} \]
Moreover, they further restrict to $\fBaseCat$ and $\fAddCat$, and on this restriction the above adjunction becomes an adjoint equivalence.  In particular, $\fBaseCat$ and $\fAddCat$ are 2-equivalent.
\end{proposition}
\begin{proof}
First, by Lemma \ref{lem:matind}, we have $\Mat\BaseCat\C \subseteq \AddCat\C$ and $\Ind\AddCat\C \subseteq \BaseCat\C$.  
So by Lemma \ref{lem:2funfull} they restrict to 2-functors.

We have 2-transformations
\[ \eta:\id_{\BaseCat}\to\Ind\Mat \;\; \text{ and } \;\; \varepsilon:\Mat\Ind\to\id_{\AddCat} \]
defined by the obvious inclusion functors.  They satisfy the triangle identities by Lemma \ref{lem:matind}, so we do have a 2-adjunction.  Note that, by Lemma \ref{lem:mat-lff} , $\eta$ is iso.

Using Lemma \ref{lem:2funfull} again, the 2-functors restrict to the finite sub-2-categories. 
If $\C$ is an additive category where every obejct is a direct sum of finitely many indecomposables then $\varepsilon_\C:\Mat\Ind\C\to\C$ is iso, so we do get an adjoint equivalence in $\tgpd$.
\end{proof}

\subsection{Algebras}\label{ss:algebras}

There is a category of (associative, unital) $\kk$-algebras whose morphisms are unital $\kk$-linear functions which preserve the multiplication.  In fact, this is a truncation of a $2$-category $\kk\Alg$.  Given $1$-cells $f,g:A\to B$, the $2$-cells from $f$ to $g$ are by definition
\[ \kk\Alg(A,B)(f,g)=\{ b\in B \st \forall a\in A, \;\; g(a)b=bf(a) \}. \]
The composition is as in $B$.  Note that $f$ and $g$ are isomorphic $1$-cells if they are related by an inner automorphism of $B$.

Given a $\kk$-algebra $A$, we can form the one-object $\kk$-linear category $\C_A$ whose hom-space is just $A$.  This defines a $2$-functor $\deloop:\kk\Alg\to \kk\Cat$ (which we may think of as delooping).
\begin{lemma}\label{lem:deloop-leq}
$\deloop$ is locally an equivalence.
\end{lemma}
\begin{proof}
The functors $\deloop_{A,B}:\kk\Alg(A,B)\to\kk\Cat(\C_A,\C_B)$ are in fact isomorphisms of categories, by construction.
\end{proof}

An algebra $A$ is called \emph{basic} if the summands of the regular left $A$-module are all non-isomorphic.  We have a full sub-2-category $\fBasicAlg$ of $\kk\Alg$ on the finite-dimensional basic algebras.

If $\C$ is a $\kk$-category with finitely many objects then we can form the following $\kk$-algebra:
\[ A_\C=\bigoplus_{x,y\in\ob\C}\C(x,y) \]
with multiplication of $f:w\to x$ and $h:y\to z$ given by
\[ hf = 
\begin{cases}
hf & \text{ if } x=y; \\
0 & \text{ otherwise.}
\end{cases}
\]

Let $\skel\D$ denote a skeletal subcategory of $\D$.
\begin{definition}\label{def:basealg}
If $\C$ is a finite $\kk$-category, the \emph{base category} of $\C$ is $\D=\Ind\ic\C$ and the \emph{base algebra} of $\C$ is $B_{\skel\D}$.
\end{definition}
By construction, $B_\D$ is a basic algebra, and $\C$ determines $B_\D$ up to isomorphism.  If $A$ is an algebra of finite representation type, i.e., $A\mMod$ is finite, then the base algebra of $A\mMod$ is just the Auslander algebra of $A$.

Note that, once we move away from single object categories, the construction of $B_\C$ is not functorial on $\Cat$: a morphism $\C\to\D$ of categories (i.e., functor) will not in general induce a unital map $B_\C\to B_\D$ of algebras.  

Now we take the core of $\deloop$, which we denote $\deloop_\ei$.
\begin{proposition}\label{prop:algcat}
There exists a 2-functor $\tilde\deloop:\fBasicAlg_\ei\arrr\sim\fBaseCat$ which is a biequivalence and  which makes the following diagram commute:
\[ \xymatrix  {
\fBasicAlg_\ei\ar[d]\ar[rr]^{\tilde{\deloop}}
  && \fBaseCat\ar[d] \\
\kk\Alg_\ei \ar[r]^{\deloop_\ei} 
& \catei \ar[r]^{\Ind\ic}  & \BaseCat
} \]
\end{proposition}
\begin{proof}
Let $\biF=\Ind\ic\deloop_\ei$ denote the composite functor $\kk\Alg_\ei \to \BaseCat$.  We want to restrict $\biF$ to full sub-2-categories, so we use Lemma \ref{lem:2funfull}.  Let $A\in\ob\fBasicAlg_\ei$.  As $A$ is finite-dimensional, $\deloop_\ei A$ is hom-finite, and so $\biF A$ is also hom-finite.  Also, as $A$ is finite-dimensional, it decomposes as a finite direct sum of left ideals, and thus it has finitely many primitive 
idempotents.  So $\biF A$ has finitely many indecomposable objects.  Therefore we get the 2-functor $\tilde{\deloop}$.

We know that $\deloop_\ei$ and $\ic$ are locally fully faithful by Lemmas \ref{lem:ic-lff} and \ref{lem:deloop-leq}.  As $A$ has finitely many idempotents, we know by Lemma \ref{lem:indic} that every object of $\ic\deloop_\ei A$ is a direct sum of finitely many indecomposables.  So $\biF$ is locally fully faithful by Proposition \ref{prop:ind2fun}.
Now given $\C\in\ob\fBaseCat$, let $A=B_\C$ be its base algebra.    
Then $\biF A$ is isomorphic to $\skel\C$ by construction. 
So $\biF$ is 2-dense and therefore a biequivalence. 
\end{proof}

Now, by Propositions \ref{prop:base-add} and \ref{prop:algcat}, we have two biequivalences:
\[ \fBasicAlg_\ei \arrr\sim  \fBaseCat \arrr\sim  \fAddCat.\]

\section{Serre structures}\label{s:serre}

Our main interest is in $\kk$-linear categories, but we work with categories enriched in a pivotal monoidal category $\V$.  We define and study Serre functors in this setting.   This will be useful later when we consider nontrivial pivotal structures on the monoidal category of graded finite dimensional vector spaces.

Enriched Serre functors have been considered previously by Lyubashenko and Manzyuk in the case where $\V$ is symmetric \cite{lm}.  The setup we develop here, where $\V$ is pivotal, applies even when $\V$ does not admit a braiding.  Our approach is similar to work of Fuchs and Stigner \cite{fs} for Frobenius algebras.

\subsection{Definitions and duals}
\label{ss:defs+duals}

We recall parts of $\V$-enriched category theory, following \cite[Section 1.3]{jy}.
Throughout, we assume that
$\V$ is a pivotal monoidal category (see Section \ref{sss:pivotal}) 
but do not assume that $\V$ is braided.

Let $\C$ be a $\V$-enriched category: for all $x,y\in\C$ we have an object $\C(x,y)\in\V$, and for all $x,y,z\in\C$ we have
unit maps
\[ u: \unit\to\C(x,x) \]
and composition maps
\[ m: \C(y,z) \otimes \C(x,y)\to \C(x,z) \]
in $\V$.

If $\C$ and $\D$ are two $\V$-enriched categories, a $\V$-functor $F:\C\to\D$ is a function 
$F:\ob\C\to \ob \D$ and, for all $y,z\in\C$, a map $F_{yz}:\C(y,z)\to\D(Fy,Fz)$ in $\V$ such that, for all $x\in\C$, the following diagram commutes:
\[ \xymatrix @C=0pt {
 \C(y,z)\ar[d]^F & \otimes &\C(x,y)\ar[d]^F  \ar[rr]^{m_\C} &{\phantom{abcdefgh}} &  \C(x,z)\ar[d]^F   \\
 \D(Fy,Fz) &\otimes &\D(Fx,Fy) \ar[rr]^{m_\D}  &{\phantom{abcdefgh}} &  \D(Fx,Fz)
} \]

Recall that the pivotal structure is denoted $\iota$.  By naturality of $\iota$, 
the following diagram commutes for any $\V$-functor $F:\C\to\D$:
\[ \xymatrix 
{
\C(x,y)\ar[d]^{F}  \ar[r]^{\iota} &  \C(x,y)^{\vee\vee}\ar[d]^{F^{\vee\vee}}   \\
\D(Fx,Fy) \ar[r]^{\iota}   &  \D(Fx,Fz)^{\vee\vee}
} \]

Given $\V$-functors $F,G:\C\to\D$, a $\V$-natural transformation $\theta:F\to G$ is a collection of morphisms $\theta_y:\unit\to \D(Fy,Gy)$ such that the following diagram commutes for all $y,z\in\C$:
\[ \xymatrix @C=40pt
{
\C(y,z)\otimes\unit\cong\C(y,z)\cong \unit\otimes\C(y,z)\ar[r]^(0.57){\theta_z\otimes F_{yz}}  \ar@<-8.8ex>[d]^{G_{yz}\otimes\theta_y} &  \D(Fz,Gz)\otimes\D(Fy,Fz)
\ar[d]^{m}   \\
\D(Gy,Gz)\otimes\D(Fy,Gy) \ar[r]^(0.57){m}   &  \D(Fy,Gz)
} \]

We will use $\V$-modules for monoidal categories $\V$.  Our reference is \cite{gs}, which is written in the generality of monoidal bicategories, but the results apply immediately to our setting by considering $\V$ to be a monoidal bicategory with only identity $2$-cells.

Let $\C$ be a $\V$-enriched category.  A right $\C$-module $M$ is a function $M:\ob\C\to\ob\V$ and, for all $x,y\in\C$, a map $M_{xy}:My\otimes \C(x,y)\to Mx$ in $\V$ satisfying unitality and associativity axioms.  The standard example is $M=\C(-,z)$ for any $z\in\ob\C$, with $M_{xy}$ given by composition.  This defines a fully faithful Yoneda embedding, which acts on objects by sending $z\in\ob\C$ to the right $\C$-module $\C(-,z)$ \cite[Proposition 9.6]{gs}.

Similarly, a left $\C$-module $N$ is a function $N:\ob\C\to\ob\V$ and, for all $x,y\in\C$, a map $N_{xy}:\C(x,y)\otimes Nx\to Ny$.  The standard example is $N=\C(w,-)$ for $w\in\ob\C$.  There is also a notion of a $\C\da\D$-bimodule, which comes with a map $\ob\D\times\ob\C\to\ob\V$.  
The standard example is the hom bimodule $\C(-,-)$.

If $M$ is a right $\D$-module and $F:\C\to \D$ is a $\V$-functor, then we obtain a right $\C$-module $MF$ with object map $MF:\ob\C\to\ob\V$ given by composition and, for $x,y\in\C$, action map given by:
\[ {MF}_{xy}:MFy\otimes \C(x,y) \arrr{\id\otimes F_{xy}} MFy\otimes \C(x,y)\arrr{M_{FxFy}}MFx. \]
The same construction turns a left $\D$-module $N$ into a left $\C$-module $NF$.

If $N$ is a left $\C$-module, taking (left) duals gives a right $\C$-module $N^\vee$ with objects $N^\vee x=(Nx)^\vee$ and action map $N^\vee_{xy}=(N_{yx})^r$
(see Section \ref{ss:techlemma}).  Explicitly, we have:
\[ N^\vee_{xy}:N^\vee y\otimes \C(x,y) \arrr{(N_{yx})^\vee\otimes\id}N^\vee y\otimes \C(x,y)^\vee\otimes \C(x,y) \arrr{\id\otimes e}N^\vee x. \]
However, if $M$ is a right $\C$-module, we need to use the pivotal structure to put a left $\C$-module structure on the (left) dual $M^\vee$.  The action map is constructed using $\iota$ and $(M_{yx})^\ell$ as follows:
\[ M^\vee_{xy}: \C(x,y)\otimes M^\vee y \arrr{\iota\otimes(M_{yx})^\vee} \C(x,y)^{\vee\vee}\otimes \C(x,y)^\vee\otimes M^\vee y  \arrr{e\otimes \id}M^\vee x. \]

Now, following the classic definition of Bondal and Kapranov \cite{bk}, we are able to give our main definition.
\begin{definition}\label{def:serre}
A \emph{$\V$-enriched Serre structure} for $\C$ is a pair $(\se,\kappa)$ where $\se:\C\arr\sim\C$ is a $\V$-autoequivalence and 
$\kappa$ is a $\C\da\C$-bimodule isomorphism $\kappa_{x,y}: \C(x,y) \to \C(y,\se x)^\vee $.

If $(\se,\kappa)$ is a $\V$-enriched Serre structure for $\C$, we say that
$\C$ \emph{has Serre duality} and that
$\se$ is a \emph{Serre functor}. 
\end{definition}

The above definition says that, for all $x,y\in\ob\C$ we have an invertible map $\kappa_{x,y}$ which is both a left $\C$-module morphism $\kappa_{x,-}: \C(x,-) \to \C(-,\se x)^\vee$ and a right $\C$-module morphism $\kappa_{-,y}: \C(-,y) \to \C(y,\se -)^\vee$ (see \cite[Section 5.4]{gs}.  
Following Section \ref{ss:techlemma}, the composition map of $\V$ determines two maps
\[ m^\ell: \C(x,y)^{\vee\vee} \otimes  \C(x,z)^\vee\to \C(y,z)^\vee \]
and
\[ m^r:  \C(x,z)^\vee\otimes \C(y,z)\to \C(x,y)^\vee \]
in $\V$.  
Unravelling the definition, the bimodule morphism condition says that two diagrams commute.  
Let $w,x,y,z\in\ob\C$.  
The left module morphism condition gives ``left naturality'' of $\kappa$:
\[ \xymatrix @C=0pt {
 \C(y,z)\ar[d]^{\iota} & \otimes &\C(x,y)\ar[d]^{\kappa_{x,y}} {\phantom{ab}} \ar[rr]^m &{\phantom{abcdefgh}} & {\phantom{ab}} \C(x,z)\ar[d]^{\kappa_{x,z}}   \\
 \C(y,z)^{\vee\vee} &\otimes &\C(y,\se x)^\vee {\phantom{ab}}\ar[rr]^{m^\ell}  & & {\phantom{ab}}\C(z,\se x)^\vee 
} \]
and the right module condition gives ``right naturality'' of $\kappa$:
\[
\xymatrix @C=0pt {
\C(x,y)\ar[d]^{\kappa_{x,y}} &\otimes &\C(w,x){\phantom{ab}} \ar[d]^{\se} \ar[rr]^m &{\phantom{abcdefgh}} & {\phantom{ab}} \C(w,y)\ar[d]^{\kappa_{w,y}}   \\
\C(y,\se x)^\vee &\otimes &\C(\se w,\se x){\phantom{ab}}\ar[rr]^{m^r}  & & {\phantom{ab}}\C(y,\se w)^\vee 
} \]
In formulas, these say
\[ \kappa_{x,z}(hg)=\iota(h)\kappa_{x,y}(f) \;\;\;\; \text{ and } \;\;\;\; \kappa_{x,y}(gf)=\kappa_{w,y}(g)\se(f) \]
for $f:w\to x$, $g:x\to y$, and $h:y\to z$.

\begin{remark}\label{rmk:pairing}
One can also specify Serre structures using trace maps or pairings: see, for example, \cite[Proposition I.1.4]{rvdb}.  We show how this matches our conventions, but we won't pursue it further.  Given a Serre structure as in Definition \ref{def:serre}, we obtain a pairing by:
\[ \C(y,\se x)\otimes\C(x,y)\arrr{\iota\otimes\kappa} \C(y,\se x)^{\vee\vee}\otimes \C(y,\se x)^{\vee}\arrr{e}\unit \]
and given a trace map $\C(x,\se x)$ we construct a pairing $p_{x,y}:\C(y,\se x)\otimes \C(x,y)\to\unit$ by precomposing with the composition in $\C$, then construct the following map:
\[ \C(y,\se x)\arrr{\id\otimes c} \C(y,\se x)\otimes \C(x,y)\otimes \C(x,y)^\vee\arrr{p\otimes\id}\C(x,y)^\vee. \]
To obtain $\kappa$, we dualize and precompose with the pivotal structure map:
\[ \C(x,y)\arrr{\iota} \C(x,y)^{\vee\vee}\arrr{} \C(y,\se x)^\vee. \]
\end{remark}

\begin{remark}
Traditionally, the definition of a Serre functor involves a binatural transformation of functors, not a morphism of bimodules.  It may comfort the reader(/author) to note that modules do induce $\V$-functors, by the following construction.  
As $\V$ is rigid, it is left and right closed, i.e., it has left and right internal homs
\[ [Y,Z]_\ell={^\vee Y} \otimes Z \in\V \;\;\;\; \text{ and } \;\;\;\; [Y,Z]_\ell=Z\otimes Y^\vee\in\V \]
where ${^\vee Y}$ denotes the right dual of $Y$. %, 
These satisfy the adjunction formulas
\[ \V(Y\otimes X,Z)\cong \V(X,[Y,Z]_\ell) \;\;\;\; \text{ and } \;\;\;\; \V(X\otimes Y,Z)\cong\V(X,[Y,Z]_r). \]
Using these, we define two $\V$-enriched categories.

The first, called $\V^r$, has hom object $\V^r(X,Y)=[X,Y]_r$, unit map $\unit\to[X,X]_r$ adjoint to $\unit\otimes X\arr\sim X$, and multiplication map 
$$[Y,Z]_r\otimes[X,Y]_r =Z \otimes Y^\vee \otimes Y \otimes X^\vee \arrr{\id\otimes e\otimes \id} Z\otimes X^\vee=[X,Z]_r.$$
So, given a left $\C$-module $N$, we have a map $N_{xy}:\C(x,y)\otimes Nx\to Ny$ which, by the second adjunction formula, determines a map $\C(x,y)\to [Nx, Ny]_r$ in $\V$ and thus determines a $\V$-functor $N:\C\to\V^r$.

The second $\V$-enriched category is suggestively called $(\V^\ell)^\rev$, with hom object $\V^\ell(X,Y)=[Y,X]_\ell$ (note the reversed order) and multiplication map 
$$[Z,Y]_\ell\otimes[Y,X]_\ell ={^\vee Z} \otimes Y \otimes {^\vee Y} \otimes X  \arrr{\id\otimes e\otimes \id}{^\vee Z} \otimes X=[Z,X]_\ell.$$
Then, given a right $\C$-module $M$, we have a map $M_{xy}:My\otimes \C(x,y)\to Mx$ which, by the first adjunction formula, determines a map $\C(x,y)\to [My, Mx]_r$ in $\V$ and thus determines a $\V$-functor $M:\C\to(\V^\ell)^\rev$.
\end{remark}

\begin{remark}
Our definition of a Serre structure uses enriched categories with dualizable hom-spaces.  In the case of vector spaces, this means hom-finite categories.  This may seem restrictive, as the classical definition of a Serre functor uses arbitrary $\kk$-linear categories, but 
in fact the finite dimensional condition is forced: $\kappa$ and $(\kappa^*)^{-1}$ give an isomorphism between $\C(x,y)$ and $\C(\se x,\se y)^{**}$, and $\se$ is an autoequivalence, so we have a $\kk$-linear isomorphism between $\C(x,y)$ 
 and its double dual.  Thus any $\kk$-linear category which has Serre duality must be hom-finite.
\end{remark}

\subsection{Uniqueness}\label{ss:uniqueserre}

It is well-known that Serre functors are unique up to isomorphism \cite[Proposition 3.4(b)]{bk}:
\begin{proposition}\label{prop:uniqueserre}
If $(\se,\kappa)$ and $(\se',\kappa')$ are two Serre structures on $\C$ then there exists a natural isomorphism $\alpha:\se\arr\sim\se'$ such that the following diagram commutes:
\[ \xymatrix{
\C(x,y) \ar[r]^{\kappa_{x,y}}\ar[d]_{\kappa_{x,y}'} &\C(y,\se x)^\vee \\
\C(y,\se' x)^\vee\ar[ur]_{\C(y,\alpha_x)^\vee}  
} \]
\end{proposition}
We can extend this idea.  Let $\C_i$ be a $\V$-category with Serre structure $(\se_i,\kappa_i)$, for $i=1,2$.
\begin{definition}\label{def:mapofss}
A \emph{morphism of Serre structures} $(\C_1,\se_1,\kappa_1)\to (\C_2,\se_2,\kappa_2)$ is a pair $(F,\alpha)$ where $F:\C_1\to\C_2$ is a $\V$-functor and $\alpha:F\se_1\arr\sim\se_2 F$ is a natural isomorphism such that the following diagram commutes:
\[\xymatrix @C=0pt {
{\phantom{\C_2( F y,\se_2  F x)^\vee}}& {\C_1(x,y)\ar[rr]^{\kappa_1}} \ar[dl]_{ F} &{\phantom{\C_2( F y,\se_2  F x)^\vee}}& {\C_1(y,\se_1 x)^\vee} &{\phantom{\C_2( F y,\se_2  F x)^\vee}} \\
{\C_2( F x, F y)} \ar[rr]_{\kappa_2} &{\phantom{\C_2( F y,\se_2  F x)^\vee}}& {\C_2( F y,\se_2  F x)^\vee}\ar[rr]_{\C_2(Fy,\alpha_x)^\vee} &{\phantom{\C_2( F y,\se_2  F x)^\vee}}& {\C_2( F y,  F\se_1 x)^\vee} \ar[ul]_{F^\vee}
}  \]
If $F$ is an equivalence then we say $(F,\alpha)$ is an \emph{equivalence of Serre structures}. 
\end{definition}
So Proposition \ref{prop:uniqueserre} says that there exists an isomorphism $\alpha:\se\arr\sim\se'$ such that $(\id_\C,\alpha)$ is an equivalence of Serre structures.

\begin{proposition}\label{prop:transport-serre}
Let $\C_1$ and $\C_2$ be $\V$-categories.  
\begin{enumerate}[label=(\alph*)]
\item Suppose $\C_1$ has a Serre structure $(\se_1,\kappa_1)$ and we have an equivalence of $\V$-categories $F:\C_1\arr\sim\C_2$.  Then
there exists a Serre structure $(\se_2,\kappa_2)$ on $\C_2$.  
\item Given Serre structures $(\se_i,\kappa_i)$ on $\C_i$ and an equivalence of $\V$-categories $F:\C_1\arr\sim\C_2$, 
$F$ can be upgraded to an equivalence of Serre structures $(F,\alpha)$.
\end{enumerate}
\end{proposition}
\begin{proof}
\begin{enumerate}[label=(\alph*)]
\item Fix an adjoint equivalence $(F,G,\eta, \varepsilon)$.  
Define $\se_2=F\se_1 G:\D\to\D$.  Then define:
\[ \kappa_2:
 \C_2(x,y)\arrr G \C_1(Gx,Gy)\arrr{\kappa_1}
 \C_1(Gy, \se_1 Gx)^\vee
\arrr{(F^{-1})^\vee} 
\C_2(FGy,F\se_1 Gx)^\vee\arrr{(\varepsilon,\id)^\vee} \C_2(y,\se_2 x)^\vee. \]
This is a binatural isomorphism by the functoriality of $F$ and $G$ and the naturality of $\kappa_1$ and $\varepsilon$.
Note that, by Proposition \ref{prop:uniqueserre}, $(\se_2,\kappa_2)$ is unique up to isomorphism.

\item
We have four invertible morphisms in the diagram of Definition \ref{def:mapofss}, so they define a map 
${\C_2( F y,\se_2  F x)^\vee}\to {\C_2( F y,  F\se_1 x)^\vee}$ natural in $x$ and $y$.  By the Yoneda lemma, this determines a natural transformation $\alpha:F\se_1\arr\sim\se_2 F$. \qedhere
\end{enumerate}
\end{proof}

We record two useful results.  The first is a direct consequence of Proposition \ref{prop:transport-serre}.
\begin{corollary}\label{cor:skel-serre}
Let $\B$ be a skeletal subcategory of $\C$. Then
$\B$ has a Serre structure if and only if $\C$ does.
\end{corollary}

\begin{lemma}\label{lem:serreonsubcat}
If $\C$ has a Serre structure $(\se,\kappa)$ and $\B$ is a full subcategory of $\C$ which is closed under $\se$ (i.e., $\se \B\subseteq \B$) then $\B$ has a Serre structure.
\end{lemma}
\begin{proof}
This is clear: $(\se,\kappa)$ restricts to $\B$.
\end{proof}
The assumption $\se \B\subseteq \B$ in Lemma \ref{lem:serreonsubcat} is necessary: consider a $\kk$-category $\C$ with object set $\Z/3\Z$ and $\C(n,n+1)=\kk$ but $\C(n,n-1)=0$.  Then $\C$ has a Serre structure but the full subcategory on the objects $0$ and $1$ does not.

We now put categories with Serre structures into a 2-categorical setting.
\begin{definition}
If $(F,\alpha)$ and $(G,\beta)$ are two morphisms of Serre structures $(\C_1,\se_1,\kappa_1)\to (\C_2,\se_2,\kappa_2)$, a \emph{transformation} $(F,\alpha)\to(G,\beta)$ is a natural transformation $\nu:F\to G$ such that $\se_2\nu\circ\alpha=\beta\circ\nu \se_1$.
\end{definition}
\begin{remark}\label{rmk:automcomm}
If the natural transformation $\nu:F\to G$ is an isomorphism then the condition $\se_2\nu\circ\alpha=\beta\circ\nu \se_1$ is automatically satisfied: see Proposition \ref{prop:serre-centre} below.
\end{remark}
Consider the following 2-category of $\V$-categories with Serre structures, denoted $\SCat$:
\begin{itemize}
\item The objects are triples $(\C,\se,\kappa)$, where $\C$ is a $\V$-category and $(\se,\kappa)$ is a Serre structure on $\C$.
\item The 1-cells $(F,\alpha):(\C_1,\se_1,\kappa_1)\to (\C_2,\se_2,\kappa_2)$ are morphisms of Serre structures.
\item The 2-cells $\nu:(F,\alpha)\to(G,\beta)$ are transformations of morphisms of Serre structures.
\end{itemize}
We have a 2-functor $\twofun{U}:\SCat\to\V\Cat$ which sends $(\C,\se,\kappa)$, $(F,\alpha)$, and $\nu$ to $\C$, $F$, and $\nu$, respectively.
\begin{proposition}\label{prop:forgetserre}
The 2-functor
\[ \core(\twofun{U}):\SCat_\ei\to\V\Cat_\ei \]
is locally an equivalence.
\end{proposition}
\begin{proof}
Faithfulness is clear: if we have $\nu,\nu':(F,\alpha)\to(G,\beta)$ and $\core(\twofun{U})(\nu)=\core(\twofun{U})(\nu')$ then $\nu=\core(\twofun{U})(\nu)=\core(\twofun{U})(\nu')=\nu'$.
Fullness on equivalences follows from Remark \ref{rmk:automcomm}.  
Essential surjectivity comes from Proposition \ref{prop:transport-serre}.
\end{proof}
Therefore we have the following result, which loosely says that we can regard a Serre structure as a property rather than a structure:
\begin{corollary}\label{cor:bieqserre2gpds}
The 2-groupoid $\SCat_\ei$ of $\V$-categories with Serre structures is biequivalent to the full sub-2-groupoid of $\V\Cat_\ei$ on the $\V$-categories which admit a Serre structure.
\end{corollary}

\subsection{The compatibility lemma}

In their original definition, Bondal and Kapranov included an extra compatibility condition which was later shown to follow from the naturality of $\kappa$: see \cite[Lemma I.1.1]{rvdb}.  
We state and prove this in the pivotal enriched setting.

\begin{lemma}\label{lem:serretwice}
The following diagram commutes:
\[
\xymatrix{
(x,y) \ar[r]^{\kappa_{x,y}}\ar[d]^\se & (y,\se x)^\vee \ar[d]^{(\kappa^{-1}_{y,\se x})^{\vee}} \\
(\se x,\se y) \ar[r]^\iota & (\se x,\se y)^{\vee\vee}
}\]
\end{lemma}
\begin{proof}
From the right naturality of $\kappa$ we get:
\begin{equation}\label{cl:box1} 
\xymatrix @C=0pt {
\C(y,y)\ar[d]^{\kappa_{y,y}} &\otimes &\C(x,y){\phantom{ab}} \ar[d]^\se \ar[rr]^m &{\phantom{abcdefgh}} & {\phantom{ab}} \C(x,y)\ar[d]^{\kappa_{x,y}}   \\
\C(y,\se y)^\vee &\otimes &\C(\se x,\se y){\phantom{ab}}\ar[rr]^{m^r}  & & {\phantom{ab}}\C(y,\se x)^\vee 
} 
\end{equation}
Then by the naturality and monoidality of $\iota$ we get:
\begin{equation}\label{cl:box2} 
\xymatrix @C=0pt {
\C(y,\se y)^\vee \ar[d]^{\iota} &\otimes &\C(\se x,\se y){\phantom{ab}}\ar[rrr]^{m^r} \ar[d]^{\iota}  & & & {\phantom{ab}}\C(y,\se x)^\vee \ar[d]^{\iota} \\
\C(y,\se y)^{\vee\vee\vee} &\otimes &\C(\se x,\se y)^{\vee\vee}{\phantom{ab}} \ar[rrr]^{(m^r)^{\vee\vee}}  & & & {\phantom{ab}}\C(y,\se x)^{\vee\vee\vee}
}
\end{equation}

By the dual of Lemma \ref{lem:tech1}, the following diagram commutes:
\begin{equation}\label{cl:part1}
 \xymatrix @C=0pt {
\C(y,\se y)^{\vee\vee\vee} \otimes \C(\se x,\se y)^{\vee\vee}{\phantom{ab}}  \ar[dr]^{\id\otimes (m^\ell)^\vee} \ar[rr]^{(m^r)^{\vee\vee}}  &&  {\phantom{ab}}\C(y,\se x)^{\vee\vee\vee} \\
& \C(y,\se y)^{\vee\vee\vee} \otimes  \C(y,\se y)^{\vee\vee} \otimes \C(y,\se x)^{\vee\vee\vee} \ar[ur]^{e\otimes\id}&
} \end{equation}
Now from left naturality of $\kappa$ we get:
\[ \xymatrix @C=0pt {
 \C(y,\se x)\ar[d]^{\iota} & \otimes &\C(y,y)\ar[d]^{\kappa_{y,y}} {\phantom{ab}} \ar[rr]^m &{\phantom{abcdefgh}} & {\phantom{ab}} \C(y,\se x)\ar[d]^{\kappa_{y,\se x}}   \\
 \C(y,\se x)^{\vee\vee} &\otimes &\C(y,\se y)^\vee {\phantom{ab}}\ar[rr]^{m^\ell}  & & {\phantom{ab}}\C(\se x,\se y)^\vee 
} \]
and dualizing this gives:
\[
 \xymatrix @C=0pt {
\C(\se x,\se y)^{\vee\vee}{\phantom{ab}} \ar[d]^{\kappa^\vee}  \ar[rrr]^{(m^\ell)^\vee}  & & & {\phantom{ab}}\C(y,\se y)^{\vee\vee} \ar[d]^{\kappa^\vee} &\otimes & \C(y,\se x)^{\vee\vee\vee} \ar[d]^{\iota^\vee} \\ 
  \C(y,\se x)^\vee{\phantom{ab}}  \ar[rrr]^{m^\vee} & & &  {\phantom{ab}}\C(y,y)^\vee &  \otimes & \C(y,\se x)^\vee
} \]
Tensoring with $(\kappa^{\vee\vee})^{-1}$, we get:
\begin{equation}\label{cl:part2}
 \xymatrix @C=0pt {
\C(y,\se y)^{\vee\vee\vee} \ar[d]^{(\kappa^{\vee\vee})^{-1}} &\otimes& \C(\se x,\se y)^{\vee\vee}{\phantom{ab}} \ar[d]^{\kappa^\vee}  \ar[rrr]^{\id\otimes(m^\ell)^\vee}  & & & {\phantom{ab}}\C(y,\se y)^{\vee\vee\vee} \ar[d]^{(\kappa^{\vee\vee})^{-1}} &\otimes&  \C(y,\se y)^{\vee\vee} \ar[d]^{\kappa^\vee} &\otimes & \C(y,\se x)^{\vee\vee\vee} \ar[d]^{\iota^\vee} \\ 
\C(y,y)^{\vee\vee} &\otimes&   \C(y,\se x)^\vee{\phantom{ab}}  \ar[rrr]^{\id\otimes m^\vee} & & &  {\phantom{ab}}\C(y,y)^{\vee\vee} &\otimes&  \C(y,y)^\vee &  \otimes & \C(y,\se x)^\vee
} \end{equation}

By definition of $m^\ell$ (see Section \ref{ss:techlemma}) we have:
\begin{equation}\label{cl:part3}
\xymatrix{
\C(y,y)^{\vee\vee} \otimes   \C(y,\se x)^\vee \ar[rr]^{m^\ell} \ar[dr]_{1\otimes m^\vee} && \C(y,\se x)^\vee \\
& \C(y,y)^{\vee\vee} \otimes  \C(y,y)^{\vee} \otimes   \C(y,\se x)^\vee \ar[ur]_{e\otimes 1} & 
}
\end{equation}

Lemma \ref{lem:coevisom} says that the following diagram commutes:
\[
 \xymatrix @C=0pt {
{\phantom{ab}}\unit{\phantom{ab}} \ar[d]^{\id}  \ar[rrrrrr]^{c}  & & & & & & {\phantom{ab}}\C(y,y) \ar[d]^{\kappa} &\otimes & \C(y,y)^{\vee} \ar[d]^{(\kappa^\vee)^{-1}} \\ 
{\phantom{ab}}\unit{\phantom{ab}}  \ar[rrrrrr]^{c} & & & & & &  {\phantom{ab}}\C(y,\se y)^\vee &  \otimes & \C(y,\se xy)^{\vee\vee}
} \]
Therefore, by tensoring with $\iota$, dualising, and using Lemma \ref{lem:dualevcoev}, the following diagram commutes:
\begin{equation}\label{cl:part4}
 \xymatrix @C=0pt {
\C(y,\se y)^{\vee\vee\vee} \ar[d]^{(\kappa^{\vee\vee})^{-1}} &\otimes&  \C(y,\se y)^{\vee\vee} \ar[d]^{\kappa^\vee} &\otimes & \C(y,\se x)^{\vee\vee\vee} \ar[d]^{\iota^\vee}  \ar[rr]^{e\otimes\id} &{\phantom{abcdefgh}} & {\phantom{ab}} \C(y,\se x)\ar[d]^{\iota^\vee}  \\ 
\C(y,y)^{\vee\vee} &\otimes&  \C(y,y)^\vee &  \otimes & \C(y,\se x)^\vee  \ar[rr]^{e\otimes\id}  & & {\phantom{ab}}\C(\se x,\se y)^\vee 
} \end{equation}

Combining diagrams \eqref{cl:part1}, \eqref{cl:part2}, \eqref{cl:part3}, and \eqref{cl:part4} gives:
\begin{equation}\label{cl:box3} 
\xymatrix @C=0pt {
\C(y,\se y)^{\vee\vee\vee} \ar[d]^{(\kappa^{\vee\vee})^{-1}} &\otimes & \C(\se x,\se y)^{\vee\vee}{\phantom{ab}} \ar[d]^{\kappa^\vee}  \ar[rrr]^{(m^r)^{\vee\vee}}  & & & {\phantom{ab}}\C(y,\se x)^{\vee\vee\vee} \ar[d]^{\iota^\vee} \\ 
\C(y,y)^{\vee\vee} &  \otimes &  \C(y,\se x)^\vee{\phantom{ab}}  \ar[rrr]^{m^\ell} & & &  {\phantom{ab}} \C(y,\se x)^\vee
} \end{equation}

Putting \eqref{cl:box1}, \eqref{cl:box2}, and \eqref{cl:box3} together gives the following commutative diagram:
\[ \xymatrix @C=0pt {
\C(y,y)\ar[d]^{\kappa} \ar@/_3pc/[ddd]_{\iota} &\otimes &\C(x,y){\phantom{ab}} \ar[d]^\se \ar[rr]^m &{\phantom{abcdefgh}} & {\phantom{ab}} \C(x,y)\ar[d]^{\kappa}   \\
\C(y,\se y)^\vee \ar[d]^{\iota} &\otimes &\C(\se x,\se y){\phantom{ab}}\ar[rr]^{m^r} \ar[d]^{\iota}  & & {\phantom{ab}}\C(y,\se x)^\vee \ar[d]^{\iota}\ar@/^3pc/[dd]^{\id} \\
\C(y,\se y)^{\vee\vee\vee} \ar[d]^{(\kappa^{\vee\vee})^{-1}} &\otimes & \C(\se x,\se y)^{\vee\vee}{\phantom{ab}} \ar[d]^{\kappa^\vee}  \ar[rr]^{(m^r)^{\vee\vee}}  & & {\phantom{ab}}\C(y,\se x)^{\vee\vee\vee} \ar[d]^{\iota^\vee} \\ 
\C(y,y)^{\vee\vee} &  \otimes &  \C(y,\se x)^\vee{\phantom{ab}}  \ar[rr]^{m^\ell} & &  {\phantom{ab}} \C(y,\se x)^\vee
} \]
The partial diagram on the left involving the curved arrow commutes by the naturality of $\iota$.  The one on the right commutes by Lemma \ref{lem:saav}. 

Finally, taking $\id_y\in\C(y,y)$ and using unitality properties gives
\[
\xymatrix{
\C(x,y) \ar[r]^{\kappa_{x,y}}\ar[d]^\se & \C(y,\se x)^\vee  \\
(\se x,\se y) \ar[r]^\iota & (\se x,\se y)^{\vee\vee} \ar[u]_{\kappa_{y,\se x}^\vee}
}\]
as desired.
\end{proof}

\subsection{Commutation}

Bondal and Orlov showed that Serre functors commute with autoequivalences \cite[Proposition 1.3]{bo}.  We give a proof of this fact, as it is both short and instructive.
\begin{proposition}\label{prop:serre-comm-autoeq}
If $F:\C\arr\sim\C$ is an autoequivalence and $\C$ has Serre structure $(\se,\kappa)$, then there is a canonical natural isomorphism $\se F\arr\sim F\se$.  
\end{proposition}
\begin{proof}
Assume $F$ has quasi-inverse $G$.
Then we have isomorphisms:
\[ \C(x,\se Fy)^\vee \arrr{\kappa_{Fy,x}^{-1}} \C(Fy,x)\arrr{\text{adj}} \C(y,Gx)\arrr{\kappa_{y,Gx}} \C(Gx,\se y)^{\vee}
\arrr{\text{adj}^\vee} \C(x,F\se y)^\vee \]
for any $x,y$.  Dualizing and using the pivotal structure gives an isomorphism $\C(x,\se Fy)\arr\sim \C(x,F\se y)$.
As the Yoneda embedding is fully faithful this determines the natural isomorphism $\se F\arr\sim F\se$.  
\end{proof}
Given $F:\C\arr\sim \C$, we write $\zeta_F:\se F\arr\sim F\se$ for the above natural transformation.

The following lemma is immediate from the proof of Proposition \ref{prop:serre-comm-autoeq}.
\begin{lemma}\label{lem:zetaunit}
$\zeta_{\id_\C}=\id_{\se}$.
\end{lemma}

Now we give a stronger version of Proposition \ref{prop:serre-comm-autoeq}.
\begin{proposition}\label{prop:serre-centre}
If $\C$ has a Serre structure $(\se,\kappa)$ then $(\se,\zeta)$ belongs to the Drinfeld centre of the monoidal category of autoequivalences of $\C$.
\end{proposition}
\begin{proof}
From Definition \ref{def:centre} we must check two properties.  
First we check naturality of $\zeta$.  Given two autoequivalences $F,G:\C\arr\sim \C$ and a natural transformation $\alpha:F\to G$,
we must show that the following diagram commutes:
\[\xymatrix{
\se F \ar[r]^{\zeta_{F}}\ar[d]_{\se\alpha} & F\se\ar[d]^{\alpha\se} \\
\se G\ar[r]^{\zeta_G} & G\se
}\]
Let $F^!$ and $G^!$ denote the left adjoints of $F$ and $G$, respectively.  Then 
$\alpha$ induces a natural transformation
\[ \alpha^!:G^!\arrr{G!^\eta} G^!FF^!\arrr{G^!\alpha F^!} G^!GF^! \arrr{\varepsilon F^!} F^!. \]
We break our diagram up into smaller squares:
\[\xymatrix{
(x,\se Fy) \ar[r]\ar[d] & (Fy,x)^\vee \ar[r]\ar[d] & (y, F^!x)^\vee \ar[r]\ar[d]  & (F^!x, \se y) \ar[r]\ar[d] & (x, F\se y) \ar[d] \\
(x,\se Gy) \ar[r] & (Gy,x)^\vee \ar[r] & (y, G^!x)^\vee \ar[r]  & (G^!x, \se y) \ar[r] & (x, G\se y) 
}\]
The first and third commute by naturality of $\kappa$ (from the Serre structure), and the second and fourth commute by an exercise in adjunctions and the fact that $F^{!!}\cong F$ naturally (as $F$ is an equivalence).

Next, we must check that the following diagram commutes:
\[\xymatrix{
\se GF \ar[rr]^{\zeta_{GF}}\ar[dr]_{\zeta_G F} && GF\se \\
&G\se F\ar[ur]_{G\zeta_F} &
}\]

By the Yoneda lemma, it is enough to check that the following diagram commutes:
\[\xymatrix{
(x,\se GFy) \ar[rr]\ar[dr] && (x,GF\se y) \\
&(x,G\se Fy)\ar[ur] &
}\]
First we note that the following diagram commutes by naturality of adjunctions:
\[ \xymatrix @C=40pt {
(x,G\se Fy) \ar[r]^{(x,G\zeta_{Fy})} \ar[d]^{\sim}  & (x, GF\se y) \ar[d]^{\sim} \\
(G^!x, \se Fy) \ar[r]^{(G^!x,\zeta_{Fy})}  & (G^!x, F\se y)
}\]
Therefore we just need the following diagram to commute:
\[\xymatrix{
(x,\se GFy) \ar[r]\ar[d] & (GFy,x)^\vee \ar[r] & (y, F^!G^!x)^\vee \ar[r]  \ar@{--}[dddr] & (F^!G^!x, \se y) \ar[r]  \ar@{--}[ddr] & (x,GF\se y)\ar[dr] & \\
(GFy,x)^\vee\ar[d]  \ar@{--}[ur] &&&&&(G^!x, F\se y) \\
(Fy, G^!x)^\vee \ar[d] \ar@{--}[ddrr] &&&& (F^!G^!x, \se y)\ar[ur]  \\
(G^!x, \se Fy) \ar[d] \ar@{--}[ddr] &&& (y, F^!G^!x)^\vee\ar[ur] \\  
(x,G\se Fy)\ar[uuuurrrr]\ar[dr] && (Fy, G^!x)^\vee\ar[ur] \\
& (G^!x, \se Fy)\ar[ur]
}\]
This is easy to see by considering the equalities indicated by dashed lines.
\end{proof}

Now fix a quasi-inverse $\se^-$ and maps
\[
\varepsilon_1:\se^-\se\to\id_\C\; \;\;\;
\varepsilon_2:\se\se^-\to\id_\C\; \;\;\;
\eta_1:\id_\C\to\se^-\se\; \;\;\;
\eta_2:\id_\C\to\se\se^-
\]
of an adjoint equivalence, so $\varepsilon_1\eta_1=\id_{\id_\C}=\varepsilon_2\eta_2$ and the pairs $(\eta_1,\varepsilon_2)$ and $(\eta_2,\varepsilon_1)$ are both unit-counit pairs of adjunctions.

\begin{lemma}\label{lem:commutewithqinv}
The natural transformation $\zeta_{\se^-}$ is just the composition
\[ \se\se^- \arrr{\varepsilon_2} \id_\C \arrr{\eta_1} \se^-\se. \]
\end{lemma}
\begin{proof}
In fact we'll prove that $\zeta_{\se^-}^{-1}=\eta_2\varepsilon_1$, which is equivalent.  The natural transformation is represented by
\[ \xymatrix @C=5pt {
(x,\se^-\se y) \ar[dr]_{\se} && (\se x, \se y) \ar[dr]^{\iota} && (y,\se x)^\vee \ar[dr]^{((\eta_2)_y,\se x)^\vee} && (\se^-y,x)^\vee \ar[dr]^{(\kappa^{-1})^\vee} && (x,\se\se^- y) \\
& \xyc{(\se x, \se\se^-\se y)} \ar[ur]^{(\se x,(\varepsilon_2)_{\se y})} && \xyc{(\se x,\se y)^{\vee\vee}} \ar[ur]^{\kappa^\vee} && \xyc{(\se\se^-y,\se x)^\vee} \ar[ur]_{\se^\vee} && \xyc{(x,\se\se^-y)^{\vee\vee}} \ar[ur]_{\iota^{-1}} &
}\]
By Lemma \ref{lem:serretwice} and by taking duals, the following squares commute:
\[
\xymatrix{
(\se^-y,x) \ar[r]^{\kappa_{\se^-y,x}}\ar[d]^\se & (x,\se\se^-y)^\vee \ar[d]^{(\kappa^{-1}_{x,\se\se^-y})^\vee} \\
(\se\se^-y,\se x) \ar[r]^\iota & (\se\se^-y,\se x)^{\vee\vee}
} \;\;  \text{ and } \;\;
\xymatrix{
(\se^-y,x)^\vee & (x,\se\se^-y)^{\vee\vee}  \ar[l]_{\kappa_{\se^-y,x}^\vee} \\
(\se\se^-y,\se x)^\vee \ar[u]_{\se^\vee} & (\se\se^-y,\se x)^{\vee\vee\vee} \ar[l]_{\iota^\vee}   \ar[u]_{(\kappa_{x,\se\se^-y}^{-1})^{\vee\vee}} 
}\]
So the following diagram is commutative:
\[ \xymatrix @C=5pt {
(x,\se^-\se y) \ar[dr]_{\se} && (\se x, \se y) \ar[dr]^{\iota} && (y,\se x)^\vee \ar[dr]^{((\eta_2)_y,\se x)^\vee} && (\se^-y,x)^\vee \ar[dr]^{(\kappa^{-1})^\vee} && (x,\se\se^- y) \\
& \xyc{(\se x, \se\se^-\se y)} \ar[ur]^{(\se x,(\varepsilon_2)_{\se y})} && \xyc{(\se x,\se y)^{\vee\vee}} \ar[ur]^{\kappa^\vee} && \xyc{(\se\se^-y,\se x)^\vee} \ar[ur]_{\se^\vee} && \xyc{(x,\se\se^-y)^{\vee\vee}} \ar[ur]_{\iota^{-1}} & \\
&&&&&& \xyc{(\se \se^-y,\se x)^{\vee\vee\vee}} \ar[ul]^{\iota^\vee} \ar[ur]^{\kappa^{\vee\vee}_{x,\se\se^-y}} & 
}\]
Then by naturality of $\iota$ we have
\[ \xymatrix @C=40pt {
 (x,\se\se^-y) \ar[r]^{\kappa_{x,\se\se^-y}} \ar[d]^\iota & (\se \se^-y,\se x)^{\vee} \ar[d]^\iota \\
 (x,\se\se^-y)^{\vee\vee}  \ar[r]^{\kappa^{\vee\vee}_{x,\se\se^-y}} & (\se \se^-y,\se x)^{\vee\vee\vee} 
} \]
so we can replace our diagram with
\[ \xymatrix @C=0pt {
(x,\se^-\se y) \ar[dr]^{\se} && (\se x, \se y) \ar[dr]^{\iota} && (y,\se x)^\vee \ar[dr]^{((\eta_2)_y,\se x)^\vee}  && (x,\se\se^- y) \\
& (\se x, \se\se^-\se y) \ar[ur]^{(\se x,(\varepsilon_2)_{\se y})} && (\se x,\se y)^{\vee\vee} \ar[ur]^{\kappa^\vee} && (\se\se^-y,\se x)^\vee \ar[ur]_{\kappa_{x,\se\se^-y}^{-1}} &
}\]
Now use naturality of $\kappa$, and Lemma \ref{lem:serretwice} again, to get:
\[ \xymatrix @C=0pt {
(x,\se^-\se y) \ar[dr]^{\se} && (\se x, \se y) \ar[dr]^{\iota} && (y,\se x)^\vee \ar[dr]^{((\eta_2)_y,\se x)^\vee}  \ar@/_/[ddr]_{\kappa_{x,y}^{-1}} && (x,\se\se^- y) \\
& (\se x, \se\se^-\se y) \ar[ur]^{(\se x,(\varepsilon_2)_{\se y})} && (\se x,\se y)^{\vee\vee}  \ar[d]^{\iota^{-1}} \ar[ur]^{\kappa^\vee} && (\se\se^-y,\se x)^\vee \ar[ur]_{\kappa^{-1}} 
 & \\
&&& (\se x,\se y) & & (x,y) \ar@/_/[uur]_{(x,(\eta_2)_y)} \ar[ll]^\se
}\]
Tidy this up to get:
\[ \xymatrix @C=0pt {
(x,\se^-\se y) \ar[dr]^{\se} && (\se x, \se y) &&      (x,\se\se^- y) \\
& (\se x, \se\se^-\se y) \ar[ur]^{(\se x,(\varepsilon_2)_{\se y})} && (x,y) \ar[ur]_{(x,(\eta_2)_y)} \ar[ul]^\se 
 &
}\]
As $(\eta_2,\varepsilon_1)$ is a unit-counit adjunction pair, the following square commutes:
\[ \xymatrix  {
 (x,y) \ar[r]^{\se} \ar[d]^{(\se x,(\eta_1)_{y}} & (\se x,\se y)  \\
 (x,\se\se^-y)   \ar[r]^{\se} & (\se x,\se \se^-\se y) \ar[u]_{(\se x,(\varepsilon_2)_{\se x})}
} \]
so, using $\varepsilon_1=\eta_1^{-1}$, we get the diagram
\[ \xymatrix @C=0pt {
(x,\se^-\se y) \ar[dr]_{(x,(\varepsilon_1)_y)} &&       (x,\se\se^- y) \\
&  (x,y) \ar[ur]_{(x,(\eta_2)_y)}  &
}\]
as required.
\end{proof}

% =====
% =====
% =====

\subsection{Frobenius algebras}
Let $A,B$ be $\kk$-algebras and let $M$ be an $A$-$B$-bimodule.  Let $\Aut(B)$ denote the algebra automorphisms of $B$.  If $\beta\in \Aut(B)$, let $M_\beta$ denote the $A$-$B$-bimodule which is equal to $M$ as a left $A$-module but has right $B$-action given by $m\cdot b=m\beta(b)$.

The vector space $M^*=\Hom_\kk(M,\kk)$ is a 
$B\da A$-bimodule by the formulas $(\xi a)(m)=\xi (am)$ and $(b\xi )(m)=\xi (mb)$, for $\xi \in M^*$, $a\in A$, $b\in B$, and $m\in M$.

\begin{definition}
A \emph{Frobenius structure} on $A$ is a pair $(\varphi,\alpha)$ where $\alpha\in\Aut(A)$ and $\varphi:A\arr\sim (A^*)_\alpha$ is an isomorphism of $A\da A$-bimodules.  
If $\alpha$ belongs to a Frobenius structure for $\C$ then we say it is a \emph{Nakayama automorphism} for $\C$.
\end{definition}

We say that $A$ is a \emph{Frobenius algebra} if there exists a Frobenius structure on $A$.  

Note that a given algebra $A$ can have different Frobenius structures, with different Nakayama automorphisms, but the Nakayama automorphism of a Frobenius algebra is well-defined up to inner automorphisms (i.e., the $1$-cell in $\kk\Alg$ is defined up to isomorphism).

\begin{remark}
A Frobenius algebra is just a looping of a one-object $\fdv$-enriched category with Serre duality: see Example \ref{eg:vec-piv} and note that the left and right bimodule actions on $A^*$ match the maps $m^\ell$ and $m^r$ described in Section \ref{ss:defs+duals}.  In fact we could consider Frobenius algebra objects in a pivotal monoidal category $\V$.
\end{remark}

We want to relate Frobenius algebras to categories with Serre duality.
To start, we will show that Serre structures pass through many of the constructions introduced in Section \ref{s:cats}.

Let $\C$ be a $\kk$-linear category.  The first part of the following Lemma, on $\ic\C$, is \cite[Lemma 3.4]{chen}.
\begin{lemma}\label{lem:serre-icmatind} 
If $\C$ has a Serre structure then so do the categories $\ic\C$, $\Mat\C$, and $\ind\C$.
\end{lemma}
\begin{proof}
Let $(\se,\kappa)$ be a Serre structure on $\C$.  
Then $\ic\se$ is an autoequivalence of $\ic\C$; we claim it is a Serre functor.
So we need to construct a natural isomorphism
\[ \ic\kappa_{(x,e),(y,d)}:\ic\C((x,e),(y,d))\to\ic\C((y,d),(\se x,\se e))^*. \]
We restrict $\kappa_{x,y}$ to $\ic\C((x,e),(y,d))=d\C(x,y)e$. 
By naturality,
\[ \kappa_{x,y}(d\C(x,y)e)= \iota(d)\kappa_{x,y}(\C(x,y))\se(e) = \iota(d)\C(y,\se x)^*\se(e) \]
which is exactly $\ic\C((y,d),(\se x,\se e))^*$.

Next we claim $\Mat\se$ is a Serre functor for $\Mat\C$.
Given $X,Y\in\Mat\C$ we can write $X=X_1\oplus\cdots\oplus X_n$ and $Y=Y_1\oplus\cdots\oplus Y_m$ with $X_i,Y_i\in\C$.  
As vector spaces,
\[ \Mat\C(X,Y)\cong\C(X_1,Y_1)\oplus\cdots\oplus \C(X_n,Y_m). \]
To construct $\Mat\kappa_{X,Y}$ we apply a map $\kappa_{X_i,Y_j}$ to each summand and use the isomorphism
\[ \C(Y_1,\se X_1)^*\oplus\cdots\oplus \C(Y_m,\se X_n)^*\stackrel\sim\leftarrow\Mat\C(Y,\se X)^*\]
which exists because $\C$ is hom-finite.

Finally, as Serre functors respect biproducts, we obtain a Serre structure on $\ind\C$ by Lemma \ref{lem:serreonsubcat}.
\end{proof}

\begin{example}
The converse statements of Lemma \ref{lem:serre-icmatind} are all false.  Let $\D$ be a $\kk$-linear category with precisely two objects $x$ and $y$ and a Serre functor which interchanges them.  Consider the full subcategory $\C_1$ of $\Mat\D$ on the objects $x$ and $x\oplus y$.  Then $\ic\C_1$ has a Serre structure but $\C_1$ does not.  Similarly, consider the full subcategory $\C_2$ of $\Mat\D$ on the objects $x$, $y$, and $x\oplus x$.  Then $\Mat\C_2$ and $\ind\C_2$ have Serre structures but $\C_2$ does not.
\end{example}

Let $\C$ be a category with finitely many objects.  We say that a Serre structure $(\se,\kappa)$ is \emph{strict} if $\se$ is an automorphism (not just an autoequivalence).  Recall the algebra $A_\C$ constructed in Section \ref{ss:algebras}.  
\begin{proposition}\label{prop:frobserre}
There is a bijection between Frobenius structures on $A_\C$ and strict Serre structures on $\C$.
\end{proposition}
\begin{proof}
Let $A=A_\C$.  If $(\se,\kappa)$ is a strict Serre structure on $\C$ then $\se$ induces an automorphism $\alpha$ on $A$ and $\kappa$ induces an isomorphism $\varphi:A\arr\sim (A^*)_\alpha$ of $A\da A$-bimodules.  Now suppose $(\varphi,\alpha)$ is a Frobenius structure on $A$.  The algebra $A$ has a primitive idempotent $e_x$ for each object $x\in\ob\C$.  As $\alpha$ is an algebra automorphism, it preserves idempotents and the identity $\id_A=\sum_{x\in\ob\C}e_x$ and therefore permutes the idempotents.  This defines an action of $\se$ on $\ob\C$.  Then $\alpha$ extends this object action to a genuine automorphism of $\C$, and $\varphi$ provides the form $\kappa$.
\end{proof}

Recall the $2$-functor $\deloop$ from Section \ref{ss:algebras}.
Considering one-object categories gives the following.
\begin{corollary}\label{cor:serre-BA}
There is a bijection between Frobenius structures on $A$ and Serre structures on $\deloop A$.
\end{corollary}

Similarly to Section \ref{ss:uniqueserre}, there is a 2-category $\FAlg$ whose objects are Frobenius algebras.  We don't give the definition explicitly because, as with Proposition \ref{prop:forgetserre} and Corollary \ref{cor:bieqserre2gpds}, we have 2-functor
\[ \twofun{U}: \FAlg\to \Alg \]
whose core is locally an equivalence.  Therefore we can just work with the full sub-2-groupoid of $\Alg_\ei$ given by the algebras which admit Frobenius structures.

Recall the composite biequivalence
\[ \biP:\fBasicAlg_\ei \arrr{\tilde{\deloop}}  \fBaseCat \arrr\Mat \fAddCat \]
from Section \ref{ss:algebras} and the base algebra from Definition \ref{def:basealg}.
Let $\fBasicFAlg_\ei$  and $\fAddSCat$ denote the full sub-2-categories of $\fBasicAlg_\ei$ and $\fAddCat$ on the objects which admit Frobenius and Serre structures, respectively.

\begin{theorem}\label{thm:frob-serre}
A Frobenius structure on an algebra $A$ induces a Serre structure on the $\kk$-category $\biP A$.
Therefore,
the biequivalence $\biP:\fBasicAlg_\ei \arr\sim \fAddCat$ restricts to a biequivalence 
\[\fBasicFAlg_\ei \arr\sim \fAddSCat. \]
Moreover, a Serre structure on a finite $\kk$-category $\C$ induces a Frobenius structure on its base algebra $B_\C$.
\end{theorem}
\begin{proof}
We get the first statement by combining Lemma \ref{lem:serre-icmatind} and Corollary \ref{cor:serre-BA}.  So by Lemma \ref{lem:2funfull} we get the 2-functor
$\fBasicFAlg_\ei \arr\sim \fAddSCat$, and it is a biequivalence by Corollary \ref{cor:bieqserre2gpds}.
To get the final statement, use Corollary \ref{cor:skel-serre} and Lemma \ref{lem:serre-icmatind}.  Then, noting that every Serre structure on a skeletal category is strict, the result follows by Proposition \ref{prop:frobserre}.
\end{proof}

\section{Graded categories}\label{s:graded}

We largely follow Asashiba \cite{asa2}, though our conventions are slightly different: see Remark \ref{rmk:compare-asa} below.  Orbit categories rose in prominence after their use in categorifying cluster algebras by Buan, Marsh, Reineke, Reiten, and Todorov \cite{bmrrt} and related work of Keller \cite{k-orbit}.  They were first systematically studied by Cibils and Marcos \cite{cm}.  

From now on, we specialise from the pivotal enhanced setting and work in $\kk$-categories, so we work with traditional linear algebra duals $(-)^*=\Hom_\kk(-,\kk)$ instead of abstract duals $(-)^\vee$.  But we will use the pivotal enhanced theory from Section \ref{sss:hom-gr-ss}, where our pivotal structure will depend on a fixed character.

\subsection{Equivariant and hom-graded categories}

Fix a group $G$.  In applications, $G$ will be abelian, so we write the composition of $p,q\in G$ additively, as $p+q$, but everything would work for an arbitrary group.

\subsubsection{Equivariant categories}

A \emph{$G$-equivariant category} is a $\kk$-category $\D$ together with a $G$-action: a group homomorphism from $G$ to the group $\Autom(\D)$ of automorphisms of $\D$. 

$G$-equivariant categories are the objects of the following $2$-category, which we denote $G\aCat$:
\begin{itemize}
\item The 0-cells are $G$-equivariant categories: pairs $(\D,F)$ where $\D$ is a $\kk$-category and $F:G\to\Autom(\D)$ is a group homomorphism.  We write $F^p=F(p)$.
\item The 1-cells $(\D_1,F_1)\to(\D_2,F_2)$ are equivariant functors: pairs $(\Phi,\phi)$ where $\Phi:\D_1\to\D_2$ is a functor and $\phi=(\phi^p)_{p\in G}$ is a $G$-indexed family of natural isomorphisms $\phi_p:\Phi F_1^p\arr\sim F_2^p\Phi $ which respect group composition, i.e., $\phi^{p+q}=F_2^p\phi^q\circ \phi^p F_1^q$. 
\item The 2-cells $(\Phi,\phi)\to(\Psi,\psi)$ are morphisms of equivariant functors: natural transformations $\alpha:\Phi\to\Psi$ such that $\psi^p\circ \alpha F_1^p=F_2^p\alpha\circ\phi^p$,
i.e., the following square commutes:
\[ \xymatrix{
\Phi F_1^p \ar[r]^{\phi^p}\ar[d]^{\alpha F_1^p}  & F_2^p \Phi \ar[d]^{F_1^p \alpha} \\
\Psi F_1^p \ar[r]^{\psi^p} & F_2^p \Psi 
} \]
\end{itemize}
Composition of 1-cells is as follows.  Given $(\Phi_1,\phi_1):(\D_1,F_1)\to(\D_2,F_2)$ and $(\Phi_2,\phi_2):(\D_2,F_2)\to(\D_3,F_3)$, their composite is $(\Phi_2,\phi_2)(\Phi_1,\phi_1)=(\Phi_2\Phi_1,\phi_2^p\Phi_1\circ \Phi_2\phi_1^p)$.  Composition of 2-cells, both vertical and horizontal, is as usual for natural transformations

Note that any $\kk$-linear category $\D$ can be given a trivial $G$-category structure: we write $\trivG \D=(\D,F)$ for $\D$ equipped with the group homomorphism $F$ sending every $p\in G$ to the identity functor on $\D$.

\begin{remark}\label{rmk:Zequivar}
Our main application involves $G=\Z$.  In this case, we can specify less data to determine a $G$-equivariant category.  Given an automorphism $F^1:\D\arr\sim\D$, we define $F^p=(F^1)^p$ for $p\in\Z$: this explains our use of superscripts for $G$-indexing.  We denote this $\Z$-equivariant category $(\D,\gen{F^1})$.  Given a functor $\Phi:\D\to\D'$ and a natural isomorphism $\phi:\Phi F^1\arr\sim F'^1\Phi $, we define $\phi^p$ as the composition
\[ \Phi F^p=\Phi F^1\cdots F^1F^1\arrr{ \phi^1\id}F'^1\Phi F^1\cdots F^1\arrr{\id \phi^1\id}\cdots\arrr{\id\phi^1 } F'^1\cdots F'^1\Phi= F'^p\Phi\]
for $n>0$, and use a similar definition for $n<0$.  This respects the group composition by construction.  Similarly, for 2-cells, we only need to check commutativity of the $p=1$ square.
\end{remark}

\begin{example}\label{eg:equivar-cat}
Let $G=\Z$ and let $\D$ be the $\kk$-category with $\ob \D=\Z$ and 
\[ \D(i,j)=\begin{cases}
\kk & \text{if } j\in{i,i+1};\\
0 & \text{otherwise.}
\end{cases} \]
This completely determines the composition.  Choose a basis $f_i$ of each space $\D(i,i+1)$.  Define a functor $F^1:\D\to\D$ which acts on objects by $F^1(i)=i+2$ and on maps by $F^1(f_i)=f_{i+2}$.  Let $F:\Z\to\Autom(\D)$ send $n$ to $(F^1)^n$.  Then $(\D,F)$ is a $\Z$-equivariant category.  It is equivalent in $\Z\aCat$ to $(\Db(\kk A_2),\tau^-)$, the bounded derived category of the $A_2$ quiver together with its derived inverse Auslander-Reiten translation.
\end{example}

\subsubsection{Strictification}\label{sss:strictify}

According to our definition, $G$-equivariant categories are equipped with strict $G$-actions: $G$ should act by automorphisms. 
However, in practice one often meets weak $G$-actions, where the elements of $G$ act by autoequivalences.  We would like to replace these by strict $G$-actions.  This isn't essential, but reduces the technical complexity of the paper.

A \emph{weak $G$-equivariant category} is a $\kk$-category $\D$ together with a weak $G$-action: a group homomorphism from $G$ to the group $\Autoeq(\D)$ of autoequivalences of $\D$.  This just means that there exist natural isomorphisms $F_{p+q}\arr\sim F_p F_q$.  A \emph{coherent $G$-equivariant category} $(\D,F,f)$ is a weak $G$-equivariant category $(\D,F)$, $F:G\to\Autoeq(\D)$, equipped with a specified natural isomorphisms $f^{p,q}:F^{p+q}\arr\sim F^p F^q$ satisfying the obvious associativity axiom $f^{pq,r}\circ f^{p,q}F^r=f^{p,qr}\circ F^pf^{p,r}$.

Let $(\D',F',f)$ be a coherent $G$-equivariant category.
\begin{definition}
A \emph{strictification} of $(\D',F',f)$ is a $G$-equivariant category $(\D,F)$ together with an equivalence $\varepsilon:\D'\arr\sim\D$ and a collection $\alpha=(\alpha^p:\varepsilon F'^{p}\arr\sim F^p \varepsilon)_{p\in G}$ of natural isomorphisms such that
\[ \alpha^{p+q} \circ \varepsilon f^{p,q} =  f^{p,q} \varepsilon\circ F^p\alpha^q\circ\alpha^p F'^q : \varepsilon F'^pF'^q\to F^{p+q} \varepsilon.\]
\end{definition}

There exist at least two ways to strictify, which involve making or category smaller or bigger.  One is to take a skeleton, so all autoequivalences must be automorphisms.  The other is to consider the 
discrete monoidal $\kk$-linear category $\looping G$ with objects given by elements of $G$, then to replace $\D'$ with a category of weakly $G$-equivariant functors from $\looping G$ to $\D'$.  In the case $G=\Z$, $\D'$ is given explicitly by sequences $(x_i,\alpha_i:Fx_i\arr\sim x_{i+1})_{i\in\Z}$ and $\varepsilon$ sends $(x_i,\alpha_i)$ to $x_0$.  See, for example, \cite[Theorem 5.4]{shinder} for more details.  Therefore we have:
\begin{theorem}
Every coherent $G$-equivariant category has a strictification.
\end{theorem}
We will use this implicitly from now on.  In particular, we use it whenever we have a coherent $G$-equivariant category but we want to apply Theorem \ref{thm:asa2eq} (below).

%----

\subsubsection{The equivariant centre}

The following result is useful.  Later we will combine it with Proposition \ref{prop:serre-comm-autoeq}.
\begin{proposition}\label{prop:ev-centre}
Let $(\D,F)$ be an equivariant category. 
Then $\centre{\Autom({\D})}  \subseteq  \centre{\Autom_{G\aCat}({(\D,F)})}$.
\end{proposition}
\begin{proof}
Let $(Z,z)\in \centre{\End({\D})}$ and let $(\Phi,\phi):(\D,F)\to(\D,F)$ be an equivariant endofunctor.  
We have a natural transformation $z_\Phi:Z \Phi\arr\sim \Phi Z$, so we just need to show that the following diagram commutes:
\[ \xymatrix @C=50pt{
Z\Phi F^p \ar[d]_{z_\Phi F^p } \ar[r]^{(z^p\Phi)\circ (Z\phi^p)} & F^p Z\Phi \ar[d]^{F^p z_\Phi} \\
\Phi ZF^p   \ar[r]_{(\phi^p Z)\circ(\Phi z^p)} & F^p\Phi Z
}.\]
We break it up as follows:
\[ \xymatrix @=40pt{
Z\Phi F^p \ar[d]_{z_\Phi F^p } \ar[dr]^{z_{\Phi F^p}} \ar[r]^{Z\phi^p} & ZF^p\Phi \ar[r]^{z^p\Phi} \ar[dr]^{z_{F^p \Phi}}   & F^p Z\Phi \ar[d]^{F^p z_\Phi} \\
\Phi ZF^p   \ar[r]_{\Phi z^p} & \Phi F^p Z \ar[r]^{\phi^p Z}  & F^p\Phi Z
}.\]
Then the triangles and the square commute by the assumption that $(Z,z)\in \centre{\End({\D})}$.
\end{proof}

% =========

\subsubsection{Hom-graded categories}

A \emph{$G$-hom-graded category} is a $\kk$-linear category whose hom spaces are $G$-graded, with composition respecting this grading.

$G$-hom-graded categories are the objects of the following $2$-category, which we denote $\GrCat G$: 
\begin{itemize}
\item The 0-cells are $G$--hom-graded categories $\C$, so each hom space has a direct sum decomposition $\C(x,y)=\bigoplus_{p\in G}\C^p(x,y)$.  If we have a homogeneous map $f\in \C^p(x,y)$ then we say $f$ has \emph{degree} $p$ and write $\deg f=p$.  The composition should respect the grading: if $g\in\C^q(y,z)$ then $\deg(gf)=p+q$.
\item The 1-cells $\C_1\to\C_2$ are degree-preserving functors: pairs $(H,\gamma)$ where $H:\C_1\to\C_2$ is a $\kk$-functor and $\gamma:\ob \C_1\to G$ is a function (the \emph{degree adjuster}) such that, for all maps $f:x\to y$ in $\C$, $\deg(Hf)=\deg(f)+\gamma(y)-\gamma(x)$.
\item The 2-cells $(H,\gamma)\to(I,\delta)$ are morphisms of degree-preserving functors: natural transformations $\theta:H\to I$ whose components $\theta_x:Hx\to Ix$ are homogeneous of degree $\delta(x)-\gamma(x)$.
\end{itemize}
Composition of 1-cells is as follows.  Given $(H_1,\gamma_1):\C_1\to\C_2$ and $(H_2,\gamma_2):\C_2\to\C_3$, their composite is $(H_2,\gamma_2)(H_1,\gamma_1)=(H_2H_1,\gamma_3)$ where $\gamma_3(x)=\gamma_1(x)+\gamma_2(H_1(x))$.  Composition of 2-cells, both vertical and horizontal, is as usual for natural transformations.

Given $p\in G$, let $\underline{p}:\ob\C_1\to G$ be the constant function with image $p$.  In particular, $\underline{0}$ sends every object to the identity of $G$.  We say $H$ is a \emph{strict} degree-preserving functor if $(H,$\underline{0}$)$ is degree-preserving.

\begin{example}\label{eg:homgr-cat}
Let $G=\Z$ and let $\ob\C=\{1,2\}$.  Let $\C(i,j)=\kk$ for all $i,j\in\{1,2\}$.  Let $\alpha$ be a basis for $\C(1,2)$ and let $\beta$ be a basis for $\C(2,1)$.  Let $\deg(\alpha)=0$ and $\deg(\beta)=1$.  Then $\C$ is a $\Z$-hom-graded category.  A degree-preserving functor on $\C$ is given in Example \ref{eg:homgr-cat-serre} below.
\end{example}

As in the ungraded case, we have 2-functors
\[ \Mat:\GrCat G\to \GrCat G \;\; \text{ and } \;\; \Ind:\GrCat G_{2,0}\to\GrCat G_{2,0}.\]
We set
\[ \BaseCat^G  = \Ind\ic\GrCat G_{2,0} \;\; \text{ and } \;\; \AddCat^G = \Mat\ic\GrCat G_{2,0} \]
and we get an analogue of Proposition \ref{prop:base-add}:
\begin{proposition}\label{prop:base-add^G}
There are biequivalences:
\[ \xymatrix @R=30pt {
  \fBaseCat^G \ar@/^/[rr]^\Mat   & \sim & \fAddCat^G \ar@/^/[ll]^\Ind
} \]
\end{proposition}

\subsection{Graded Serre structures}

\subsubsection{Equivariant Serre structures}

Let $(\D, F)$ be a $G$-equivariant category and suppose $\D$ has Serre structure $(\se,\kappa)$.  By Proposition \ref{prop:serre-comm-autoeq} we have canonical commutation maps 
\[ \zeta_{F^g}:\se F^g\arr\sim F^g \se \]
which we could use to upgrade $\se$ to a $G$-equivariant functor.  But for applications, especially with triangulated categories, it will be useful to consider more general commutation maps.

Let $\chi:G\to\kk^\times$ be a character of $G$.  
\begin{definition}\label{def:equivar-Serre}
A \emph{$\chi$-equivariant Serre structure} on $(\D,F)$ is a triple $(\se,\sse,\kappa)$ where $(\se,\kappa)$ is a Serre structure for $\D$ and 
$\sse=(\sse_g:\se F^g\arr\sim F^g\se)_{g\in G}$ with $\sse_g=\chi(g)\zeta_{F^g}$.
\end{definition}
Note that equivariant Serre functors have previously appeared in the physics literature: see \cite[Appendix A]{laz}.

It is immediate from the definition that: if a Serre structure exists for $\D$ then, for every character $\chi$, a $\chi$-equivariant Serre structure exists on $\D$.  Therefore there is some redundancy in the previous definition, but we still find it useful to record the natural isomorphism $\sse$ explicitly.
Note also that: if $(\se,\sse,\kappa)$ is a $\chi$-equivariant Serre structure  for $(\D,F)$ then the pair $(\se,\sse)$ is automatically an equivariant autoequivalence of $(\D,F)$.

\begin{example}\label{eg:equivar-serre}
We continue Example \ref{eg:equivar-cat}.  
Let $\chi=\sgn$ be the character with $\sgn(1)=-1$.
Define $\se:\D\to\D$ by $\se(i)=i+1$ and $\se(f_i)=f_{i+1}$.  
Define $\kappa_{i,j}:\D(i,j)\to\D(j,\se(i))^*$ by $\kappa_{i,i}(\id_i)=f_i^*$ and $\kappa_{i,i+1}(f_i)=\id_{i+1}^*$.  Then $(\se,\kappa)$ is a Serre structure on $\D$.
Our pivotal structure identifies double duals with the identity functor so,
following Proposition \ref{prop:serre-comm-autoeq}, we get maps:
\[ \D(i,F^1\se(j))\arrr{\text{adj}} \D(F^{-1}i,\se(j))\arrr{\kappa^*} \D(j,F^{-1}(i))^* \arrr{\text{adj}^*}  \D(F^1(j),i)^*\arrr{(\kappa^{-1})^*} D(i,\se F^1(j)). \]
Note that $F^1\se=\se F^1$, with $F^1\se(j)=j+3$.  So if $j=i-3$ we have $\id_i\mapsto \id_{i-2}\mapsto f_{i-3}^*\mapsto f_{i-1}^*\mapsto \id_i$, and if $j=i-2$ we have $f_i\mapsto f_{i-2}\mapsto \id_{i-2}^*\mapsto \id_{i}^*\mapsto f_i$.  This map is represented by the inverse of the identity natural transformation $\zeta_{F^1}:\se F^1\arr\sim F^1 \se$.  Let $\sse_g=(-1)^g\id_{F^1\se}$.  Then $(\se,\sse,\kappa)$ is a $\chi$-equivariant Serre structure on $(\D,F)$.
\end{example}

We will upgrade the existence and uniqueness results of Section \ref{ss:defs+duals} to the $\chi$-equivariant setting.
\begin{proposition}\label{prop:unique-ev}
If $(\se,\sse,\kappa)$ and  $(\se',\sse',\kappa')$ are two $\chi$-equivariant Serre structures on $(\D,F)$ then there exists an isomorphism of equivariant functors $\alpha:(\se,\sse)\to(\se',\sse')$ such that $(\id_\D,\alpha)$ is an isomorphism of Serre structures.
\end{proposition}
\begin{proof}
By Proposition \ref{prop:uniqueserre}
there exists an isomorphism $(\id_\D,\alpha)$ of Serre structures, where $\alpha:\se\to\se'$ is represented by the following map of hom spaces:
\[ \D(x,\se y) \arrr{\kappa^*} \D(y, x)^* \arrr{((\kappa')^{-1})^*} \D(x,\se' y). \]
Let $F^-$ be a quasi-inverse of $F$.  Then, by functoriality, $F\alpha:F\se\to F\se'$ is represented by:
\[ \D(x,F\se y) \arrr{\text{adj}} \D(F^-x,\se y) \arrr{\kappa^*} \D(y, F^-x)^* \arrr{((\kappa')^{-1})^*} \D(F^-x,\se' y) \arrr{\text{adj}}\D(x,F\se' y). \]

Recall that $s=\chi(g)\zeta_{F^g}$, where
$\zeta_{F^g}$ 
is defined by Proposition \ref{prop:serre-comm-autoeq}, so is represented by the following composition:
\[ \D(x,F\se y) \arrr{\text{adj}} \D(F^-x, \se y) \arrr{\kappa^*} \D(y, F^-x)^* \arrr{\text{adj}^*} \D(Fy, x)^* \arrr{\kappa^{-1}} \D(x, \se Fy) \]
and $\zeta'_{F^g}$ is defined similarly.

We need to check that $\alpha$ is a morphism of equivariant functors.  
So, by factoring out $\chi(g)$, we just need to check that the following diagram commutes:
\[\xymatrix{
\se F^g \ar[d]^{\alpha F^g} \ar[r]^{\zeta_{F^g}} &  F^g\se \ar[d]_{F^g\alpha} &  \\
\se' F^g \ar[r]^{\zeta'_{F^g}}  & F^g \se'
}.  \]
We can check this on the representing morphisms, giving the outside rectangle of the following diagram:
\[ \xymatrix{
(x,F^g\se y) \ar[r]\ar[d] &(F^{-g}x,\se y) \ar[r] & (y, F^{-g}x)^* \ar[r] & (F^gy,x)^* \ar[r] & (x,\se F^gy) \ar[dd] \\
(F^{-g}x,\se y) \ar[d] \ar@{=}[ur] &&&&\\
(y,F^{-g}x)^* \ar[d] \ar@{=}[ddrr] \ar@{=}[uurr] &&&& (F^gy,x)^* \ar[dd] \ar@{=}[uul] \ar@{=}[ddl] \\
(F^{-g}x,\se'y) \ar[d] \ar@{=}[dr] &&&&&\\
(x,F^g\se'y) \ar[r] &(F^{-g}x,\se' y) \ar[r] & (y, F^{-g}x)^* \ar[r] & (F^gy,x)^* \ar[r] & (x,\se' F^gy)
}  \]
We have indicated how to break this into smaller diagrams, each of which clearly commutes.
\end{proof}

\begin{definition}\label{def:equiss2}
Let $(\se_i,\sse_i,\kappa_i)$ be a $\chi$-equivariant Serre structure for $(\D_i,F_i)$.
An \emph{equivalence of $\chi$-equivariant Serre structures} is a triple $(\Phi,\phi,\alpha)$, where $\alpha$ is a natural isomorphism $\Phi\se_1\to\se_2\Phi$, such that:
\begin{itemize}
\item $(\Phi,\phi):(\D_1,F_1)\to (\D_2,F_2)$ is an equivalence of $G$-equivariant categories,
\item $(\Phi,\alpha)$ is an equivalence of Serre structures, and
\item $\alpha$ is a morphism $(\Phi,\phi)(\se_1,\sse_1)\to (\se_2,\sse_2)(\Phi,\phi)$ of $G$-equivariant functors.
\end{itemize}
\end{definition}

\begin{proposition}\label{prop:upgrade-ev}
Let $(\D_1,F_1)$ and $(\D_2,F_2)$ be $G$-equivariant categories.  Suppose $\D_1$ has a 
$\chi$-equivariant Serre structure $(\se_1,\sse_1,\kappa_1)$ and we have an equivalence of $G$-equivariant categories $(\Phi,\phi):(\D_1,F_1)\to (\D_2,F_2)$.  Then
there exists a 
$\chi$-equivariant Serre structure $(\se_2,\sse_2,\kappa_2)$ on $(\D_2,F_2)$, and
$(\Phi,\phi)$ can be upgraded to an equivalence of $\chi$-equivariant Serre structures $(\Phi,\phi,\alpha)$. 
\end{proposition}
\begin{proof}
By Proposition \ref{prop:transport-serre} we know that $(\se_2,\kappa_2)$ exists, and $\sse_2$ exists automatically, so we have $(\se_2,\sse_2,\kappa_2)$.  By Proposition \ref{prop:transport-serre} again, we know that $\alpha$ exists and that $(\Phi,\alpha)$ is an equivalence of Serre structures.  So we just need to check that $\alpha$ is a morphism of $G$-equivariant functors, i.e., that the following diagram commutes:
\[\xymatrix{
F_2^g\Phi\se_1\ar[r]^{\phi^g\se_1}\ar[d]_{F^g_2\alpha} & \Phi F_1^g\se_1 \ar[r]^{\Phi \sse_1^g} & \Phi\se_1 F^g_1 \ar[d]^{\alpha F_1^g} \\
F_2^g \se_2\Phi \ar[r]^{\sse_2^g\Phi} & \se_2 F^g_2\Phi \ar[r]^{\se_2\phi^g} & \se_2 \Phi F^g  
}  \]

Let $\Phi^-$ denote an (adjoint) quasi-inverse of $\Phi$, so we get a natural transformation $\phi^-: F_1^{-g}\Phi^- \arr\sim \Phi^- F_2^{-g}$ defined by the adjunction.

The diagram we need to draw is similar to the proof of Proposition \ref{prop:unique-ev},
but much bigger.  We suggest the interested reader draws this on a large piece of paper; here we only sketch the details.  After some simplification, it reduces to a diagram between hom-spaces for $\D_2$ of the following form:
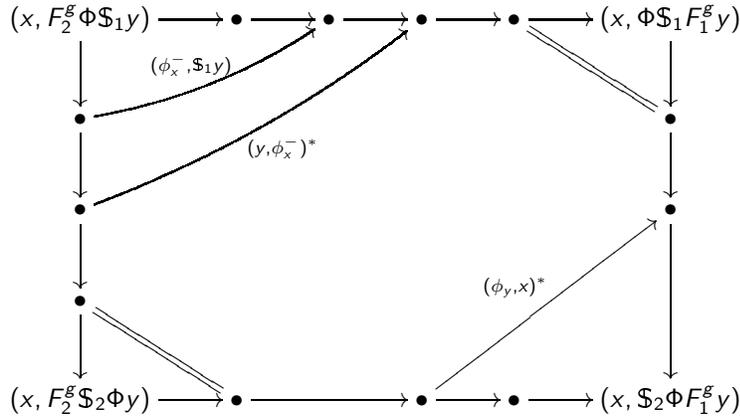
\[ \xymatrix{
(x,F_2^g\Phi \se_1y) \ar[r]\ar[d] & \bullet \ar[r] & \bullet \ar[r] & \bullet \ar[r] & \bullet \ar[r] \ar@{=}[dr] & (x, \Phi\se_1 F^g_1 y) \ar[d] \\
\bullet \ar[d] \ar@/_/[urr]^{(\phi^-_x,\se_1y)} &&&&& \bullet \ar[d] \\
\bullet \ar[d] \ar@/_/[uurrr]_{(y,\phi^-_x)^*} &&&&& \bullet \ar[dd]  \\
\bullet \ar[d] \ar@{=}[dr] &&&&&&\\
(x,F_2^g\se_2\Phi y) \ar[r] & \bullet \ar[rr] & & \bullet \ar[r] \ar[uurr]^{(\phi_y,x)^*} & \bullet \ar[r] & (x,\se_2 \Phi F_1^gy)
}  \]
After removing the top-right and bottom-left corners, and contracting the equalities, we are left with four squares which need to commute.  Going from top-left to bottom-right, they commute because of: the definition of $\phi^-$, the naturality of $\kappa_1$, the dual of the definition of $\phi^-$, and the naturality of $\kappa_2$.
\end{proof}

As with Corollary \ref{cor:bieqserre2gpds}, we can define a 2-category $G\aSCat$ whose objects are $G$-equivariant categories equipped with $\chi$-equivariant Serre structures and the above uniqueness results show that it is biequivalent to the full sub-2-category of $G\aCat$ on the $G$-equivariant categories admitting such structures.

\subsubsection{Hom-graded Serre structures}\label{sss:hom-gr-ss}

Let $\fdv^G$ denote the monoidal category of $G$-graded $\kk$-vector spaces $V=\bigoplus_{g\in G}V^g$, where $V^g\in\fdv$.  This category is left rigid, with $V^\vee=V^*$ and $(V^*)^g=(V^{-g})^*$.  Note that if $f:X^p\to Y^q$ is of degree $q-p$ then $f^*:(Y^*)^{-q}\to(X^*)^{-p}$ is also of degree $q-p$.

Let $\chi:G\to\kk^\times$ be a character of $G$.  
Recall the notion of a pivotal structure from Section \ref{ss:moncat}.
We define a pivotal structure $\iota^\chi_V:V\arr\sim V^{**}$ on $\fdv^G$ by sending the homogeneous vector $v\in V^g$ to $\chi(g)\ev_v$.  With this pivotal structure, we denote our monoidal category $\fdv^\chi$.

Let $\C$ be a $G$-hom-graded locally hom-finite category, so $\C$ is enriched in $\fdv^G$.  We write $\C^p(x,y)^*$ for the $p$th homogeneous summand of $\C(x,y)^*$, so 
$\C^p(x,y)^*=(\C^{-p}(x,y))^*$.  The composition 
$ m: \C^q(y,z) \otimes \C^p(x,y)\to \C^{p+q}(x,z) $
respects the grading by definition, and thus so do the maps
$ m^\ell: \C^q(x,y)^{**} \otimes  \C^p(x,z)^*\to \C^{p+q}(y,z)^*$
and
$ m^r:  \C^q(x,z)^*\otimes \C^p(y,z)\to \C^{p+q}(x,y)^*$.
\begin{definition}\label{def:hom-gr-se}
A \emph{$\chi$-hom-graded Serre structure} on $\C$ is a triple $(\se,\ell,\kappa)$ where $(\se,\ell)$ is a degree-preserving endofunctor of $\C$
and $(\se,\kappa)$ is an $\fdv^\chi$-enriched Serre structure for $\C$ 
such that, for all $x\in\ob\C$, the map $\kappa_x$ restricts to an isomorphism
\[ \C^0(x,x)\arr\sim \C^{-\ell(x)}(x,\se x)^*. \]
\end{definition}
Note that, by the bimodule conditions, this implies that $\kappa$ restricts to isomorphisms
\[ \C^p(x,y)\arr\sim \C^{p-\ell(x)}(y,\se x)^*. \]

\begin{example}\label{eg:homgr-cat-serre}
We continue to work with the $\Z$-hom-graded category $\C$ from Example \ref{eg:homgr-cat}.  
Let $\chi=\sgn$ be the character with $\sgn(1)=-1$.  First we will give a $\fdv^\chi$-enriched Serre structure $(\se,\kappa)$ for $\C$.
Define $\se:\C\to\C$ by $\se(1)=2$ and $\se(2)=1$ on objects.  Define $\kappa_{1,1}:\C(1,1)\to\C(1,2)^*$ by $\kappa_{1,1}(\id_1)=\alpha^*$ and $\kappa_{2,2}:\C(2,2)\to\C(2,1)^*$ by $\kappa_{2,2}(\id_2)=\beta^*$.  Then the bimodule conditions for $\kappa$, using the pivotal structure for $\chi$, stipulate that $\kappa_{1,2}(\alpha)=\id_2^*=\beta^*\se(\alpha)$ and $\kappa_{2,1}(\beta)=-\id_1^*=\alpha^*\se(\beta)$.  Thus $\se$ must act on maps by $\se(\alpha)=\beta$ and $\se(\beta)=-\alpha$.  Now define $\ell:\ob\C\to\Z$ by $\ell(1)=0$ and $\ell(2)=1$.  Then $\kappa_x$ does restrict to an isomorphism $\C^0(x,x)\arr\sim \C^{-\ell(x)}(x,\se x)^*$ (notice that $\deg(\beta)=1$, so $\deg(\beta^*)=-1$), and $(\se,\ell)$ is a degree-preserving functor because $1=\deg(\beta)=\deg(\se \alpha) = \deg(\alpha)+\ell(2)-\ell(1)=0+1-0$ and $0=\deg(-\alpha)=\deg(\se\beta)=\deg(\beta)+\ell(1)-\ell(2)=1+0-1$.
So $(\se,\ell,\kappa)$ is a $\sgn$-hom-graded Serre structure on $\C$.
\end{example}

Again, we have a uniqueness result.
\begin{proposition}\label{prop:unique-hg}
If $(\se,\ell,\kappa)$ and  $(\se',\ell',\kappa')$ are two $\chi$-hom-graded Serre structures on $\C$ then there exists an isomorphism of degree-preserving functors $\alpha:(\se,\ell)\to(\se',\ell')$ such that $(\id_\C,\alpha)$ is an isomorphism of Serre structures.
\end{proposition}
\begin{proof}
By Proposition \ref{prop:uniqueserre} we have an isomorphism $\alpha$, and we just need to check that it is a morphism of degree-preserving functors.  It is constructed from the following diagram:
\[ \xymatrix{
\C^p(x,y) \ar[r]^{\kappa_{x,y}}\ar[d]_{\kappa_{x,y}'} &\C^{p-\ell(x)}(y,\se x)^* \\
\C^{p-\ell'(x)}(y,\se' x)^*\ar[ur]_{\C(y,\alpha_x)^*}  
} \]
Taking duals, we see that $\alpha$ is constructed using the Yoneda lemma from the following map:
\[ \C^{p-\ell(x)}(y,\se x) \to \C^{p-\ell'(x)}(y,\se' x) \]
So the components $\alpha_x:\se x\to \se'x$ are homogeneous of degree $\ell'(x)-\ell(x)$ and thus $\alpha$ is a morphism of degree-preserving functors.
\end{proof}

Let $(\se_i,\ell_i,\kappa_i)$ be a $\chi$-hom-graded Serre structure on $\C$.
\begin{definition}\label{def:hgss2}
An \emph{equivalence of $\chi$-hom-graded Serre structures} is a triple $(H,\gamma,\alpha)$, where $\alpha$ is a natural isomorphism $H\se_1\to\se_2 H$, such that:
\begin{itemize}
\item $(H,\gamma):\C_1\to\C_2$ is an equivalence of $G$-hom-graded categories,
\item $(H,\alpha)$ is an equivalence of Serre structures, and
\item $\alpha$ is a morphism $(H,\gamma)(\se_1,\ell_1)\to (\se_2,\ell_2)(H,\gamma)$ of degree-preserving functors.
\end{itemize}
\end{definition}

\begin{proposition}\label{prop:upgrade-hg}
Let $\C_1$ and $\C_2$ be $G$-hom-graded categories.  Suppose $\C_1$ has a 
$\chi$-hom-graded Serre structure $(\se_1,\ell_1,\kappa_1)$ and we have an equivalence of $G$-hom-graded categories $(H,\gamma):\C_1\to\C_2$.  Then
there exists a 
$\chi$-hom-graded Serre structure $(\se_2,\ell_2,\kappa_2)$ on $\C_2$, and
$(H,\gamma)$ can be upgraded to an equivalence of $\chi$-hom-graded Serre structures $(H,\gamma,\alpha)$.
\end{proposition}
\begin{proof}
First fix an adjoint equivalence $(H,I,\eta,\varepsilon)$, and upgrade $I$ to a degree-preserving functor $(I,\delta)$ by setting $\delta(w)=-\gamma(Iw)$, so the components of $\eta$ and $\varepsilon$ are all of degree $0$.  Then, using Proposition \ref{prop:transport-serre}(a), we construct our degree-preserving Serre functor as follows:
\[ (\se_2,\ell_2) = (H,\gamma)(\se_1,\ell_1)(I,\delta)=(H\se_1I,\delta+\ell_1I+\gamma\se_1I). \]
Using 
$\delta=-\gamma I$ we get $\ell_2=\ell_1I+\gamma\se_1I-\gamma I$. 
Now, using Proposition \ref{prop:transport-serre}(a) again, we get our binatural isomorphism $\kappa_2$, and $(\se_2,\kappa_2)$ is a Serre structure for the $(\fdv^G,\iota^\chi)$-enriched category $\C_2$.

Now we check the degree condition of Definition \ref{def:hom-gr-se}.  The natural isomorphism $\kappa_2$ restricts to an isomorphism:
\[ \C_2^0(x,x)\arr{I} \C_1^0(Ix,Ix)\arr{\kappa_1}\C_1^{-\ell_1(Ix)}(Ix,\se_1 Ix)^* \arr{(H^{-1})^*} \C_2^{
(-\ell_1I-\gamma\se_1I+\gamma I)(x)}
(HIx, \se_2x)^*
\arr{\C_2(\varepsilon,\se_2 x)^*} \C_2^{-\ell_2(x)}(x,x)\]
and thus $(\se_2,\ell_2,\kappa_2)$ is a $\chi$-hom-graded Serre structure.  

Finally, using Proposition \ref{prop:transport-serre}(b), we  
get an isomorphism of degree-preserving functors  $\alpha:H\se_1\to \se_2 H$ such that $(H,\alpha)$ is an equivalence of Serre structures.
\end{proof}

Again, we can define a 2-category $\GrSCat G$ whose objects are $G$-hom-graded categories equipped with $\chi$-hom-graded Serre structures and the above uniqueness results show that it is biequivalent to the full sub-2-category of $\GrCat G$ on the hom-graded categories admitting such structures.

\subsection{Graded Frobenius algebras}

\subsubsection{Graded algebras}
Graded algebras are much more well-studied than graded categories.
A graded algebra is a unital algebra $A=\bigoplus_{p\in G}A^g$ with homogeneous composition.  A map of graded algebras is an algebra map which preserves degree.  
These form the objects and 1-cells of a $2$-category $\kk\Alg^G$, with 2-cells given by conjugation of degree 0 elements, but the resulting 2-functor $\deloop^G:\kk\Alg^G\to \GrCat G$ is not locally an equivalence: it is only locally fully faithful.  $\GrCat G$ has more 1-cells than $\kk\Alg^G$: the $1$-cells in $\kk\Alg^G$ correspond to the strict 1-cells in $\GrCat G$.

\begin{example}\label{eg:twistorno}
Let $A$ be the path algebra of the quiver
\[
\xymatrix{
1 \ar@/^/[r]^{\alpha} & 2\ar@/^/[l]^{\beta}
}
\]
modulo all paths of length $2$.  Put the arrow $\alpha$ in degree 0 and $\beta$ in degree $1$, so $A$ becomes a $\Z$-graded algebra.  Then there does not exist a map of graded algebras $A\to A$ which interchanges the arrows $\alpha$ and $\beta$, because their degrees are different.  But if we consider the finite $\Z$-graded category $\C=\deloop^G A$, we have an isomorphism $\C\arr\sim \C$ in $\kk\Alg^G$ which interchanges $\alpha$ and $\beta$.
\end{example}

Given an algebra $A$, let $\prim(A)$ denote its set of primitive idempotents. 
By Lemma \ref{lem:indic}, the indecomposable objects of $\ic\deloop A$ are indexed by $\prim(A)$.  If $A$ is basic then these are pairwise non-isomorphic.  Notice that if $A$ is graded then $\prim(A)\subset A^0$.

Given a graded algebra $A$, let $uA$ denote its underlying (ungraded) algebra.
\begin{definition}\label{def:damap}
The 2-category $\fBasicAlg^G$ is as follows:
\begin{itemize}
\item The 0-cells are basic finite-dimensional $G$-graded $\kk$-algebras.
\item The 1-cells $A\to B$ are \emph{degree-adjusted morphisms}: pairs $(f,\gamma)$ where $f:uA\to uB$ is a map of $\kk$-algebras and $\gamma:\prim(A)\to G$ is a function such that $a\in A^p$ implies $f(dae)\in A^{p+\gamma(e)-\gamma(d)}$.
\item The 2-cells $(f,\gamma)\to (g,m)$ are elements $b\in B$ such that,  for all  $a\in A$, $g(a)b=bf(a)$ and, for all $c\in\prim(B)$, $cbc\in B^{\gamma(c)-m(c)}$.
\end{itemize}
\end{definition}
By construction of $\fBasicAlg^G$ we have an analogue of Proposition \ref{prop:algcat}: 
\begin{proposition}\label{prop:gralgcat}
The 2-functor $\deloop^G_\ei$ induces a biequivalence $\tilde\deloop^G:\fBasicAlg^G_\ei\arr\sim\fBaseCat^G$.
\end{proposition}

\subsubsection{Graded Nakayama automorphisms}\label{ss:grfrob}

Let $A,B$ be $G$-graded algebras and let $M=\bigoplus_{g\in G}M^g$ be a graded $A$-$B$-bimodule.  Let $\chi:G\to\kk$ be a character on $G$.  We define the $\chi$-dual ${_\chi M^*}$ to be $M^*$ as a graded vector space, with 
$B\da A$-bimodule structure given by the formulas $(\xi a)(m)=\xi (am)$ and $(b\xi )(m)=\chi(g)\xi (mb)$, for $\xi \in {_\chi M^*}$, $a\in A$, $b\in B^g$, and $m\in M$.

\begin{definition}
A \emph{$\chi$-graded Frobenius structure} on $A$ is a triple $(\varphi,\sigma,\ell)$ where $(\sigma,\ell):A\arr\sim A$ is a 1-cell in $\fBasicAlg^G$ which is an equivalence and $\varphi:A\arr\sim (_\chi A^*)_{(\sigma,\ell)}$ is an isomorphism of $A\da A$-bimodules.  
We say $(\sigma,\ell)$ is a \emph{$\chi$-graded Nakayama automorphism} for $\C$.
\end{definition}

\begin{remark}
Traditionally, when studying graded Frobenius algebras, one considers bimodule morphisms $A\arr\sim (A^*)_\sigma\grsh{n}$ where $\grsh{n}$ denotes a grading shift.  But as our 2-category has more 1-cells than usual, the grading shift can be packaged within our degree-adjusted morphism $(\sigma,\ell)$ (see Definition \ref{def:damap}).
\end{remark}

Let $\C$ be a hom-graded category with finitely many objects.  Then, just as for ordinary categories in Section \ref{ss:algebras}, we construct a graded algebra $A_\C$.  We repeat Proposition \ref{prop:frobserre} in the graded setting:
\begin{proposition}\label{prop:frobserre-gr}
There is a bijection between $\chi$-graded Frobenius structures on $A_\C$ and strict $\chi$-hom-graded Serre structures on $\C$.
\end{proposition}

\begin{example}\label{eg:graded-nak}
Let $\C$ be as in Example \ref{eg:homgr-cat-serre}.  Then $A_\C$ is the algebra $A$ in Example \ref{eg:twistorno}, with $\prim(A)=\{e_1,e_2\}$.  It has a Frobenius structure with $\sgn$-graded Nakayama automorphism $(\sigma,\ell)$, where $\ell(e_1)=0$ and $\ell(e_2)=1$.  The map $\sigma$ interchanges $e_1$ and $e_2$, and sends $\alpha$ to $\beta$ and $\beta$ to $-\alpha$.  This is a degree-adjusted morphism because $\alpha=e_2\alpha e_1\in A^0$ and $\sigma(\alpha)\in A^{0+1-0}$, and $\beta=e_1\beta e_2\in A^1$ and $\sigma(\beta)\in A^{1+0-1}$.
\end{example}

There is a 2-category whose objects are $G$-graded algebras equipped with $\chi$-graded Frobenius structures.  On restricting to basic algebras and taking the core, we get a 2-groupoid which is biequivalent to the full sub-2-groupoid of the bicategory of $\fBasicAlg^G_\ei$ on algebras admitting graded Frobenius structures.

We have a $G$-graded biequivalence
\[ \biP^G:\fBasicAlg^G_\ei \arrr{\tilde{\deloop}^G}  \fBaseCat^G \arrr\Mat^G \fAddCat^G. \]
The following is a graded generalization of Theorem \ref{thm:frob-serre}, and is proved in the same way.
\begin{theorem}\label{thm:graded-ss-base}
A $\chi$-graded Frobenius structure on a $G$-graded algebra $A$ induces a $\chi$-hom-graded Serre structure on the $G$-graded $\kk$-category $\biP^G A$.
Therefore,
the biequivalence $\biP^G:\fBasicAlg^G_\ei \arr\sim \fAddCat^G$ restricts to a biequivalence 
\[\fBasicFAlg_\ei^G \arr\sim \fAddSCat^G. \]
Moreover, a $\chi$-hom-graded Serre structure on a finite $G$-graded $\kk$-category $\C$ induces a $\chi$-graded Frobenius structure on its $G$-graded base algebra $B_\C$.
\end{theorem}

Let $\tr$ denote the trivial character $G\to\kk^\times$ which sends every $g\in G$ to $1\in\kk$.  Given a linear map $f:A\to A$, let $f^\chi:A\to A$ be the map defined on homogeneous elements $a\in A^p$ by $f^\chi(a)=\chi(p)f(a)$.
As there is lots of existing knowledge of $\tr$-graded Frobenius structures on graded algebras, the following straightforward result is useful.  
\begin{lemma}\label{lem:trgr-chigr}
Let $A$ be $G$-graded.  Then
$(\varphi,\sigma,\ell)$ is a $\tr$-graded Frobenius structure on $A$ if and only if $(\varphi^\chi,\sigma^\chi,\ell)$ is a $\chi$-graded Frobenius structure on $A$.
\end{lemma}

%=====

\subsection{Smash products and orbit categories}

\subsubsection{The 2-functors}

First we explain the smash product $2$-functor
\[ -\smsh G:\GrCat G\to G \aCat. \]
Let $\C$ and $\C'$ be $G$-graded categories, let $(H,\gamma)$ and $(H',\gamma')$ be degree-preserving functors $\C\to\C'$, and let $\theta:(H,\gamma)\to(H',\gamma')$ be a morphism of degree-preserving functors.
\begin{itemize}
\item The $G$-category $\C\smsh G$ has objects $\ob\C\times G$, which we write as either $(x,p)$ or $x^p$ depending on context.  Homs are given by $\C\smsh G(x^p,y^q)=\C^{q-p}(x,y)$.  
The $G$-action $F:G\to\Autom(\C\smsh G)$ on objects is the obvious one: for $r\in G$ we have $F_r(x^p)=x^{p+r}$.  On morphisms it is trivial: $\C\smsh G(x^g,y^h)$ and $\C\smsh G(x^{p+r},y^{q+r})$ are both copies of the same set $\C^{q-p}(x,y)$,  
so $F_r$ takes $f:x^p\to y^q$ to $f:x^{p+r}\to y^{q+r}$ using the identity map.
\item The equivariant functor $(H,\gamma)\smsh G$ acts on objects of $\C\smsh G$ by sending $(x,p)$ to $(Hx,p+\gamma(x))$.  Given a morphism $f\in \C\smsh G(x^p,y^q)=\C^{q-p}(x,y)$ we send it to $Hf:Hx\to Hy$, which has degree $\gamma(y)+q-p-\gamma(x)$ and is thus a morphism in $\C\smsh G(x^{p+\gamma(x)},y^{q+\gamma(y)})$.  Note that this action is strict: the $G$-action commutes with the functor, so the natural isomorphism of our equivariant functor is just the identity.
\item The natural transformation $\alpha=\theta\smsh G$ is defined on components by $\alpha_{(x,p)}=\theta_x:Hx\to H'x$.  This map is homogeneous of degree $\gamma'(x)-\gamma(x)=(p+\gamma'(x))-(p+\gamma(x))$ so is a map from $(Hx)^{p+\gamma(x)}$ to $(H'x)^{p+\gamma'(x)}$.  As both $G$-actions are strict, and defined in the same way, $\alpha$ automatically commutes with the $G$-actions.
\end{itemize}

\begin{example}\label{eg:smash-cat}
Let $\C$ be the $\Z$-hom-graded category from Example \ref{eg:homgr-cat}.  Then $\C\smsh\Z$ has objects $x^p$, with $x\in\{1,2\}$ and $p\in\Z$.  Its nonzero morphism spaces are 1-dimensional, with bases $\alpha^p:1^p\to 2^p$ and $\beta^p:2^p\to 1^{p+1}$, for as well as the identity morphisms.  The $\Z$-action is given by $F^1(x^p)=x^{p+1}$, with $F(\alpha^p)=\alpha^{p+1}$ and $F(\beta^p)=\beta^{p+1}$.

Let $\D$ denote the $\Z$-equivariant category from Example \ref{eg:equivar-cat}.  We have an equivariant equivalence $(\Phi,\phi):\C\smsh\Z\to\D$ where $\Phi(x^p)=2p+x$ on objects, and $\Phi(\alpha^p)=f_{2p+1}$ and $F(\beta^p)=f_{2p+2}$ on maps.  The two compositions $\Phi F^1$ and $F^1\Phi$ are equal: both send $x^p$ to $2p+2+x$ and act in the obvious way on maps.  The natural transformation $\phi:\Phi F^1\to F^1\Phi$ is the identity, and it clearly satisfies the necessary commutative diagram.
\end{example}

Next we explain the orbit $2$-functor
\[ -/ G: G \aCat\to \GrCat G. \]
Let  $(\D,F)$ and $(\D',F')$ be $G$-equivariant categories, let $(\Phi,\phi)$ and $(\Psi,\psi)$ be equivariant functors $(\D,F)\to (\D',F')$, and let $\alpha:\Phi\to\Psi$ be a morphism of equivariant functors.  
\begin{itemize}
\item The $G$-graded category $\D/G$ is the \emph{orbit category}: it has the same objects as $\D$ and its homogeneous morphism spaces are $(\D/G)^p(x,y)=\D(x,F^py)$.
The composition of $f\in (\D/G)^p(x,y)$ and $g\in (\D/G)^q(y,z)$ is $F^p(g)\circ f$.
\item The degree-preserving functor $(\Phi,\phi)/G$ is strict (its degree adjuster is zero).  It is just $\Phi$ on objects, and it sends a degree $p$ morphism $f:x\to F^py$ to the composite $\phi_y\circ \Phi f$:
\[ \xymatrix @=10pt {
\Phi x \ar@{-->}[rr] \ar[dr]_{\Phi f} &&F^p \Phi y \\
& \Phi F^p y\ar[ur]_{\phi_y}&
}\]
\item The morphism $\alpha/G$ of degree-preserving functors is just $\alpha$.
\end{itemize}

If $(\D,F)$ is a $G$-equivariant category, note that $G$ acts on $(\D,F)$ (not just $\D$) in the following way:
\begin{align*} 
G \to &\End_{G\aCat}\left((\D,F)\right)\\
q \mapsto &(F^q,\phi_q)
\end{align*}
where $\phi_q=(\phi_q^p)_{p\in G}$ and 
\[ \phi_q^p=  \id_{F^{p+q}}: F^qF^p \to F^pF^q.\]
The grading of $\D/G$ encodes the group action of $\D$ in the following way:
\begin{lemma}\label{lem:gshift}
Let $\C=\D/G$.  There is an isomorphism of degree-preserving endofunctors of $\C$:
\[ \theta:(F^q,\phi_q)/G\arr\sim (\id_\C,\underline{q}).\]
\end{lemma}
\begin{proof}
For $x\in\C$, define the degree $q$ map $\theta_x:F_qx\to x$ by $\theta_x=\id_{F^qx}\in\D(F^qx,F^qx)=\C^q(F_qx,x)$.  As it's the identity, it's clearly natural.  Its inverse is the degree $-q$ map $\id_{x}\in\D(x,x)=\D(x,F^{-q}F^qx)=\C^-q(x,F_qx)$.
\end{proof}

\subsubsection{Asashiba's biequivalence}

The main result is the following \cite[Theorem 7.5]{asa2}:
\begin{theorem}[Asashiba]\label{thm:asa2eq}
Taking orbit categories is a biequivalence
with quasi-inverse given by taking smash products:
\[ ?/G:G\aCat\stackrel\sim\rightleftarrows\GrCat G :?\smsh G. \]
\end{theorem}

\begin{remark}\label{rmk:compare-asa}
Our categories are defined oppositely to \cite{asa2}, where the natural transformations given as part of an equivariant structure are in the opposite direction.  The orbit category and smash product category are defined with opposite signs, which correspond to our opposite conventions: see \cite[Proposition 2.11]{asa1}.  In \cite[Section 7.1]{asa2}, the orbit 2-functor is defined by the existence of 1- and 2-cells which fit in commutative diagrams.  It is straightforward to check that, under the above identifications, the constructions of 1- and 2-cells given above do make these diagrams commutative.
\end{remark}

So, combining Proposition \ref{prop:bieqmoneq} and Theorem \ref{thm:asa2eq}, we get:
\begin{corollary}
If $\D$ is a $G$-category then there is a group isomorphism 
$$\PicEnd_{G\aCat}(\D)\cong\PicEnd_{\GrCat G}(\D/G).$$
\end{corollary}

\subsubsection{Orbiting Serre functors} 

We show that the orbit 2-functor sends equivariant Serre functors to hom-graded Serre functors.  This improves on previous results: see \cite[Proposition 9.4]{dug-mesh} and \cite[Proposition 4.2]{chen}.

Let $(\se,\sse,\kappa)$ be a $\chi$-equivariant Serre structure on $(\D,F)$.
We want to show that taking the orbit category sends $(\se,\sse,\kappa)$ to a $\chi$-graded Serre structure $(\se^\C,\alpha,\kappa^\C)$ on $\C=\D/G$.  The orbit 2-functor gives us a degree-preserving functor $(\se^\C,\ell)=(\se,\sse)/G$ which is strict, so $\se^\C$ is $\se$ on objects, $\ell=\underline{0}$ sends every object to $0\in G$, and a map $f:w\to F^px$ in $\C$ is sent by $\se^\C$ to
\[ \se w \arr{\se f} \se F^px\arr{\sse^p_x} F^p\se x.\]  
So we need to define a map
\[ \kappa^\C_{u,v}:\C(u,v)\arr\sim \C(v,\se u)^* \]
which preserves the degree of morphisms.
Our map has components:
\[ \kappa^{\C,p}_{u,v}:\C^p(u,v)=\D(u,F^pv)\arrr{\chi(g)\kappa_{u,F^pv}}\D(F^pv,\se u)^*\arr{(F^{g})^*}\D(v,F^{-p}\se u)^* =\C^p(v,\se^\C u)^*. \]
Note the scaling by $\chi(g)$.

\begin{proposition}\label{prop:orbiting-serre}
Suppose $(\se,\sse,\kappa)$ is a 
$\chi$-equivariant Serre structure for $(\D,F)$.  
Then
$(\se^\C,\underline{0},\kappa^\C)$
 is a $\chi$-hom-graded Serre structure for $\C=\D/F$. 
\end{proposition}
\begin{proof}
We just need to check that $(\se^\C,\underline{0},\kappa^\C)$ is a $(\fdv^G,\iota^\chi)$-enriched Serre structure.
Fix maps $f:w\to x$, $g:x\to y$, and $h:y\to z$ of degrees $p$, $q$, and $r$, respectively.  Then we want the following diagram to commute:
\[ \xymatrix @C=0pt {
 \C^r(y,z)\ar[d]^{\iota^\chi_{\C(y,z)}} & \otimes &\C^q(x,y)\ar[d]^{\kappa_{x,y}} &\otimes &\C^p(w,x){\phantom{ab}} \ar[d]^\se \ar[rr] &{\phantom{abcdefgh}} & {\phantom{ab}} \C^{p+q+r}(w,z)\ar[d]^{\kappa_{w,z}}   \\
 \C^r(y,z)^{**} &\otimes &\C^q(y,\se x)^* &\otimes &\C^p(\se w,\se x){\phantom{ab}}\ar[rr]  & & {\phantom{ab}}\C^{p+q+r}(z,\se w)^* 
} \]
This says that $\kappa(hgf)=\iota(h)\kappa(g)\se(f)$.  We split this check into two halves, multiplying on  the left or right of $g$.
The ``left'' diagram check is:
\[ \kappa(h g)=\iota(h)\kappa(g). \]
The ``right'' diagram check is:
\[ \kappa(gf)=\kappa(g)\se(f). \]

Let's do the ``left'' check first.
Writing out the diagram carefully, using $\C^q(x,y)=\D(x, F^qy)$, we get:
\[ \xymatrix @C=35pt 
{
\D(y,F^rz) \otimes \D(x, F^qy) \ar[d]^{\iota^\chi\otimes \chi(q)\kappa } \ar[r] ^{F^q\otimes\id}
& \D(F^qy, F^{q+r}z) \otimes \D(x, F^qy) \ar[d]^{ \iota^\chi\otimes \chi(q)\kappa } \ar[r]^(.6)m & \D(x,F^{q+r} z)\ar[d]^{ \chi(q+r)\kappa} \\
\D(y,F^rz)^{**} \otimes \D(F^qy, \se x)^* \ar[d]^{\id\otimes (F^q)^* } \ar[r] ^(0.46){(F^q)^{**}\otimes\id}
& \D(F^qy, F^{q+r}z)^{**} \otimes\D(F^qy, \se x)^* \ar[r]^(.6)a & \D(F^{q+r}z, \se x)^* \ar[d]^{ (F^{q+r})^*} \\
\D(y,F^rz)^{**}  \otimes \D(y, F^{-q}\se x)^* \ar[r] ^{a} & \D(F^rz, F^{-q}\se x)^* \ar[r]^{(F^r)^*}   &\D(z, F^{-q-r}\se x)^* 
} \]
The top left square commutes by naturality of the pivotal structure $\iota^\chi$.  The top right square commutes by $\iota^\chi=\chi(r)\iota$ and naturality of $\kappa$.  The commutativity of the bottom hexagon is an exercise in taking duals of maps in a rigid tensor category.

Next we do the ``right'' check.
Recall how $\se^\C$ is defined from $(\se,\sse)$.
The diagram is:
\[ \xymatrix 
{
\D(x, F^q y) \otimes \D(w, F^px)  \ar[d]^{ \chi(q)\kappa\otimes\se } \ar[r] ^{F^p\otimes\id}
& \D(F^px, F^{p+q} y) \otimes \D(w, F^px)  \ar[d]^{ \chi(p)\chi(q)\kappa\otimes\se }  \ar[r]^m & \D(w,F^{p+q} y) \ar[d]^{ \chi(p+q)\kappa }  \\
\D(F^qy,\se x)^* \otimes \D(\se w, \se F^px) \ar[dd]^{(F^q)^*\otimes \D(\se w, \sse^p_x) } \ar[dr]^{(F^{-p})^*\otimes\id}
& \D(F^{p+q}y, \se F^{p}x)^* \otimes \D(\se w, \se F^px)  \ar[r]^(0.7)a & \D(F^{p+q} y, \se w) \ar[dd]^{(F^{p+q})^* }   \\
& \D(F^{p+q}y, F^{p}\se x)^* \otimes \D(\se w, \se F^px) \ar[dl]^(0.3){\phantom{abcd}(F^{p+q})^*\otimes\D(\se w, \sse^p_x) }  \ar[u]_{ \D(F^{p+q}y, \sse^p_x )^* \otimes \id} & \\
\D(y, F^{-q}\se x)^* \otimes \D(\se w, F^p\se x)\ar[r] ^(0.45){\id\otimes F^{-p-q}} &\D(y, F^{-q}\se x)^* \otimes \D( F^{-p-q}\se w, F^{-q}\se x)\ar[r]^(0.7)a   & \D(y,F^{-p-q}\se w)^* 
} \]
This consists of a triangle, square, pentagon, and hexagon.  The triangle (bottom left) commutes by strictness of the action of $F$ on $\D$.  The square (top right) commutes by naturality of $\kappa$.  The hexagon (bottom right) commutes by associativity of multiplication in $\D$ and definition of the right action on duals: this is another exercise in tensor categories.  For the pentagon (top right), we need the following diagram to commute:
\[ \xymatrix @C=5pt 
{
\D(x, F^q y)   \ar[dr]_{ \kappa } \ar[rr] ^{F^p}
&& \D(F^px, F^{p+q} y)  \ar[rr]^{ \chi(p)\kappa } && \D(F^{p+q}y, \se F^{p}x)^*   \\
&\D(F^qy,\se x)^*  \ar[rr]^{(F^{-p})^*}
&& \D(F^{p+q}y, F^p\se x)^*  \ar[ru]_{\phantom{abc} \D(F^{p+q}y, \sse^p_x )^*} &
} \]
but this follows from $\chi$-equivariance: $\se$ is defined as a $\chi$-scaling of the map inducing the composition in the proof of Proposition \ref{prop:serre-comm-autoeq}. 
\end{proof}

\subsubsection{Smashing Serre functors}

Now we go in the other direction.  Let $(\se,\ell,\kappa)$ be a $\chi$-hom-graded Serre structure for $\C$.  We want to show that taking the smash product sends $(\se,\ell,\kappa)$ to a $\chi$-equivariant Serre structure.  
Recall that $\D=\C\smsh G$ has objects $x^p$, with $x\in\C$ and $p\in G$.
Let $(\se^\D,\sse)$ denote the strictly equivariant functor $(\se,\ell)\smsh G$, so 
\[ \se^\D(x^p)=(\se x)^{p+\ell(x)} \]
 and $\sse^p:F^p\se^\D\arr\id\se^D F^p$ is the identity natural transformation.  

We need to define maps
\[ \kappa^\D_{u,v}:\D(u,v)\arr\sim \D(v,\se^\D u)^*. \]
If $u=x^p$ and $v=y^q$ then, using the definition of $\D$, this reduces to
\[ \kappa^\D_{(x,p),(y,q)}:\C^{q-p}(x,y)\arr\sim \C^{q-p-\ell(x)}(y,\se x)^* \]
So we set
\[ \kappa^\D_{(x,p),(y,q)}=\chi(p)\kappa_{x,y}. \]

\begin{proposition}\label{prop:smashing-serre}
Suppose $(\se,\ell,\kappa)$ is a $\chi$-hom-graded Serre structure for $\C$.  Then $(\se^\D,\id,\kappa^\D)$ is a  $\chi$-equivariant Serre functor on $\C\smsh G$.
\end{proposition}
\begin{proof}
First we check naturality of $\kappa^\D$.  This follows from $\iota^\chi$-naturality for $\kappa$ and the definition, as the following diagrams show:
\[ \xymatrix @C=0pt {
\D(x^j,y^k)\ar[d]^{\chi(k)\kappa_{x,y}} &\otimes &\D(w^i,x^j){\phantom{ab}} \ar[d]^\se \ar[rr] &{\phantom{abcdefgh}} & {\phantom{ab}} \D(w^i,y^k)\ar[d]^{\chi(k)\kappa_{w,y}}   \\
\D(y^k,\se (x^j))^* &\otimes &\D(\se (w^i),\se (x^j)){\phantom{ab}}\ar[rr]  & & {\phantom{ab}}\D(y^k,\se (w^i))^* 
} \]
and
\[ \xymatrix @C=0pt {
 \D(y^k,z^\gamma)\ar[d]^{\chi(\gamma-k)\iota} & \otimes &\D(x^j,y^k)\ar[d]^{\chi(k)\kappa_{x,y}}  \ar[rr] &{\phantom{abcdefgh}} & {\phantom{ab}} \D(x^j,z^\gamma)\ar[d]^{\chi(\gamma)\kappa_{x,z}}   \\
 \D(y^k,z^\gamma)^{**} &\otimes &\D(y^k,\se (x^j))^*  \ar[rr] & & {\phantom{ab}}\C(z^\gamma,\se (x^j))^* 
} \]

It remains to check $\sse^p=\chi(p)\zeta_{F^g}$.  
Recall that $\zeta_{F^g}$ is defined by the following diagram:
\[ \xymatrix @C=15pt {
\C^{-p-\ell(y)}(x,\se y)^* \ar[d]^{(\chi(p)\kappa_{y,x})^{-1}} \ar@{=}[r] &  \D(x, \se^\D F^py)^* \ar[d]\ar@{-->}[rr] && \D(x,F^p\se^\D y)^* \ar@{=}[r] & \C^{-p-\ell(y)}(x,\se y)^* \\
\C^{-p}(y,x) \ar@{=}[r]   &  \D(F^py, x) \ar[r] \ar@{=}[d] & \D(y,F^{-p} x) \ar@{=}[d] \ar[r] & \D(F^{-p} x,\se^\D y)^* \ar[u] \ar@{=}[d] \ar@{=}[r] & \C^{-p-\ell(y)}(x,\se y)^* \ar[u]^\id \\
  &  \C^{-p}(y,x)\ar[r]^\id &\C^{-p}(y,x) \ar[r]^(0.4){\kappa_{x,y}} &\C^{-p-\ell(y)}(x,\se y)^*
} \]
So $ \D(x, \se^\D F^py)^* \to \D(x,F^p\se^\D y)^* $ is multiplication by $\chi(-p)$, so $\zeta_{F^g}$ is multiplication by $\chi(-p)$, so $\sse^p=\id=\chi(p)\chi(-p)=\chi(p)\zeta_{F^g}$, as required.  
\end{proof}

\begin{example}
Given the $\sgn$-hom-graded Serre structure $(\se,\ell,\kappa)$ in Example \ref{eg:homgr-cat-serre}, we get a $\chi$-equivariant Serre structure on $\C\smsh\Z$ by the above construction.  On objects it sends $1^p$ to $2^p$ and $2^p$ to $1^{p+1}$, and on maps it sends $\alpha_p$ to $\beta_p$ and $\beta_p$ to $-\alpha_{p+1}$.  The maps $\kappa^{\C\smsh\Z}$ send $\id_{1^p}$ to $(-1)^p\alpha_p^*$ and $\id_{2^p}$ to $(-1)^{p+1}\beta^*_{p+1}$.  Under the equivalence of Example \ref{eg:smash-cat}, this gives a $\sgn$-equivariant Serre structure $(\se',\sse',\kappa')$ on $(\D,F)$ with $\se'(i)=i+1$, $\se'(f_i)=(-1)^{i+1}f_{i+1}$, and $\kappa'(\id_i)=(-1)^{i+1}f_i^*$.  This is not the same $\sgn$-equivariant Serre structure as given in Example \ref{eg:equivar-serre}, but it is $\sgn$-equivariantly equivalent, via the triple $(\id_\D,\id_{FF},\nu)$ where $\nu:\se\to\se'$ has components $\nu_i=(-1)^{i+1}\id_{i+1}$.
\end{example}

Fix $G$ and $\chi:G\to\kk^\times$.
Immediately from the uniqueness results (Propositions  \ref{prop:unique-ev}
 and \ref{prop:unique-hg})
 we get:
\begin{theorem}\label{thm:transfersss}
Asashiba's biequivalence restricts to a biequivalence of 2-groupoids
\[ ?/G:  G\aSCat_\ei  \stackrel\sim\rightleftarrows  {\GrSCat G}_\ei :?\smsh G. \]
\end{theorem}

% ===============

\section{Calabi-Yau categories}\label{s:cy}

From now on we set $G=\Z$ and work with $\Z$-equivariant categories.

\subsection{Triangulated Calabi-Yau categories}\label{ss:cy}

A triangulated category is a triple $(\D,\Sigma,\Delta)$ where
$\D$ is a $\kk$-linear category, $\Sigma:\D\arr\sim \D$ is an autoequivalence, and $\Delta$ is a distinguished subset of the set of ``triangles''
\[ x \arr f y \arr g z \arr h \Sigma x. \]
Triangulated categories should satisfy some well-known axioms.

If $\Sigma$ is in fact an automorphism of categories (as is sometimes specified in the definition) then $(\D,\Sigma)$ is a $\Z$-equivariant category.  If not, we can strictify as in \cite[Section 2]{kv-scd} to get an equivalent triangulated category $(\D',\Sigma')$ which is $\Z$-equivariant.  Hereafter, we will assume that our triangulated categories come equipped with automorphisms.

There is a 2-category $\Tri$ whose 0-cells are triangulated categories.  The 1-cells 
$(\D,\Sigma,\Delta) \to (\D',\Sigma',\Delta')$
are triangulated functors: pairs $(\Phi,\phi)$ consisting of a $\kk$-linear functor $\Phi:\D\to\D'$ and a natural isomorphism $\phi:\Phi\Sigma\arr\sim\Sigma'\Phi$ such that, given a distinguished triangle in $\Delta$, the triangle
\[ \xymatrix @=10pt {
\Phi x \ar[rr]^{\Phi f} && \Phi y \ar[rr]^{ \Phi g} && \Phi z \ar@{-->}[rr] \ar[dr]_{\Phi h} &&\Sigma \Phi x \\
&&&&& \Phi\Sigma z\ar[ur]_{\phi_z}&
}\]
is in $\Delta'$.  Note that, under the correspondence described in Remark \ref{rmk:Zequivar},
the triangulated functors give a subset of the equivariant functors from $(\D,\Sigma)$ to $(\D',\Sigma')$.  The 2-cells in $\Tri$ are just the morphisms of equivariant functors.

Due to the rotation axiom (TR2) for triangulated categories, $(\Sigma,\id_{\Sigma^2})$ is not a triangulated functor but $(\Sigma,-\id_{\Sigma^2})$ is a triangulated functor \cite{k-orbit,k-orbit-corr}.  

Now suppose $(\D,\Sigma,\Delta)$ is triangulated and that $\D$ has a Serre functor $\se:\D\arr\sim\D$.  Then \cite[Proposition 3.3]{bk}: 
\begin{proposition}[Bondal-Kapranov]\label{prop:bk-tri}
There exists a natural isomorphism $\sse:\se\Sigma\arr\sim\Sigma\se$ such that $(\se,\sse)$ is a triangulated functor.
\end{proposition}
Surprisingly, the natural isomorphism $\sse$ does not depend on $\Delta$.  Let $\sgn:\Z\to\kk$ denote the sign character on $\Z$ which sends the generators of $\Z$ to $-1\in\kk$.  Suppose $\D$ has a Serre structure $(\se,\kappa)$.
Then \cite[Theorem A.4.4]{bvdb}:
\begin{theorem}[Van den Bergh]\label{thm:vdb-tri}
If $(\se,\sse,\kappa)$ is a $\sgn$-equivariant Serre structure for $(\D,\Sigma)$ then $(\se,\sse)$ is a triangulated functor.
\end{theorem}
As Serre structures can always be made $\sgn$-equivariant, Theorem \ref{thm:vdb-tri} recovers Proposition \ref{prop:bk-tri}.  In fact, Theorem \ref{thm:vdb-tri} is an if and only if: see \cite[Proposition 2.2]{cz}.

If a triangulated category has Serre duality, it comes with two canonical triangulated autoequivalences.  It is natural to ask whether there is any relation between them.  Kontsevich noted that relations do exist in at least two cases: for derived categories of Calabi-Yau varieties and for derived categories of some quivers.  
He called these categories Calabi-Yau and fractional Calabi-Yau, respectively \cite{kon-ens}.  

The existence of a natural isomorphism between powers of $\se$ and $\Sigma$ might be called a weak Calabi-Yau condition.  We ask for a strong Calabi-Yau condition, as in \cite{k-cy}, using the whole graded structure.  
\begin{definition}\label{def:cy}
Suppose $(\D,\Sigma,\Delta)$ has a $\sgn$-equivariant Serre structure   $(\se,\sse,\kappa)$.  $\D$ is called:
\begin{itemize}
\item \emph{Calabi-Yau of dimension $n$} (or \emph{$n$-CY}) if 
$(\se,\sse)\cong(\Sigma,-\id_{\Sigma^2})^n$ in $\Z\aCat$;
\item \emph{(fractional) Calabi-Yau of dimension $N/m$} (or \emph{$N/m$-fCY}) if 
$(\se,\sse)^m\cong(\Sigma,-\id_{\Sigma^2})^N$ in $\Z\aCat$.  
\end{itemize}
\end{definition}
We repeat the standard warning with this definition: $N/m$ should be treated as a pair of integers and not as a rational number.
\begin{example}
The derived category of the path algebra of a quiver of type $A_3$ is fractional Calabi-Yau of dimension $2/4$, but is \emph{not} fractional Calabi-Yau of dimension $1/2$.
\end{example}

\subsection{Auslander-Reiten functors}
We now restrict to $G=\Z$.  Hereafter, $F$ will denote an element of $\Autom(\D)$ instead of a map $G\to\Autom(\D)$: see Remark \ref{rmk:Zequivar} for more details.

\subsubsection{Change of action}\label{sss:changeaction}

Let $(\D,\gen{F})$ be a $\Z$-equivariant category, so $\D$ is $\kk$-linear and $F:\D\arr\sim\D$ is an isomorphism of categories.
Let $(Z,z)\in \centre\Autom(\D)$ (see Section \ref{sss:centre}), so $Z:\D\arr\sim\D$ is an isomorphism of categories and $z=(z_E:ZE\arr\sim EZ)_{E\in\Autom(\D)}$ is a natural isomorphism. 
 
We write
\[\M_{\D,F}=\End_{\Z\aCat}((\D,\gen{F}))\] 
to denote the monoidal category of equivariant endofunctors of $(\D,\gen{F})$, so $(\Phi,\phi)\in \M_{\D,F}$ consists of a functor $\Phi:\D\to\D$ and a natural isomorphism $\phi:\Phi F\arr\sim F\Phi$.
From  $(\Phi,\phi)$ 
we can construct an equivariant endofunctor $(\Phi,\phi')\in\M_{\D,ZF}$ as follows:
\[ \xymatrix @R=10pt {
\Phi ZF \ar@{-->}[rr]^{\phi'} \ar[dr]_{z^{-1}_{\Phi}F} && ZF \Phi \\
  & Z\Phi F \ar[ur]_{Z\phi} &
} \]
Write $(Z,z)_\star(\Phi,\phi)=(\Phi,\phi')=(\Phi,Z\phi \circ z^{-1}_{\Phi}F)$ and, for a morphism $\alpha:(\Phi,\phi)\to(\Psi,\psi)$ of equivariant functors, write $(Z,z)_\star(\alpha)=\alpha$.

\begin{proposition}\label{prop:central-equiv}
$(Z,z)\in\centre\Autom(\D)$ is a strict monoidal isomorphism 
\[ (Z,z)_\star:\M_{\D,F}\to \M_{\D,ZF}.\]
\end{proposition}
\begin{proof}
First we check that $\alpha$ is a morphism of equivariant functors $(Z,z)_\star(\Phi,\phi)\to(Z,z)_\star(\Psi,\psi)$.  The diagram
\[ \xymatrix { 
\Phi ZF \ar[r]^{z^{-1}_{\Phi}F} \ar[d]^{\alpha ZF} & Z\Phi F \ar[r]^{Z\phi} \ar[d]^{Z\alpha F} & ZF \Phi \ar[d]^{ZF\alpha} \\
\Psi ZF \ar[r]^{z^{-1}_{\Psi}F} & Z\Psi F \ar[r]^{Z\psi} & ZF \Psi \\
} \]
commutes by the naturality of $z$ and the equivariance of $\alpha$.  Composition of morphisms in both categories is the same, so we have shown that $(Z,z)_\star$ is a functor.

Next we check that $(Z,z)_\star$ is strict monoidal.  We want to compare $(Z,z)_\star \left(  (\Psi,\psi) \circ (\Phi,\phi)  \right)$ with $\left((Z,z)_\star  (\Psi,\psi) \right) \circ  \left((Z,z)_\star(\Phi,\phi)  \right)$.  Both have functor $\Psi\Phi$.  The commutation transformations are: 
\[ \xymatrix @R=15pt {
\Psi\Phi Z F \ar[rr] \ar[dr] \ar[dd] && ZF\Psi\Phi &&  \Psi\Phi Z F \ar[rr] \ar[dr] \ar[dd] && ZF\Psi\Phi \\
& Z\Psi\Phi F \ar[dr] \ar[ur] & &\text{ and }& & \Psi Z F \Psi  \ar[dr] \ar[ur] & \\
\Psi Z\Phi F \ar[ur] && Z \Psi F \Phi \ar[uu]  &&  \Psi Z \Phi F \ar[ur] && Z\Psi F \Phi \ar[uu] 
} \]
These are equal by the interchange law for natural transformations.

Now we show that $(Z,z)_\star$ is an isomorphism of categories.  Let $Z^-$ be a right adjoint quasi-inverse of $Z$ with counit natural isomorphism $\eta:\id_\D\arr\sim Z^-Z$.
We define an inverse functor which sends the object $(\Phi,\phi')\in(\D,\gen{ZF})$ to $(\Phi,\phi'')\in\M_{\D,ZF}$, where $\phi''$ is defined by:
\[ \xymatrix{
  & \Phi F \ar@{-->}[rr]^{\phi''} \ar[dl]_{\eta\id} && F\Phi & \\
Z^-Z \Phi F \ar[rr]^{Z^-} &&  Z^-\Phi Z F \ar[rr]^{Z^- \varphi'} && Z^-Z F\Phi \ar[ul]_{\eta^{-1}\id}
} \]
This map on objects naturally extends to a functor $\M_{\D,ZF}\to \M_{\D,F}$, and one checks that it is a two-sided inverse to $(Z,z)_\star$ by using standard properties of adjunctions.
\end{proof}

%=====

\subsubsection{The Auslander-Reiten functor}\label{sss:inv-ar}

Fix $\chi:\Z\to\kk^\times$ and let $(\se,\sse,\kappa)$ be a $\chi$-equivariant Serre structure on $(\D,\gen{F})$.  We will concentrate on $\sse^1$ but from now on we will drop the superscript $1$, just writing $\sse:\se F\to F\se$.

By Proposition \ref{prop:serre-centre} we know that $(\se,\zeta)\in\centre\Autom(\D)$.
Therefore, by Lemma \ref{lem:dualincentre}, we have $(\se^-,\zeta^\dagger)\in\centre\Autom(\D)$.  Let's revise the notation: given $\Phi\in\Autom(\D)$, the map $\zeta^\dagger_\Phi$ is defined by:
\[ \zeta^\dagger_\Phi: \se^- \Phi \arrr{\id\eta} \se^- \Phi \se\se^- \arrr{ \se^-\zeta_\Phi^{-1} \se^-} \se^- \se \Phi \se^- \arrr{\id\varepsilon} \Phi \se^-. \]

Recall that a $\chi$-equivariant Serre functor $(\se,\sse)$ for $(\D,F)$ satisfies $\sse=\chi(1)\zeta_F$.  A quasi-inverse of $(\se,\sse)$ in $\Z\aCat$ is given by $(\se^-,\sse^-)$ where $\se^-$ is a quasi-inverse functor to $\se$ and $\sse^-$ is defined by:
\[ \sse^-: \se^- F\arrr{\id\eta} \se^- F\se\se^- \arrr{ \se^-\sse^{-1} \se^-} \se^- \se F\se^- 
\arrr{\varepsilon\id} F\se^- \]
(see \cite[Section 9.1]{asa2}).  
Therefore we have
\[ s^- = \chi(-1)\zeta^\dagger_F. \]
\begin{lemma}\label{lem:invserrecentral}
$(\se^-,\sse^-)\in\centre\Autom(\D,F)$.
\end{lemma}
\begin{proof}
As explained above, $(\se^-,\zeta^\dagger)\in\centre\Autom(\D)$, and $s^-$ only differs from $\zeta^\dagger_F$ by the scalar $\chi(-1)\in\kk^\times$.  So this follows from Proposition \ref{prop:ev-centre}.
\end{proof}
We might call $(\se^-,\sse^-)$ a $\chi$-equivariant \emph{inverse Serre functor} for $(\D,F)$.  
We use it to define another important equivariant functor.
\begin{definition}
The \emph{Auslander-Reiten functor} (or \emph{AR functor}) on $\D$ is the functor $\te:=\se^-\circ F$.  The  \emph{$\gen{F}$-equivariant AR functor} is the equivariant functor
\[ (\te,t):=(\se^-,\sse^-)\circ(F,\chi(1)\id_{F^2}) : (\D,\gen F)\to (\D,\gen F). \]
\end{definition}
\begin{remark}
When $\D=\Db(\kk Q)$, $\te$ is the derived functor of the inverse Auslander-Reiten translate $\tau^-$.  For us, $\te$ is more fundamental than its quasi-inverse.
\end{remark}

Now we apply the change of action isomorphism: we use Proposition \ref{prop:central-equiv} with $(Z,z)=(\se^-,\zeta^\dagger)$. 
\[ (\se^-,\zeta^\dagger)_\star:\M_{\D,F}\to \M_{\D,\te}. \]  
The image of $(\te,t)$ under this isomorphism could be called the $\chi$-equivariant AR functor for $(\D,{\te})$.
\begin{lemma}\label{lem:taucomm1}
The commutation map of the $\chi$-equivariant inverse AR translation for $(\D,{\te})$
 is the identity, i.e., 
\[ (\se^-,\zeta^\dagger)_\star\left( (\te,t) \right) = (\te,\id_{\te^{2}}). \]
\end{lemma}
\begin{proof}
First, by definition, $t:\te F\arr\sim F\te$ is defined as
\[ t:\te F = \se^- FF\arrr{\chi(1)\id} \se^-FF\arrr{\sse^-}F\se^- F = F\te. \]
Using that 
$ s^- = \chi(-1)\zeta^\dagger_F$, we have that $t$ is
\[ t:\te F = \se^- FF\arrr{\zeta_F^\dagger F} F\se^- F = F\te. \]
Let $t'$ denote the $\gen{\te}$-equivariant commutation map, so $(\se^-,\zeta^\dagger)_\star\left( (\te,t) \right) = (\te,t')$.  
Then $t'$ is the composition
\[ t':\te \se^- F\arrr{ (\zeta^\dagger_{\te})^{-1}F } \se^-\te F=\se^-\se^-FF \arrr{ \se^- \zeta^\dagger_F  F }  \se^-F\se^-F = \se^- F \te. \]
We want to understand the maps in this composition.

From the definition of $\zeta^\dagger$ we have:
\[\xymatrix{
\se^- \te \ar[rr]^{\zeta^\dagger_{\te} }\ar[d]_{\se^- \te \eta} &&  \te \se^- \\
\se^-\te \se\se^- \ar[rr]^{\se^- (\zeta_{\te} )^{-1} \se^-} && \se^-\se\te \se^- \ar[u]^{\varepsilon \te \se^-}
}\]
By Proposition \ref{prop:serre-centre} we know
\[\xymatrix{
\se\se^- F  =\se\te \ar[rr]^{\zeta_{\te}}\ar[dr]_{\zeta_{\se^-} F } && \te\se= \se^- F \se \\
& \se^-\se F  \ar[ur]_{\se^-\zeta_{ F }} &
}\]
and by Lemma \ref{lem:commutewithqinv} we know that $\zeta_{\se^-}$ is the following composition:
\[ \zeta_{\se^-}: \se\se^-\arrr{\eta^{-1}}\id_\C \arrr{\varepsilon^{-1}} \se^-\se \]
So we get:
\[\xymatrix{
\se^- F \se \ar[rr]^{(\zeta_{\te})^{-1}} \ar[d]^{\se^-(\zeta_{ F })^{-1}}  && \se\se^- F  \\
\se^-\se F  \ar[rr]^{\varepsilon F } &&  F  \ar[u]_{\eta F } 
}\]
Putting the last four diagrams together gives:
\[\xymatrix @C=60pt {
\se^- \se^- F \ar[rrr]^{\zeta^\dagger_{\te} }\ar[d]_{\se^- \se^- F \eta} &&&  \se^- F \se^- \\
\se^-\se^- F \se\se^- 
 \ar[r] ^{\se^- \se^-(\zeta_{ F })^{-1} \se^-}
 & \se^-  \se^-\se F \se^-  \ar[r] ^{\se^- {\varepsilon F } \se^-}
 & \se^- F  \se^-        \ar[r] ^{\se^- {\eta F }   \se^-} \ar[ur]^=
& \se^-\se\se^- F \se^- \ar[u]_{\varepsilon \se^- F \se^-}
}\]
so $\zeta^\dagger_{\te}=\se^-\zeta^\dagger_F$.  Therefore $t':\te\te\to \te\te$ is the identity natural transformation.
\end{proof}

%=====

\subsubsection{Transfer of equivariant Serre structures}

We now want to compare Serre structures on the equivariant categories $(\D,F)$ and $(\D,\te)$.  Recall that we have natural transformations
$\zeta_{\se^-}:\se\se^-\arr\sim \se^-\se$ 
and $\zeta^\dagger_\se: \se^-\se\arr\sim\se\se^-$. 
\begin{lemma}\label{lem:zetainv}
$\zeta^\dagger_\se=(\zeta_{\se^-})^{-1}$.
\end{lemma}
\begin{proof}
Our quasi-inverse $\se^-$ is both left and right adjoint to the Serre functor $\se$, so this follows from Lemma \ref{lem:inversebraid}, using Lemma \ref{lem:zetaunit}.
\end{proof}

Recall the isomorphism of monoidal categories
\[ (\se^-,\zeta^\dagger)_\star:\M_{\D,F}\to \M_{\D,\te}. \]  
described in Sections \ref{sss:inv-ar} and \ref{sss:changeaction}.
Fix a character $\chi:\Z\to\kk^\times$.  
\begin{proposition}\label{prop:transfer-serre}
Given two natural isomorphisms
\[ \sse:\se F\arr\sim F\se \;\; \text{ and } \;\; \sse':\se \te\arr\sim \te\se, \]
any two of the following statements imply the third:  
\begin{enumerate}[label=(\alph*)]
\item $(\se^-,\zeta^\dagger)_\star(\se,\sse)=(\se,\sse')$;
\item $(\se,\sse,\kappa)$ is a {$\chi$-equivariant Serre structure} for $(\D,F)$;
\item $(\se,\sse',\kappa)$ is a {$\chi$-equivariant Serre structure} for $(\D,\te)$;
\end{enumerate}
\end{proposition}
\begin{proof}
The assumptions ensure that $(\se,\kappa)$ is a Serre structure on $\D$, so this reduces to proving the ``two implies three'' property for the following list of statements:
\begin{enumerate}[label=(\alph*)]
\item $\sse'=(\se^-\sse)\circ ((\zeta^\dagger_\se)^{-1} F)$.
\item $\sse=\chi(1)\zeta_F$;
\item $\sse'=\chi(1)\zeta_{\te}$;
\end{enumerate}
By Proposition  \ref{prop:serre-centre} we know that $\zeta_{\te}=(\se^-\zeta_F)\circ(\zeta_{\se^-}F)$, so by  Lemma \ref{lem:zetainv} we have
\[ \zeta_{\te}=(\se^-\zeta_F)\circ((\zeta^\dagger_\se)^{-1} F). \]
So as all maps are isomorphisms this is clear.
\end{proof}

%=====

\subsection{Synthetic Calabi-Yau categories}

Now we prove a result which will give a characterisation of fractional Calabi-Yau triangulated categories.   However, we consider a general equivariant category $(\D,F)$ with Serre duality, as even when $\D$ is triangulated it can be useful to choose an automorphism $F$ other than the automorphism $\Sigma$ coming from the triangulated structure: for example, in Section \ref{ss:hrfa} below, $F$ will be a power of $\Sigma$.  So, in general, we think of these relations as ``synthetic'' (not ``genuine'') fractional Calabi-Yau relations.
\begin{definition}\label{def:syncy}
Let $(\D,F)$ be a $\Z$-equivariant category with $\chi$-equivariant Serre functor $(\se,\sse)$.  We say $(\D,F)$ is \emph{$\chi$-synthetic Calabi-Yau of dimension $N/m$} if 
\[ (\se,\sse)^m\cong(F,\chi(1)\id_{F^2})^N \]
in $\Z\aCat$.  
\end{definition}

Let  $(\D,F)$ be a $\Z$-equivariant category and let $\te=\se^- F$.  Let  $\chi:\Z\to\kk^\times$ be a character and let $0\leq N\leq m$.
\begin{theorem}\label{thm:scy-criteria}
Let $k=m-N$.  The following are equivalent:
\begin{enumerate}[label=(\alph*)]
\item\label{scy:SF} 
$(\D,F)$ is synthetic Calabi-Yau of dimension $N/m$;
\item\label{scy:ST} the $\Z$-equivariant category $(\D,\te)$ has a $\chi$-equivariant Serre functor $(\se,\sse')$ and $$(\se,\sse')^{k}\cong(\te,\id_{\te^{2}})^N;$$
\item\label{scy:TF} the $\Z$-equivariant category $(\D,F)$ has a $\chi$-equivariant Serre functor $(\se,\sse)$ and 
$$(\te,\zeta^\dagger_FF)^m\cong(F,\chi(1)\id_{F^2})^k.$$
\end{enumerate}
\end{theorem}
\begin{proof}
``$\ref{scy:SF}\then\ref{scy:ST}$'':
Suppose $(\D,F)$ has a
 $\chi$-equivariant Serre functor $(\se,\sse)$ and
we have an isomorphism of equivariant functors 
\[ \alpha:(\se,\sse)^m\arr\sim(F,\chi(1)\id_{F^2})^N \]
 in $\Z\aCat$.  
 This gives 
\[(\se^-,\sse^-)^N(\se,\sse)^m  \arr{\sim}  (\se^-,\sse^-)^N(F,\chi(1)\id_{F^2})^N.\]
As $\se$ is an equivalence we have an isomorphism 
\[ (\id,\id)\arr\sim (\se^-,\sse^-)(\se,\sse). \]
By Lemma \ref{lem:invserrecentral} we have an isomorphism
\[ (\se^-,\sse^-)(F,\chi(1)\id_{F^2})\arr\sim  (F,\chi(1)\id_{F^2})(\se^-,\sse^-)\]
which we use repeatedly.
Composing these gives an isomorphism
\[(\se,\sse)^{m-N}\arr{\sim}(\te,t)^N\]
of equivariant endofunctors of $(\D,F)$.   Applying the change of structure isomorphism $(\se^-,\zeta^\dagger)_\star$ and using Lemma \ref{lem:taucomm1}
gives an isomorphism
\[ (\se,\sse')^{m-N}\arr\sim(\te,\id_{\te^{2}})^N \]
and we know $(\se,\sse')$ is $\chi$-equivariant by Proposition \ref{prop:transfer-serre}.

``$\ref{scy:ST}\then\ref{scy:TF}$'':
Suppose we have
\[ (\se,\sse')^{k}\arr\sim(\te,\id_{\te^{2}})^N \]
where $(\se,\sse')$ is $\chi$-equivariant for $(\D,\te)$.  Applying $(\te,\id_{\te^2})^k$ to the right and using the methods above gives an isomorphism
\[ \left( (\se,\sse') (\te,\id_{\te^2})  \right)^{k}\arr\sim(\te,\id_{\te^{2}})^{N+k} \]
As $\te=\se^-F$, we have
\[ (\se,\sse')(\te,\id_{\te^{2}}) = ( \se\te,   \;  \se\te\te 
\arrr{\sse'\te} \te\se\te ) \]
so using the isomorphism
\[ F\arr{\eta F} \se\se^- F=\se \te \]
we get
\[ (\se,\sse')(\te,\id_{\te^{2}}) \cong ( F,   \;  F\te =F\se^- F\arrr{\eta F\se^- F} \se\se^- F\se^- F=\se \te\se^- F \arrr{\sse'  \se^- F}   \te \se\se^- F  \arrr{\te\eta^{-1} F} \te F    ) \]
Now apply $(\se,\zeta)_\star$.  The map $F\te\to\te F $ above is sent to:
\[ F \se \se^- F \arrr{{\zeta_F}^{-1} \se^- F}  \se F \se^- F \arrr{\se \eta F \se^- F}  \se\se\se^- F \se^- F \arrr{ \se \sse' F \se^- F }  \se\se^- F\se \se^- F  \arrr{\se\se^- F\eta^{-1} F} \se\se^- F F
\]
To get an endomorphism in $(\D,F)$ we use $\eta$ and $\eta^{-1}$.  Using Proposition \ref{prop:transfer-serre} and the zigzag equations, this leaves us with $\chi(1)\id_{F^2}:FF\to FF$.  A similar calculation shows that
 \[  (\se,\zeta)_\star\left(  (\te,\id_{\te^{2}}) \right) = \se^- F \se \se^- F \arrr{\zeta_{\te}^{-1}\se^- F} 
\se \se^- F  \se^- F \]
and so using $\eta$ and $\eta^{-1}$ leaves us with $\zeta^\dagger_FF:\se^-FF\to F\se^- F$.

``$\ref{scy:TF}\then\ref{scy:SF}$'':
An isomorphism
$$(\te,\zeta^\dagger_FF)^m\arr\sim(F,\chi(1)\id_{F^2})^k$$
gives an isomorphism
$$(F^-,\chi(-1)\id_{F^2}^-)^k \arr\sim  (\te^-,(\zeta^\dagger_FF)^-)^m.$$
We compose with $(F,\chi(1)\id_{F^2})^m$ on the right to get
\[ (F,\chi(1)\id_{F^2})^{m-k}\arr\sim (\se,\sse)^m. \]
\end{proof}

Now we move between equivariant and hom-graded categories.

\begin{theorem}\label{thm:fcy-orbit}
Let $k=m-N$.  The following are equivalent:
\begin{enumerate}[label=(\alph*)]
\item\label{it:se-tau} the $\Z$-equivariant category $(\D,\te)$ has a $\chi$-equivariant Serre functor $(\se,\sse')$ and $$(\se,\sse')^{k}\cong(\te,\id_{\te^{2}})^N;$$
\item\label{it:se-gr} the orbit category $\C=(\D,\te)/\Z$ has a $\chi$-hom-graded Serre functor $(\se^\C,\ell)$ and $$(\se^\C,\ell)^{k}\cong(\id_\C,\underline{N}).$$
\end{enumerate}
Moreover, if
$\C$ has finitely many isoclasses of indecomposable objects and $A$ denotes the $\Z$-graded base algebra of $\C$, then \ref{it:se-gr} is equivalent to:
\begin{enumerate}[label=(\alph*)]\setcounter{enumi}{2}
\item\label{it:nak-alg} $A$ is a $\chi$-graded Frobenius algebra with Nakayama automorphism $(\alpha,\ell)$ satisfying $$(\alpha,\ell)^{k}\cong(\id_A,\underline{N}).$$
\end{enumerate}
\end{theorem}
\begin{proof}
``$\ref{it:se-tau}\iff\ref{it:se-gr}$'':
Given an isomorphism $(\se,\sse')^{k}\arr\sim(\te,\id_{\te^{2}})^N$,
apply the orbit 2-functor $-/G$.  By Proposition \ref{prop:orbiting-serre}, $(\se,\sse')/G$ is a $\chi$-hom-graded Serre functor and, by Lemma \ref{lem:gshift}, $((\te,\id_{\te^{2}})^N)/G$ is isomorphic to $(\id_\C,\underline{N})$.  By Theorems \ref{thm:asa2eq}
and \ref{thm:transfersss}
these steps are reversible.

Now suppose $\C$ has finitely many isoclasses of indecomposable objects.
\\``$\ref{it:se-gr}\iff\ref{it:nak-alg}$'': Use Theorem \ref{thm:graded-ss-base}.
\end{proof}

\begin{corollary}\label{thm:arcyalg}
Let $(\D,\Sigma,\Delta)$ be a triangulated category and let $\C=\D/\te$ be its orbit category by $\te$.   Suppose $\C$ has finitely many isoclasses of indecomposable objects and let $A$ be the $\Z$-graded base algebra of $\C$.  Then the following are equivalent:
\begin{enumerate}[label=(\alph*)]
\item\label{it:fcy-cat} $(\D,\Sigma,\Delta)$ is fractional Calabi-Yau of dimension $N/m$;
\item\label{it:fcy-alg} $A$ is a $\sgn$-graded Frobenius algebra with Nakayama automorphism $(\alpha,\ell)$ satisfying $$(\alpha,\ell)^{m-N}\cong(\id_A,\underline{N}).$$
\end{enumerate}
\end{corollary}

% ----------

\section{Applications}\label{s:apps}

\subsection{Derived categories}

Let $A,B,C$ be $\kk$-algebras.  Given chain complexes $X$ of $C\da B$-bimodules and $Y$ of $B\da A$-bimodules, we get a chain complex $X\otimes_B Y$ of $C \da A$-bimodules defined as follows.  Its degree $n$ term is
\[ (X\otimes_B Y)^n=\bigoplus_{i+j=n}X^i\otimes_B Y^j \]
and its differential $d: (X\otimes_B Y)^n\to  (X\otimes_B Y)^{n+1}$ on the summand $X^i\otimes_B Y^j $ is $d_X\otimes \id+ (-1)^i\id\otimes d_Y$.  

Given an algebra $\Lambda$, let $\Db(\Lambda)$ denote the bounded derived category of left $\Lambda$-modules.  This is a triangulated category with automorphism $\Sigma$ which shifts cochain complexes one place to the left and changes the sign of all differentials.  Similarly, let $\Db(\Gamma\da\Lambda)$ denote the bounded derived category of $\Gamma\da\Lambda$-bimodules.

Let $\fBim$ denote the following bicategory:
\begin{itemize}
\item The objects are finite-dimensional algebras $\Lambda, \Gamma$ of finite global dimension.
\item The 1-cells $X,Y:\Lambda\to \Gamma$ are objects of $\Db(\Gamma\da\Lambda)$.
\item The 2-cells $f:X\to Y$ are morphisms in $\Db(\Gamma\da\Lambda)$.
\end{itemize}
Composition of 1-cells is by derived tensor product.

Recall the 2-category $\Tri$ of triangulated categories from Section \ref{ss:cy}.  We have a weak 2-functor
\[ \twofun{D}:\fBim\to\Tri \]
which sends the object $\Lambda\in\Tri$ to $\Db(\Lambda)$.  On 1-cells, it sends $X\in\Db(\Gamma\da\Lambda)$ to the triangulated functor $(\Phi, \phi)$ where $\Phi=X\dert_\Lambda-$ is the left derived tensor product.  If $M\in\Db(\Lambda)$ is a bounded complex of projective left $\Lambda$-modules then the commutation isomorphism
\[ \phi_M: X\otimes_\Lambda (\Sigma M)\arr\sim \Sigma(X\otimes_\Lambda M) \]
is given by summing the following maps:
\[ X^i\otimes_\Lambda M^j \arrr {(-1)^i} X^i\otimes_\Lambda M^j . \]
On 2-cells, it sends $f:X\to Y$ to $f\otimes\id:X\dert_\Lambda-\to Y\dert_\Lambda-$.

\begin{lemma}\label{lem:derff}
The weak 2-functor $\twofun{D}:\fBim\to\Tri$ is locally fully faithful.
\end{lemma}
\begin{proof}
The action of $\twofun{D}$ on 1-cells has an inverse which sends the natural transformation $\alpha:X\dert_\Lambda-\to Y\dert_\Lambda-$ to the morphism 
\[ X\cong X\otimes_\Lambda\Lambda \arrr{\alpha_\Lambda}  Y\otimes_\Lambda\Lambda\cong\Lambda. \]
To see this is a bijection one uses naturality of $\alpha$ for maps $\Lambda\to M^i$.
\end{proof}
So by Proposition \ref{prop:picinjsurj} we get:
\begin{corollary}
We have a group monomorphism
\[ \DPic(\Lambda)\into \PicEnd_\Tri(\Db(\Lambda)). \]
\end{corollary}

Let $\Lambda[1]$ denote the identity bimodule concentrated in degree $-1$. 
Note that both $\Lambda[1]\otimes_\Lambda-$ and $\Sigma$ shift complexes one place to the left and reverse the sign of the differentials, so we get a natural isomorphism $\alpha:\Lambda[1]\otimes_\Lambda-\arr\sim \Sigma$ with components defined by identity maps.

As $\Lambda$ has finite global dimension, the category $\Db(\Lambda)$ has a Serre functor $\Lambda^*\dert_\Lambda-$: see, for example, \cite[Proposition 20.5.5]{gin-nc}.  
Then by Theorem \ref{thm:vdb-tri}, the graded category $(\Db(\Lambda), \Sigma)$ has a graded Serre functor $(\se,\sse)=\twofun{D}\Lambda^*$.

The following result says that our ``strong'' fractional Calabi-Yau condition is equivalent to the ``bimodule'' Calabi-Yau condition.
\begin{theorem}\label{thm:bimfcy}
The triangulated category $\Db(\Lambda)$ is $N/m$-fCY if and only if there exists an isomorphism $$(\Lambda^*)^{\dert m}\arr\sim\Lambda[N]$$ in $\Db(\Lambda\da\Lambda)$.
\end{theorem}
\begin{proof}
From Lemma \ref{lem:derff} and Definition \ref{def:cy}, all that we need to check is that $\twofun{D}(\Lambda[1])\cong (\Sigma,-\id_{\Sigma^2})$.  
We know that $\twofun{D}(\Lambda[1])\cong (\Lambda[1]\otimes_\Lambda,\phi)$ and so we get a diagram
\[ \xymatrix{
\Lambda[1]\otimes_\Lambda\Sigma M \ar[r]^{\phi_M} \ar[d]^{\alpha_{\Sigma M}} & \Sigma(\Lambda[1]\otimes_\Lambda M)\ar[d]^{\Sigma\alpha_M} \\
\Sigma\Sigma M \ar@{-->}[r] & \Sigma\Sigma M 
} \]
and we need to show that the bottom arrow is multiplication by $-1$.  But from the formula for $\phi_M$ above, applied when $X=\Lambda[1]$ which is concentrated in degree $-1$, we get that $\phi_M$ is $-1$.  So, as both maps $\alpha$ are $1$, we are done.
\end{proof}

% =======

\subsection{Dynkin quivers}\label{ss:dynkin}

\subsubsection{Dynkin diagrams}\label{sss:dynkin}

We revise some standard Lie theory, referring to \cite{hum} for details.

The simply laced Dynkin diagrams are those of ``$ADE$ type'': they are simple (unoriented) graphs belonging to the following list: 
$A_n$ for $n\geq1$,  $D_n$  for $n\geq4$, $E_6$, $E_7$, and $E_8$.   The subscript denotes the number of vertices in the graph and is called the \emph{rank}.
For each Dynkin diagram $\Gamma$ we have a graph automorphism $\rho:\Gamma\to\Gamma$ defined on the vertices of $\Gamma$ as follows:
\begin{itemize}
\item {\makebox[3.5cm]{\textbf{Type $A_n$}:\hfill} $\rho$ swaps vertices $i\leftrightarrow n+1-i$.}
\item {\makebox[3.5cm]{\textbf{Type $D_n$, $n$ is even}:\hfill} $\rho$ is the identity.}
\item {\makebox[3.5cm]{\textbf{Type $D_n$, $n$ is odd}:\hfill} $\rho$ swaps vertices $n-1\leftrightarrow n$ and fixes the others. }
\item {\makebox[3.5cm]{\textbf{Type $E_6$}:\hfill} $\rho$ swaps vertices $1\leftrightarrow 5$ and $2\leftrightarrow 4$ and fixes $3$ and $6$.}
\item {\makebox[3.5cm]{\textbf{Types $E_7$ and $E_8$}:\hfill} $\rho$ is the identity.}
\end{itemize}
Notice that $\rho$ has order $1$ or $2$, depending on the type.

Each Dynkin diagram has an associated finite root system $\Phi$ which is the disjoint union $\Phi=\Phi^+ \cup \Phi^-$ of positive and negative roots.  Let $R=\abs{\Phi^+}=\abs{ \Phi^-}$ denote the number of positive roots.  (This is traditionally denoted $N$, but we use this symbol elsewhere.)

Each type has a \emph{Coxeter number}, denoted $h$, defined as the order of the Coxeter element of the associated reflection group.  This is related to the number $R$ of positive roots by the following formula:
\begin{proposition}\label{prop:rankrootsh}
Let $n$, $R$, and $h$ denote the rank, number of positive roots, and Coxeter number of a Dynkin diagram, respectively.  Then $h=2R/n$.
\end{proposition}

We record the relevant information:
\begin{center}
  \begin{tabular}{ c | c | c | c | c | c }
    Type & $A_n$ & $D_n$ & $E_6$ & $E_7$ & $E_8$ \\ \hline
    $R$ & $n(n+1)/2$ & $n(n-1)$ & $36$ & $63$ & $120$ \\ \hline
    $h$ & $n+1$ & $2n-2$ & $12$ & $18$ & $30$ 
  \end{tabular}
\end{center}

\subsubsection{Path algebras}

Let $\kk$ be an algebraically closed field.  We refer to \cite{grant-nak} for more details and for references.

Let $Q$ be an acyclic quiver and let $\Lambda=\kk Q$ denote its path algebra.
Let $Q_0$ denote the vertices of $Q$, then we have a canonical bijection $Q_0\arr\sim \prim(\Lambda)$ sending the vertex $i$ to the length zero path $e_i$.  Write $P_i=\Lambda e_i$.  The left $\Lambda$-modules $P_i$, $i\in Q_0$, form a complete list of indecomposable projective $\Lambda$-modules up to isomorphism.

We say that $Q$ is a  \emph{Dynkin quiver} if the underlying graph of $Q$ is of $ADE$ type.
Let $\Lambda\mMod$ denote the category of 
finite-dimensional left $\Lambda$-modules. 
\begin{theorem}[Gabriel]\label{thm:gabriel}
The category $\kk Q\mMod$ has finitely many isomorphism classes of indecomposable objects if and only if $Q$ is a Dynkin quiver.  In this case, the set of indecomposable objects up to isomorphism is in canonical bijection with the positive roots $\Phi^+$ of the Dynkin diagram.
\end{theorem}

Let $\tau$ and $\tau^-$ denote the classical Auslander-Reiten translate and its inverse.
\begin{theorem}[Platzeck-Auslander, Gabriel]\label{thm:pa-g}
If $Q$ is Dynkin then every indecomposable $\Lambda$-module is isomorphic to $\tau^{-p}P_i$ for some $i\in Q_0$ and $p\geq0$.
\end{theorem}

The algebra $\Lambda=\kk Q$ is hereditary, so its global dimension is $\leq1$.  It is exactly $1$ if $Q$ has at least one arrow.  

Let  $\D=\Db(\Lambda)$.  
As $\Lambda$ is hereditary, every indecomposable $X\in\D$ is of the form $\Sigma^nM$ where $n\in\Z$ and $M\in\Lambda\mMod\into\D$ under the embedding taking a module to a stalk complex in degree $0$.  The derived functor of $\tau^-$ is $\se^-\Sigma$, so using the same notation as with the module category, we have $\tau^-(e_i\Lambda)^*\cong \Sigma P_i$.
Thus, by Theorem \ref{thm:pa-g},
\[ \ind\D \cong \gen{ \tau^{-p}P_i\st i\in Q_0, p\in\Z} \subseteq\D.\]

The following result was suggested by calculations of Gabriel which describe the action of $\se$ on $\ob\Db(\kk Q)$ \cite[Section 6.5]{gab-ar}.  It was proved by Miyachi and Yekutieli \cite[Theorem 4.1]{my}.
\begin{theorem}\label{thm:my}
If $Q$ is a Dynkin quiver then $\Db(\kk Q)$ is fractional Calabi-Yau of dimension $(h-2)/h$.  Moreover, if $\rho$ has order $1$ then $\Db(\kk Q)$ is fractional Calabi-Yau of dimension $(h/2-1)/(h/2)$.  
\end{theorem}

\subsubsection{Preprojective algebras}

The preprojective algebra $\Pi(Q)$ of the quiver $Q$ was defined classically using generators and relations.  We have
 \[ \Pi(Q)=\kk \overline{Q}/I \]
where $\overline{Q}$ is the \emph{doubled} quiver of $Q$ and $I$ is an ideal of relations.  For example, 
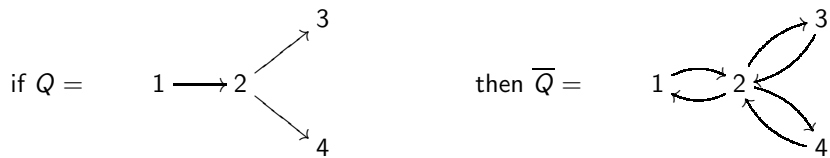
\[ 
\xymatrix @R=10pt @C=20pt {
&&&3
&&
&&&3 \ar@/^/[dl] \\
{\text{ if } Q=} &1\ar[r] &2\ar[ur]\ar[dr] & 
&&
{\text{ then } \overline{Q}=} &1\ar@/^/[r] &2\ar@/^/[ur]\ar@/^/[dr]\ar@/^/[l] & 
\\
&&& 4
&&
&&& 4 \ar@/^/[ul]
}
\]
where for each arrow $a:i\to j$ in $Q$ we added an arrow $a^*:j\to i$ in $\overline{Q}$.  The ideal of relations is generated by the sum $\sum (aa^*-a^*a)$ over all arrows of $Q$.

Note that the graph automorphism $\rho$ from Section \ref{sss:dynkin} induces an automorphism $\overline{\rho}$ of the quiver $\overline{Q}$.

There are two natural $\Z$-gradings on $\Pi(Q)$.  The first is the path length grading, where all arrows $a$ and $a^*$ have degree $1$.  The other is the tensor grading, where arrows $a$ in $Q$ have degree $0$ and arrows $a^*$ in $Q^*$ have degree $1$.  Note that the ideal of relations is homogeneous of degree 2 with respect to the path length grading, and of degree 1 with respect to the tensor grading.

Given $x\in\Pi(Q)$ we write $\deg(p)$ for the tensor grading of $Q$.
Recall that $\sgn$ denotes the character $\Z\to\kk^\times$ sending $1\in\Z$ to $-1\in\kk$.

Recall $\Lambda=\kk Q$.  
It is known that $\Pi(Q)$ is isomorphic to the following algebra:
\[ \Pi(\Lambda)= \bigoplus_{p\geq0}\Hom_\Lambda(\Lambda,\tau^{-p}\Lambda). \]
Here we see a natural grading $\Pi(\Lambda)_p=\Hom_\Lambda(\Lambda,\tau^{-p}\Lambda)$ which corresponds to the tensor grading on $\Pi(Q)$.

If $Q$ is Dynkin then 
up to isomorphism, $\Pi(Q)$ does not depend on the orientation of $Q$ \cite[Lemma 4.1]{bbk}.  
Building on the calculations of Gabriel mentioned above, the following result was proved by Brenner, Butler, and King 
\cite[Theorem 4.8]{bbk}:
\begin{theorem}\label{thm:bbk}
If $Q$ is Dynkin then $\Pi(Q)$ is a Frobenius algebra.  
If $1\neq-1$ in $\kk$ and $Q$ is not of type $A_1$, then its Nakayama automorphism has order exactly 2.
\end{theorem}
In fact they describe the Nakayama automorphism $\beta$ of $\Pi(Q)$ explicitly: it acts as $\rho$ on the vertices of $\overline{Q}$ and asks on arrows by:
\[ \beta(a) = \overline{\rho}(a) \]
\[ \beta(a^*) = \sgn(\deg( \overline{\rho}(a^*) )) \overline{\rho}(a^*) \]

\subsubsection{Comparison of results}

Let $\D=\Db(\kk Q)$.  Note that 
\[ \Pi(\Lambda)\cong \bigoplus_{p\geq0}\Hom_\D(\Lambda,\tau^{-p}\Lambda) =\Hom_{\D/\tau^-}(\Lambda,\tau^{-p}\Lambda) \]
so $\Pi(\Lambda)$, with the tensor grading, is the graded base algebra of the orbit category $\D/\tau^-$.

We will show that Theorems \ref{thm:my} and \ref{thm:bbk} are closely related.

\begin{proof}[Proof of Theorem \ref{thm:bbk} from Theorem \ref{thm:my}]
As $\D$ has a Serre functor, so does $\C=\D/\tau^-$ by Proposition \ref{prop:orbiting-serre}, and its base algebra $A_\C\cong\Pi$ is graded Frobenius by Theorem \ref{thm:graded-ss-base}.  By Theorems \ref{thm:scy-criteria} and \ref{thm:fcy-orbit}, the graded Nakayama automorphism is of the form $(\alpha,\ell)$ with $\alpha^2=\id$.  We have $\beta^{\sgn}=\alpha$ by Lemma \ref{lem:trgr-chigr}.  In the cases where $\rho=\id$ we have $\alpha=\id$, so we must have $\beta=\alpha^{\sgn}\neq1$.
\end{proof}

\begin{proof}[Proof of Theorem \ref{thm:my} from Theorem \ref{thm:bbk}]
The Serre functor $(\se,\ell)$ on $\C=\D/\tau^-$ induces a $\sgn$-hom-graded Nakayama functor $(\alpha,\ell)$ on the graded algebra $\Pi$.  We have $\beta^{\sgn}=\alpha$ by Lemma \ref{lem:trgr-chigr}.  Note that $\alpha^2=\beta^2$, and we know $\beta^2=\id$, so $\alpha^2=\id$.  Therefore we have $(\se,\ell)^2=(\id,\gamma)$ on $\C$, where $\gamma:\ob\C\to\Z$.  

We claim that $\gamma$ is a constant function.  Write $m_i=\gamma(i)+2$.
By Theorem \ref{thm:scy-criteria} we have that, for each $P_i$, $\tau^{-m_i}P_i\cong\Sigma^2 P_i$.  Now, following an argument from \cite[Section 4.1]{hi-frac}, suppose $\Hom_\Lambda(P_i,P_j)\neq0$.  Then apply $\se^{m_im_j}=\se^{m_jm_i}$: we see $\se^{m_im_j}P_i=\Sigma^{2m_j}P_i$ and $\Sigma^{2m_i}P_j$.  As there is a nonzero map between them, they must be concentrated in the same degree.  Hence $m_i=m_j$ so, as $Q$ is connected, $\gamma$ is constant.  By Theorem \ref{thm:gabriel} $\Lambda$ has $R$ indecomposable modules, so as $\tau^{-m}\cong \Sigma^2$ we must have $nm=2R$.  So, by Proposition \ref{prop:rankrootsh}, we have $m=h$.  So by Theorem \ref{thm:scy-criteria}, $\D$ is fractional Calabi-Yau of dimension $(h-2)/h$.

For the cases with $\rho=\id$ we have that $\beta$ acts as $-1$ on the elements of $\Pi$ of (tensor) degree $1$, so $\alpha$ acts as $\id$ and so, arguing as above, in these cases $\D$ is fractional Calabi-Yau of dimension $((h/2)-1)/(h/2)$.
\end{proof}

% ----------

\subsection{Higher representation finite algebras}\label{ss:hrfa}

\subsubsection{Higher preprojective algebras}

Fix an integer $0\leq d<\infty$ and let $\Lambda$ be an algebra of global dimension $\leq d$.  
Let $\Db(\Lambda)$ be the bounded derived category of finitely generated left $\Lambda$-modules with shift functor $\Sigma$ and inverse shift functor $\Sigma^-$.  So $\Db(\Lambda)$ has a Serre functor $\se$.  Define $\se_d^{-}:=\Sigma^d\circ\se^-$.

Following \cite{io-stab}, we consider the following full subcategory of $\Db(\Lambda)$:
\[ \U = \gen{ \se_d^{-p}\Lambda \st p\in\Z }. \]
Then $\U$ is a $d$-cluster tilting subcategory of $\Db(\Lambda)$ in the sense of Iyama \cite[Section 3]{iy}.

The algebra $\Lambda$ is called \emph{$d$-representation finite} if the category $\Lambda\mMod$ contains a $d$-cluster tilting object.  We won't elaborate on what that means here; a definition can be found in \cite[Definition 0.1]{hi-frac}.  Instead, we make use of the following result of Iyama and Oppermann \cite[Theorem 3.1]{io-stab}:
\begin{theorem}[Iyama-Oppermann]\label{thm:io-seu}
$\Lambda$ is {$d$-representation finite} if and only if ${\se}\U=\U$.
\end{theorem}
Note that, by Lemma \ref{lem:serreonsubcat}, $\se\U=\U$ implies that $\U$ has Serre duality.
Note also that, by definition of $\U$, if $\se\U=\U$ then we also have $\Sigma^d\U=\U$.  So if 
$\Lambda$ is {$d$-representation finite}
 we get a $\Z$-equivariant category
\[ (\D,F):=(\U,\Sigma^d).  \]
Note that $\te=\se_d^-$.

Set $\C=\U/\se_d^{-}$.  Then the object $\Lambda\in\C$ generates $\C$ (by which we mean the functor $\Mat\ic \gen{\Lambda}\into \C$ is an equivalence of graded categories, where $\gen{\Lambda}$ denotes the full subcategory of $\C$ on the object $\Lambda$).  

We may assume without loss of generality that $\Lambda$ is basic.
Note that, by definition, $\C$ has as many objects as there are summands of $\Lambda$, which is finite because $\Lambda$ is finite-dimensional.
\begin{definition}
The \emph{preprojective algebra} of $\Lambda$, denoted $\Pi$, is the $\Z$-graded base algebra of $\C$. 
\end{definition}
Note that this agrees with the usual \cite{bgl,io-stab} definition of the higher $(d+1)$-preprojective algebra, with its tensor grading.  We usually omit the word \emph{higher}.
\begin{theorem}[Iyama-Oppermann]\label{thm:io-frob}
If $\Lambda$ is basic and {$d$-representation finite} then $\Pi$ is Frobenius.
\end{theorem}
\begin{proof}
By Theorem \ref{thm:io-seu}, $\U$ has a Serre functor.  So by Theorem \ref{thm:transfersss}
$\C$ has a Serre functor.  So by Theorem \ref{thm:graded-ss-base}  $\Pi$ is Frobenius.
\end{proof}

Recall Definition \ref{def:syncy}.  Together, Theorems \ref{thm:scy-criteria} and \ref{thm:fcy-orbit} immediately give the following:
\begin{proposition}
$(\U,\Sigma^d)$ is $\chi$-synthetic Calabi-Yau of dimension $N/m$ if and only if the $\chi$-graded Nakayama automorphism $(\alpha,\ell)$ of $\Pi$ satisfies $(\alpha,\ell)^{m-N}\cong(\id,\underline{N})$.
\end{proposition}

We can take powers of characters.  In particular, $(\sgn)^d$ sends $1\in\Z$ to $(-1)^d\in\kk^\times$.
\begin{proposition}\label{prop:u-fcy}
$\Db(\Lambda)$ is $dp/q$-fCY if and only if $(\U,\Sigma^d)$ is $(\sgn)^d$-synthetic $p/q$-CY.
\end{proposition}
\begin{proof}
Suppose $\Db(\Lambda)$ is $dp/q$-fCY, so the $\sgn$-equivariant Serre functor $(\se,\sse)$ satisfies $(\se,\sse)^q=(\Sigma,-\id_{\Sigma^2})^{dp}$.  So $(\se,\sse)^q=(\Sigma^d,(-1)^d\id_{\Sigma^{d+1}})^{p}$.  We now take $d$th powers of the commutator natural transformations to move from the $\Z$-graded category $(\Db(\Lambda),\Sigma)$ to $(\Db(\Lambda),\Sigma^d)$: this gives  $(\se,\sse^d)^q=(\Sigma^d,(-1)^{d^2}\id_{\Sigma^{2d}})^{p}$.  Note that $s^d:\se\Sigma^d\arr\sim\Sigma^d\se$ makes $(\se,\sse^d)$ a $(\sgn)^d$-equivariant Serre functor, and $(\sgn)^d(1)=(-1)^{d^2}=(-1)^{d}$, so this says that $(\Db(\Lambda),\Sigma^d)$ is $(\sgn)^d$-synthetic $p/q$-CY.  Thus the equivariant subcategory $(\U,\Sigma^d)$ is also $(\sgn)^d$-synthetic $p/q$-CY.

Now suppose $\U$ is $p/q$-fCY.
Note that $\Lambda=\se_d^0\Lambda\in\U$, so applying $\se^q\cong(\Sigma^d)^p$ to $\Lambda$ we get an isomorphism $\se^q\Lambda\cong \Lambda[dp]$ in $\U\subseteq\Db(\Lambda)$.  By a standard argument (e.g, see \cite[Lemma 4.2]{hi-frac}) this implies that we have a morphism of endofunctors of $\Db(\Lambda)$ given by $\se^q\cong \Lambda_\sigma[dp]\otimes_\Lambda-$ for some $\sigma\in\Aut(\Lambda)$.  So $\se^q$ is given by tensoring with some $\Lambda\da\Lambda$-bimodule $\Lambda_\sigma$ and shifting by $[dp]$.  But now as $\Lambda\in\U$ is concentrated in a single degree, and the right $\Lambda$-module structure corresponds to endomorphisms of $\Lambda$, the naturality of our isomorphism shows that our left module isomorphism $\Lambda\cong\Lambda_\sigma$
is in fact a morphism of $\Lambda\da\Lambda$-bimodules.  Hence we have an isomorphism $\se^q\Lambda\cong \Lambda[dp]$ of $\Lambda\da\Lambda$-bimodules and so, by Theorem \ref{thm:bimfcy}, 
$\Db(\Lambda)$ is $dp/q$-fCY.
\end{proof}
Note that $(\U,\Sigma^d)$ is a $(d+2)$-angulated category \cite{gko}, and the $(\sgn)^d$ character in the above theorem ensures our graded Serre functor preserves $(d+2)$-angles \cite[Theorem 3.3]{zhou}.

From Theorem \ref{thm:graded-ss-base} and Proposition \ref{prop:u-fcy} we get:
\begin{theorem}\label{thm:drf-fcy}
Let $\Lambda$ be a $d$-representation finite algebra and $\Pi$ its higher preprojective algebra.  Then $\Db(\Lambda)$ is fractional Calabi-Yau of dimension $dN/m$ if and only if the $(\sgn)^d$-graded Nakayama automorphism $(\alpha,\ell)$ of $\Pi$ satisfies $(\alpha,\ell)^{m-N}\cong(\id,\underline{N})$.
\end{theorem}
This theorem is useful because we have examples of $d$-representation finite algebras where the Nakayama automorphism of the higher preprojective algebra is known.  In Section \ref{sss:cuts} below we describe one such case.  It also provides a link between existing results in the literature, as we explain in Section \ref{sss:highertypea}.

Recall that a finite-dimensional algebra is \emph{connected} if it cannot be written as a non-trivial product of algebras; equivalently, its Gabriel quiver is connected.  The following simple result is quite useful.
\begin{lemma}\label{lem:constda}
Let $A$ be a $\Z$-graded, connected finite-dimensional algebra.  Let $(\alpha,\ell)$ be a degree-adjusted automorphism of $A$.  If $\alpha^k=\id$ then there exists $N\in\Z$ such that $(\alpha,\ell)^k=(\id,\underline{N})$.
\end{lemma}
Lemma \ref{lem:constda} can be useful in practice to calculate explicit Calabi-Yau dimensions.  Combined with Theorem \ref{thm:drf-fcy}, it also gives a theoretical result.  We say an algebra is fractional Calabi-Yau if it is $p/q$-fCY for some $p,q\in\Z$.
\begin{corollary}\label{cor:fcyfinitenak}
Let $\Lambda$ be a $d$-representation finite algebra and let $(\alpha,\ell)$ denote the $\sgn$-graded Nakayama automorphism of its higher preprojective algebra.  Then $\alpha$ has finite order if and only if $\Lambda$ is fractional Calabi-Yau.
\end{corollary}

In \cite[Remark 1.6]{hi-frac}, Herschend and Iyama ask:  is every $d$-representation-finite algebra fractionally Calabi-Yau?  Keeping the notation above, we translate this into the following:
\begin{question}
Does $\alpha$ have finite order?
\end{question}

%======

\subsubsection{Higher type A algebras}\label{sss:highertypea}

Recall the higher type A preprojective algebras as studied in \cite{iya-ct,io-napr}.
In type $A^d_s$, 
they are given as $kQ/I$ where $Q$ has vertex set:
\[ Q_0=\{x=(x_1,x_2,\ldots,x_{d+1})\in\Z^{d+1}_{\geq0}\st\sum_{i=1}^{d+1}x_i=s-1\} \]
The arrows are of the form
\[ \alpha_{i,x}:x\to x+f_i, \;\; x, x+f_i\in Q_0 \]
for $1\leq i\leq d+1$ 
where 
\[f_1=(-1,1,0,\ldots,0), \;\; f_2=(0,-1,1,0,\ldots,0), \;\; \ldots, \;\; f_d=(0,\ldots,0,-1,1), \;\; f_{d+1}=(1,0,\ldots,0,-1).\]

Define a permutation $\omega_0$ of $Q_0$ by:
\[ \omega_0: (x_1,x_2,\ldots,x_{d+1}) \mapsto (x_{d+1},x_1,\ldots,x_d). \]
Then $\omega_0$ extends to a permutation $\omega=(\omega_0,\omega_1)$ of the quiver $Q$.  It fixes the ideal $I$ in $kQ$.
Let $\sigma$ denote the algebra automorphism induced by $\omega$ on $\Pi^d_s$.  
Note that $\sigma^{d+1}=\id$.
We have:
\begin{theorem}[{\cite[Theorem 3.5]{hi-frac}}]
$\sigma$ is a Nakayama automorphism of the (ungraded) algebra $\Pi$.
\end{theorem}

Let $G=\Z^{d+1}$ with generating set $e_1,\ldots,e_{d+1}$.  The algebra $\Pi$ is graded with arrows $\alpha_{x,i}$ in degree $e_i$.  If $p:x\to y$ is a path from $x$ to $y$ of degree $\delta\in G$, we have the relation $y-x=\sum_i \delta_if_i$: see the proof of {\cite[Theorem 3.5]{hi-frac}}.  In components, this says:
\[ y_i-x_i=\delta_{i-1}-\delta_i \]
where we set $\delta_0=\delta_d$.

Let $\PP^G$ denote the free hom-graded $\kk$-linear category on $Q$ modulo the relations from $I$, so $A_{\PP^G}\cong\Pi$.   Note that the vertices are already elements of $\Z^{d+1}_{\geq0}\subset G$, so we have a natural degree adjuster $n:\ob\PP^G\to G$ sending $x\in Q_0$ to $x\in G$.
\begin{lemma}
$(\sigma,n)$ is a 1-cell $\PP^G\to\PP^G$ in $\GrCat G$.
\end{lemma}
\begin{proof}
We want the following equation to hold:
\[ \deg f = \deg\sigma(f) +n(x)-n(y) \]
If $\deg f=\delta\in G$ then $\deg\sigma(f)=\omega(\delta)$.
So, in degree $i$, this says
\[ \delta_i = \delta_{i-1}+x_i-y_i \]
which is true by the formula given above.
\end{proof}

We have a group homomorphism $\varphi:G\to\Z$ which projects onto the last component, so $\sum \lambda_ie_i$ is sent to $\lambda$.  Therefore we get a $\Z$-graded category $\PP=\varphi_*(\PP^G)$, and $(\sigma,\ell)$ is a 1-cell $\PP\to \PP$, where $\ell(x)=x_{d+1}$ is the last term of $x$.  
\begin{proposition}
$(\sigma,\ell)$ is the $\tr$-graded Nakayama automorphism of $\PP=\PP^d_s$, and we have $(\sigma,\ell)^{d+1}=(\id,\underline{s-1})$.
\end{proposition}

If we cut $\Pi$ at the arrows $f_{d+1}$ (or, equivalently, take the degree $0$ subalgebra with respect to the $\Z$-grading) we get the $d$-representation finite algebra of type $A_s^d$, denoted $\Lambda^d_s$. 

The following result was first proved by Dyckerhoff, Jasso, and Walde \cite[Remark 2.29]{djw}.  Their proof uses the language of $\infty$-categories and a description of the algebras $\Lambda^d_s$ from \cite{jk}.  It was also proved by Dyckerhoff, Jasso, and Lekili using symplectic geometry \cite[Remark 2.5.2]{djl}.
\begin{theorem}\label{thm:typea-fcy}
$\Lambda^d_s$ is fCY of dimension $\displaystyle \frac{d(s-1)}{(s+d)}$.
\end{theorem}
\begin{proof}[Proof of Theorem \ref{thm:typea-fcy}]
We use Proposition \ref{prop:u-fcy}.  If $d$ is even then $(\sgn)^d=\tr$.  If $d$ is odd then we use the fact that we take an even power $\sigma^{d+1}$ together with Lemma \ref{lem:trgr-chigr}.
\end{proof}

\subsubsection{Example: an algebra coming from a Postnikov diagram}\label{sss:cuts}

Given a quiver with potential, i.e., a formal sum of cycles up to cyclic permutation, one obtains an algebra by formally differentiating the potential with respect to the arrows: this is known as the Jacobi algebra.  Herschend and Iyama studied self-injective Jacobi algebras and showed that they always arise as preprojective algebras of 2-representation finite algebras; moreover, one can recover the underlying 2-representation finite algebra by taking a \emph{cut} of the quiver \cite{hi-qp}.

Pasquali studied Jacobi algebras associated to Postnikov diagrams  \cite{pas}.  These algebras are obtained by taking stable endomorphism algebras of tilting modules for some other algebras, which were constructed by Jensen, King, and Su to categorify the cluster algebra structure on the homogeneous coordinate ring of the Grassmannian of $k$-planes in $n$-dimensional space \cite{jks}.  Pasquali showed that the Jacobi algebra is self-injective precisely when the Postnikov diagram is \emph{symmetric}, i.e., invariant under a rotation by $2k\pi/n$.  In this case, the rotation of the Postnikov diagram induces the Nakayama automorphism of the corresponding Jacobi algebra \cite[Theorem 8.2 and Corollary 8.3]{pas}.  In particular, it has finite order.  Pasquali notes the connection to fractionally Calabi-Yau algebras, but the theory of \cite{hi-frac} requires the quiver to have a cut which is fixed by the Nakayama automorphism in order for the 2-representation finite algebra to satisfy be \emph{homogeneous}.  This rarely happens in examples.

Using Pasquali's description of the Nakayama automorphism together with Corollary \ref{cor:fcyfinitenak}, we obtain:
\begin{theorem}
Every 2-representation finite algebra obtained as a cut of a self-injective quiver with potential coming from a Postnikov diagram is fractional Calabi-Yau.
\end{theorem}

Computing the Calabi-Yau dimension explicitly takes more work.  We show how to do this in an explicit example.
Consider the following quiver $Q$:
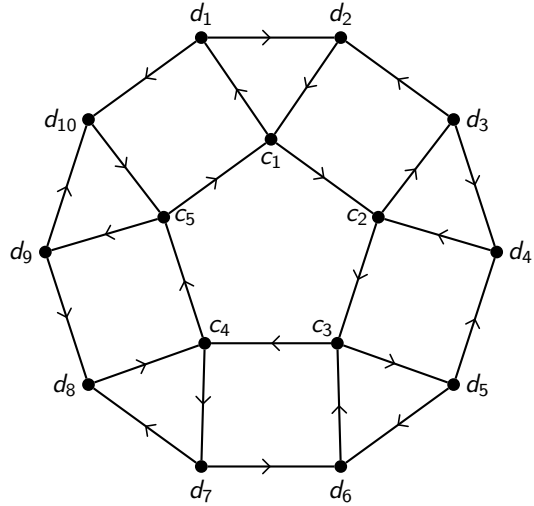
\[
\tikzstyle{every node}=[circle, draw, fill=black!50, inner sep=0pt, minimum width=4pt]
\begin{tikzpicture}[thick,scale=0.5]
\tikzstyle{vertex}=[circle, draw, fill=black, inner sep=0pt, minimum width=4pt]
\tikzstyle{arc}=[draw,  decoration={markings,mark=at position 0.5 with {\arrow[semithick]{angle 90}, node[below left]{1}}},  postaction={decorate}]
\tikzstyle{rel}=[draw, dashed  decoration={markings,mark=at position 0.5 with {\arrow[semithick]{angle 90}, node[below left]{1}}},  postaction={decorate}]
\draw (90:3) node[vertex] (c1) [label=below:$c_1$] {}
      (18:3) node[vertex] (c2) [label=left:$c_2$] {}
      (306:3) node[vertex] (c3) [label=above left:$c_3$] {}
      (234:3) node[vertex] (c4) [label=above right:$c_4$] {}
      (162:3) node[vertex] (c5) [label=right:$c_5$] {};
\draw (108:6) node[vertex] (d1) [label=above:$d_1$] {}
      (72:6) node[vertex] (d2) [label=above:$d_2$] {}
      (36:6) node[vertex] (d3) [label=right:$d_3$] {}
      (0:6) node[vertex] (d4) [label=right:$d_4$] {}
      (324:6) node[vertex] (d5) [label=right:$d_5$] {}
      (288:6) node[vertex] (d6) [label=below:$d_6$] {}
      (252:6) node[vertex] (d7) [label=below:$d_7$] {}
      (216:6) node[vertex] (d8) [label=left:$d_8$] {}
      (180:6) node[vertex] (d9) [label=left:$d_9$] {}
      (144:6) node[vertex] (d10) [label=left:$d_{10}$] {};
\draw (c1) edge[arc]  (c2) (c2) edge[arc]  (c3) (c3) edge[arc]  (c4)  (c4) edge[arc]  (c5) (c5) edge[arc]  (c1);
\draw (d1) edge[arc]  (d2) (d2) edge[arc]  (c1) (c1) edge[arc]  (d1) ;
\draw (d3) edge[arc]  (d4) (d4) edge[arc]  (c2) (c2) edge[arc]  (d3) ;
\draw (d5) edge[arc]  (d6) (d6) edge[arc]  (c3) (c3) edge[arc]  (d5) ;
\draw (d7) edge[arc]  (d8) (d8) edge[arc]  (c4) (c4) edge[arc]  (d7) ;
\draw (d9) edge[arc]  (d10) (d10) edge[arc]  (c5) (c5) edge[arc]  (d9) ;
\draw (d1) edge[arc]  (d10) ;
\draw (d9) edge[arc]  (d8) ;
\draw (d7) edge[arc]  (d6) ;
\draw (d5) edge[arc]  (d4) ;
\draw (d3) edge[arc]  (d2) ;
\end{tikzpicture}
\]
It is a \emph{cobweb} quiver and comes from a $(4,10)$-Postnikov diagram: see \cite[Proposition 10.2]{pas}.  It has an obvious potential $W$ given by the sum of the clockwise cycles minus the sum of the anticlockwise cycles.  The Jacobi algebra $A=J(Q,W)$ has Nakayama automorphism induced by the rotation of $\frac 2 5 2\pi$, i.e., by the unique graph automorphism which acts on vertices by
\begin{align*}
 c_1 &\mapsto c_3 \mapsto c_5 \mapsto c_2 \mapsto c_4 \mapsto c_1; \\
d_1& \mapsto d_5 \mapsto d_9 \mapsto d_3 \mapsto d_7 \mapsto d_1; \\
d_2 &\mapsto d_6 \mapsto d_{10} \mapsto d_4 \mapsto d_8 \mapsto d_2. 
\end{align*}
Note that $\sigma^5=\id_\Pi$.

Any cut must contain exactly one arrow from each cycle, so by considering the central pentagon we see this quiver has no cut which is invariant under the Nakayama permutation.  Consider the following cut, consisting of all but one outer arrow and one arrow from the central pentagon:
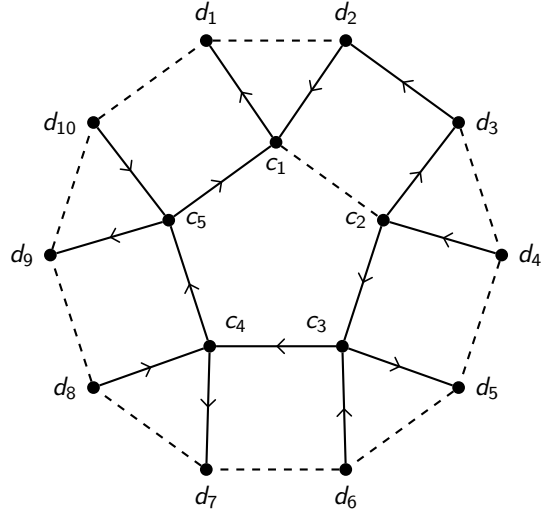
\[\begin{tikzpicture}[thick,scale=0.5]
\tikzstyle{vertex}=[circle, draw, fill=black, inner sep=0pt, minimum width=4pt]
\tikzstyle{arc}=[draw,  decoration={markings,mark=at position 0.5 with {\arrow[semithick]{angle 90}, node[below left]{1}}},  postaction={decorate}]
\draw (90:3) node[vertex] (c1) [label=below:$c_1$] {}
      (18:3) node[vertex] (c2) [label=left:$c_2$] {}
      (306:3) node[vertex] (c3) [label=above left:$c_3$] {}
      (234:3) node[vertex] (c4) [label=above right:$c_4$] {}
      (162:3) node[vertex] (c5) [label=right:$c_5$] {};
\draw (108:6) node[vertex] (d1) [label=above:$d_1$] {}
      (72:6) node[vertex] (d2) [label=above:$d_2$] {}
      (36:6) node[vertex] (d3) [label=right:$d_3$] {}
      (0:6) node[vertex] (d4) [label=right:$d_4$] {}
      (324:6) node[vertex] (d5) [label=right:$d_5$] {}
      (288:6) node[vertex] (d6) [label=below:$d_6$] {}
      (252:6) node[vertex] (d7) [label=below:$d_7$] {}
      (216:6) node[vertex] (d8) [label=left:$d_8$] {}
      (180:6) node[vertex] (d9) [label=left:$d_9$] {}
      (144:6) node[vertex] (d10) [label=left:$d_{10}$] {};
\draw[dashed] (c1) to (c2);
\draw (c2) edge[arc]  (c3) (c3) edge[arc]  (c4)  (c4) edge[arc]  (c5) (c5) edge[arc]  (c1);
\draw  (d2) edge[arc]  (c1) (c1) edge[arc]  (d1) ;
\draw  (d4) edge[arc]  (c2) (c2) edge[arc]  (d3) ;
\draw (d6) edge[arc]  (c3) (c3) edge[arc]  (d5) ;
\draw  (d8) edge[arc]  (c4) (c4) edge[arc]  (d7) ;
\draw (d10) edge[arc]  (c5) (c5) edge[arc]  (d9) ;
\draw[dashed] (d3) to (d4) (d4) to (d5) (d5) to (d6) (d6) to (d7) (d7) to (d8) (d8) to (d9) (d9) to (d10) (d10) to (d1) (d1) to (d2) ;
\draw (d3) edge[arc]  (d2) ;
\end{tikzpicture}
\]
This defines a 2-representation finite algebra $\Lambda$, with quiver given by the arrows not in the cut.  It also induces a grading on the preprojective algebra $\Pi=A$ of $\Lambda$, with arrows in the cut having degree $1$.  

By direct calculation we see that the projective module associated to the vertex $d_1$ has socle generated by the following path:
\[ d_1 \to d_2 \to c_1 \to c_2 \to c_3 \to d_5. \]
This path has degree $2$, and so the graded Nakayama automorphism $(\sigma,\ell)$ of $\Pi$ has $\ell(d_1)=2$.  Rotating, we see that:
\[ \ell(d_1)=2; \;\;\; \ell(d_5)=1; \;\;\; \ell(d_9)=2; \;\;\; \ell(d_3)=1; \;\;\; \ell(d_7)=1 \]
so $\sigma^5=\id_\Pi$ and $\sum_{i=0}^4\ell(\sigma^i(d_1))=7$.  Thus, by Lemma \ref{lem:constda}, we have $(\sigma,\ell)^5=(\id,\underline{7})$.

As $d=2$ we have $\sgn^d=\tr$.  So by Theorem \ref{thm:drf-fcy}, $\Lambda$ is fractionally Calabi-Yau of dimension $\frac{14}{12}$.

%\textbf{Data availability statement:} 
\subsection*{Data availability statement}
Data sharing not applicable to this article as no datasets were generated or analysed during the current study.

School of Mathematics, University of East Anglia, Norwich NR4 7TJ, UK
\\\url{j.grant@uea.ac.uk}

\end{document}